\documentclass[12pt]{amsart}

\usepackage{amsmath}

\usepackage{amssymb}

\usepackage{mathtools}

\usepackage{array}

\usepackage{a4wide}
  
\usepackage{stmaryrd}

\usepackage{tikz}

\usepackage{tikz-cd}
\usetikzlibrary{cd}

\usepackage{bm}

\usepackage{mathrsfs}

\usepackage{amsthm}

\usepackage{enumerate}

\usepackage{fancyhdr}

\usepackage{hyperref}

\usepackage[inline]{enumitem}

\usepackage{marginnote}

\theoremstyle{plain}

\newtheorem{thm}{Theorem}[section]

\newtheorem{lem}[thm]{Lemma}

\newtheorem{cor}[thm]{Corollary}

\newtheorem{prop}[thm]{Proposition}

\theoremstyle{definition}

\newtheorem{rem}[thm]{Remark}

\newtheorem{conv}[thm]{Convention}

\newtheorem{hyp}[thm]{Hypotheses}

\newtheorem*{note}{Note}

\newtheorem{defn}[thm]{Definition}

\newtheorem{eg}[thm]{Example}

\newtheorem{egs}[thm]{Examples}

\newcommand{\cO}{\mathcal{O}}

\newcommand{\bC}{\mathbb{C}}

\newcommand{\bP}{\mathbb{P}}
\newcommand{\bQ}{\mathbb{Q}}
\newcommand{\bR}{\mathbb{R}}

\newcommand{\bZ}{\mathbb{Z}}

\newcommand{\fm}{\mathfrak{m}}

\newcommand{\fx}{\mathfrak{x}}
\newcommand{\fy}{\mathfrak{y}}

\newcommand{\scA}{\mathscr{A}}

\newcommand{\scD}{\mathscr{D}}

\newcommand{\scH}{\mathscr{H}}

\newcommand{\scJ}{\mathscr{J}}

\newcommand{\scL}{\mathscr{L}}
\newcommand{\scM}{\mathscr{M}}

\newcommand{\scO}{\mathscr{O}}
\newcommand{\scP}{\mathscr{P}}

\DeclareMathOperator{\Torder}{\emph{t}-ord}
\DeclareMathOperator{\ann}{ann}
\DeclareMathOperator{\specialize}{sp_{s-1 \mapsto -1}}
\DeclareMathOperator{\spe}{sp}
\DeclareMathOperator{\ord}{ord}
\DeclareMathOperator{\ind}{ind}
\DeclareMathOperator{\gr}{gr}
\DeclareMathOperator{\Der}{Der}
\DeclareMathOperator{\rank}{rank}
\DeclareMathOperator{\genLevel}{GenLevel}

\usepackage{appendix}

\numberwithin{equation}{section}

\newcommand\dan[1]{\ensuremath{\spadesuit}{DB:\tiny{#1}}}

\begin{document}

\title[The Hodge filtration and parametrically prime divisors]{The Hodge filtration and parametrically prime divisors}

\author{Daniel Bath and Henry Dakin}

\date{}
\thanks{DB was supported by FWO Grant \#12E9623N}
\maketitle

\begin{abstract}
We study the canonical Hodge filtration on the sheaf $\scO_X(*D)$ of meromorphic functions along a divisor. For a germ of an analytic function $f$ whose Bernstein-Sato's polynomial's roots are contained in $(-2,0)$, we: give a simple algebraic formula for the zeroeth piece of the Hodge filtration; bound the first step of the Hodge filtration containing $f^{-1}$. If we additionally require $f$ to be Euler homogeneous and parametrically prime, then we extend our algebraic formula to compute every piece of the canonical Hodge filtration, proving in turn that the Hodge filtration is contained in the induced order filtration. Finally, we compute the Hodge filtration in many examples and identify several large classes of divisors realizing our theorems.
\end{abstract}

\tableofcontents

\section{Introduction}\label{intro}

Let $X$ be a complex manifold of dimension $n$ and $D$ a reduced effective divisor on $X$. We are interested in the canonical Hodge filtration $F_{\bullet}^H \scO_X(*D)$ associated to the (regular holonomic) left $\scD_X$-module $\scO_X(*D)$ of meromorphic functions on $X$ which are holomorphic on $U:=X\backslash D$. This canonical Hodge filtration is obtained from the natural mixed Hodge module structure overlying the $\scD_X$-module $\scO_X(*D)$ which, following Saito's \cite{MSai90} general theory, is given by $j_*\mathbf{\bQ}_U^H[n]$. Here $j:U\to X$ is the open embedding of $U$ into $X$ and $\bQ_U^H$ is the trivial Hodge module structure on $U$. 

This Hodge filtration, and the weight filtration on $\scO_X(*D)$ also coming from the mixed Hodge module structure, give a way to understand and calculate the mixed Hodge structures associated to the divisor complement $U$.\footnote{They in fact determine the mixed Hodge structure on the cohomology groups of $U$ as defined by Deligne, at least for $X$ compact and Kähler (see \cite{Del71} and Section 4.6 \cite{MSai93})}. In this paper we give explicit formulas for $F_{\bullet}^H \scO_{X}(*D)$ for a large class of divisors. Future work will be concerned with the weight filtration, completing the picture.

\vspace{5mm}

Results on $F_\bullet^H \scO_X(*D)$ broadly arise from one of two philosophies: either proceed globally and birationally or proceed locally and algebraically. The birational approach takes a log resolution $(Y, E) \to (X, D)$ and pushes forward the Hodge filtration on $\scO_Y(*E)$, see \cite{MP19a}, Section D. Mustaţă and Popa have obtained general results this way in a series of recent papers \cite{MP19a}, \cite{MP19b}, \cite{MP19c}. But to actually compute $F_\bullet^H \scO_X(*D)$, the birational approach requires having an explicit log resolution (which one generally does not) and extracting explicit derived direct image data of that log resolution (which is difficult even when the resolution is known). We do not proceed birationally.

Instead, we follow the algebraic approach to understand $F_\bullet^H \scO_X(*D)$. This requires working locally, so henceforth, unless stated otherwise, our divisor $D$ is defined by a global reduced equation $f \in \scO_X$. The algebraic methodology exploits the graph embedding of $f$:
\begin{equation*}
    i_f : X \to X \times \bC \quad \text{given by} \quad \fx \mapsto (\fx, f(\fx)), 
\end{equation*}
where $t$ denotes the coordinate of the $\bC$-factor. On the $\scD_X$-module direct image $i_{f,+} \scO_X(*D)$, the \emph{Kashiwara-Malgrange $V$-filtration} $V^\bullet i_{f,+}\scO_X(*f)$ on $i_{f,+} \scO_X(*D)$ gives, essentially by definition (Section 2.8 \cite{MSai90}), an expression for the Hodge filtration on $\scO_X(*D)$. This method for studying $F_\bullet^H \scO_X(*D)$ has been reaped to great effect for: isolated singularities \cite{Saito09,Zhang2021,Saito2022}; a general algorithm \cite{Blanco22}; strongly Koszul free divisors \cite{CNS22}.

From the graph embedding, we have a natural isomorphism of $\scD_{X \times \bC_t}$-modules
\begin{equation*}
    i_{f,+} \scO_X(*D) \simeq \scO_X(*D)[\partial_t]\delta
\end{equation*}
as recalled in Section 2. This leads to the \emph{$t$-order filtration} $F_\bullet^{\Torder} i_{f,+} \scO_X(*D)$, Definition \ref{defntord}, induced by pulling back elements of appropriate $\partial_t$-order along said isomorphism. Our first main result Theorem \ref{thmformula}, which is essentially a reformulation of, and in fact utilizes, Theorem A' \cite{MP20}, is an expression for $F_\bullet^H \scO_X(*D)$ in terms of this $t$-order filtration:

\begin{thm} \label{thmformula-intro}

Let $f \in \scO_X$. For each non-negative integer $k$,
\begin{center} \vspace{-0.4cm} \begin{align*}F_k^H\scO_X(*f) & = \left(\partial_tF_k^Hi_{f,+}\scO_X(*f)+V^0i_{f,+}\scO_X(*f)\right)\cap\left(\scO_X(*f)\delta\right) \\ &= \left(\partial_tF_{k-1}^{t-\text{\emph{ord}}}i_{f,+}\scO_X(*f)+V^0i_{f,+}\scO_X(*f)\right)\cap\left(\scO_X(*f)\delta\right),\end{align*}\end{center} where the two intersections take place in $\scO_X(*f)[\partial_t]\delta\simeq i_{f,+}\scO_X(*f)$.

\end{thm}

\noindent Consequently, to understand $F_\bullet^H \scO_X(*D)$ one must understand the interactions of the $t$-order filtration and $V$-filtration. In this, and throughout the rest of the paper, we are inspired by the ideas of Castaño Domínguez, Narváez Macarro, and Sevenheck from \cite{CNS22}.

\vspace{5mm}

On this quest we consider sort of competing object to $i_{f,+}\scO_{X,\hspace{1pt}\fx}(*f)$, where we have picked the $\fx \in D$ we are interested in. (The notation $i_{f,+}\scO_{X,\hspace{1pt}\fx}(*f)$ denotes $(i_{f,+}\scO_U(*f))_{(\fx,0)}$, for $U \ni \fx$ suitably chosen, cf. Convention \ref{convlocal1}.) Now that we work at stalks, we use $\scO_{X,\hspace{1pt}\fx}(*f)$ instead of divisor notation. Construct the $\scD_{X,\hspace{1pt}\fx}[s]=\scD_{X,\hspace{1pt}\fx} \otimes_\bC \bC[s]$-module $\scD_{X,\hspace{1pt}\fx}[s]\cdot f^s \subseteq (\scO_{X,\hspace{1pt}\fx}(*f) \otimes_\bC \bC[s] )f^s$, where $s$ is a new indeterminate (see subsection 3.2) and the $\scD_{X,\hspace{1pt}\fx}$-action is given by a formal chain rule. One can similarly define $\scD_{X,\hspace{1pt}\fx}[s]\cdot f^{s+\ell}$ for $\ell \in \mathbb{Z}$, which is the cyclic submodule generated by $f^{s+\ell}$ and naturally isomorphic to $\scD_{X,\hspace{1pt}\fx}[s] / \ann_{\scD_{X,\hspace{1pt}\fx}[s]} f^{s+\ell}$. These objects lead to the \emph{Bernstein-Sato polynomial} $b_f(s)$ of $f \in \scO_{X,\hspace{1pt}\fx}$\footnote{For this to be well-defined analytically, one must work at stalks.}, which is the minimal monic polynomial in $\bC[s]$ satisfying the \emph{functional equation}
\begin{equation*}
    b_f(s) f^s \in \scD_{X,\hspace{1pt}\fx}[s]\cdot f^{s+1}.
\end{equation*}
The roots of the Bernstein-Sato polynomial (or any $\bC[s]$-polynomial) are denoted by $Z(b_f(s))$. 

The divisors considered in \cite{CNS22} all have (locally everywhere) Bernstein-Sato polynomials' whose roots lie in $(-2,0)$, cf. \cite{Nar15}. If one defines new polynomials $\beta_f(s), \beta_f^\prime(s) \in \bC[s]$ by
\begin{equation*}
    \beta_f(s) \beta_f^\prime(s) = b_f(-s-1) \quad \text{where } Z(\beta_f(s)) \subseteq (0,1) \text{ and } \beta_f(s), \beta_f^\prime(s) \text{ are coprime},
\end{equation*}
then, when the roots of the Bernstein-Sato polynomial are all contained in $(-2,0)$, Proposition 3.3 \cite{CNS22} gives a formula for integral steps of the Kashiwara-Malgrange $V$-filtration on $i_{f,+} \scO_{X,\hspace{1pt}\fx}(*f)$ in terms of: $\beta_f(s)$, the generator choice $\scD_{X,\hspace{1pt}\fx}\cdot f^{-1} = \scO_{X,\hspace{1pt}\fx}(*f)$, and the $V$-filtration on $\scD_{X \times \bC_t, (\fx, t)}$ along $\{t=0\}$. 

Using this we give an algebraic formula for $F_0^H \scO_{X,\hspace{1pt}\fx}(*f)$, for any germ $f \in \scO_{X,\hspace{1pt}\fx}$ such that $Z(b_f(s)) \subseteq (-2,0)$. This can be stated in two ways. For one, let $\Gamma_f$ be the $\scD_{X,\hspace{1pt}\fx}[s]$-ideal
\begin{equation*}
    \Gamma_f = \scD_{X,\hspace{1pt}\fx}[s] \cdot f + \scD_{X,\hspace{1pt}\fx}[s] \cdot \beta_f(-s) + \scD_{X,\hspace{1pt}\fx}[s] \cdot \ann_{\scD_{X,\hspace{1pt}\fx}[s]} f^s.
\end{equation*}
For the other, consider all $P(s) \in \scD_{X,\hspace{1pt}\fx}[s]$ satisfying the ``functional equation''
\begin{equation*}
    P(s) \beta_f^\prime(-s) \cdot f^{s-1} \in \scD_{X,\hspace{1pt}\fx}[s]\cdot f^s.
\end{equation*}

Our second main result, Corollary \ref{cor-zeroHodgePiece}, computes $F_0^H \scO_{X,\fx}(*f)$, the localization of the zeroeth piece of the Hodge filtration $F_0^H \scO_{X}(*f)$ at $\fx$, cf. Remark \ref{remlocal2}, under the assumption $Z(b_f(s)) \subseteq (-2,0)$:

\begin{thm}     \label{thm-cor-zeroHodgePiece-intro}
Assume that $f \in \scO_{X,\hspace{1pt}\fx}$ satisfies $Z(b_f(s)) \subseteq (-2,0)$. Then
\begin{align*}
    F_0^H\scO_{X,\hspace{1pt}\fx}(*f) 
    =(\Gamma_f\cap\scO_{X,\hspace{1pt}\fx})f^{-1}.
\end{align*}
That is, $F_0^H \scO_{X,\hspace{1pt}\fx}(*f)$ is the $\scO_{X,\hspace{1pt}\fx}$-submodule generated by $\{g f^{-1}\}$ where $g$ ranges over all elements of $\scO_{X,\hspace{1pt}\fx}$ satisfying the ``functional equation''
\begin{equation*}
    g \beta_f^\prime(-s) f^{s-1} \in \scD_{X,\hspace{1pt}\fx}[s]\cdot f^{s}.
\end{equation*}
\end{thm}

\noindent This is a significant generalization of Corollary 5.3 \cite{CNS22} which requires the ``strongly Koszul free divisor'' assumption (both in statement and proof methodology) and only proves the statement preceding ``equivalently''. Moreover, argumentation in \cite{CNS22} differs in this result and throughout due to our exploitation of the tight relationship between $(\scO_X(*f) \otimes_\bC \bC[s])f^s$ and $i_{f,+} \scO_X(*f)$.

A fundamental result of Saito \cite{MSai93} is that the Hodge filtration is contained inside the pole order filtration. The containment is often strict and we are interested in the Hodge filtration's relationship to a finer filtration we call the \emph{induced order filtration} $F_\bullet^{\ord} \scO_{X,\hspace{1pt}\fx}(*f)$, cf. Definition \ref{def-orderFilLocalization}. This is given by applying the elements of $\scD_{X,\fx}$ of prescribed order to a fixed generator $f^{-1}$ of $\scO_{X,\hspace{1pt}\fx}(*f)$.

In general, the $k^{\text{th}}$ piece of the Hodge filtration on $\scO_{X,\hspace{1pt}\fx}(*f)$ will not contain the $k^{\text{th}}$ piece of the induced order filtration, nor vice versa. But, assuming only that the roots of the Bernstein-Sato polynomial are contained in $(-2,0)$, we find an explicit bound for when the $k^{\text{th}}$ piece of the Hodge filtration contains a lower piece of the induced order filtration. As $f^{-1} \in F_0^{\ord} \scO_{X,\hspace{1pt}\fx}(*f)$, this is related to finding the first piece of $F_\bullet^H \scO_{X,\hspace{1pt}\fx}(*f)$ containing $f^{-1}$. Corollary \ref{corrf} says:

\begin{thm} \label{thm-rf-intro}

Assume that $f \in \scO_{X,\hspace{1pt}\fx}$ satisfies $Z(b_f(s)) \subseteq (-2,0)$. Write $r_f$ for the degree of the polynomial $\beta_f(s)\in\bC[s]$, cf. Definition \ref{def-funnyBSpoly}. Then, for all $k \in \mathbb{Z}$, 
\[F_{k-r_f}^{\text{\emph{ord}}}\scO_{X,\hspace{1pt}\fx}(*f)\subseteq F_k^H\scO_{X,\hspace{1pt}\fx}(*f).\]

\end{thm}
Again this generalizes \cite{CNS22}, now specifically Corollary 4.8 of loc. cit., by relaxing the ``strongly Koszul free divisor'' hypothesis therein to the weaker requirement that $Z(b_f(s)) \subseteq (-2,0)$.

\vspace{5mm}

Our final two theorems are our main results: they compute higher pieces of the Hodge filtration on $\scO_{X,\hspace{1pt}\fx}(*f)$ and clarify its relationship to the induced order filtration. We must impose additional hypotheses (see Hypotheses \ref{hyp-mainHypotheses}) on $f$ (or $D$). Namely:
\begin{hyp} \label{hyp-mainHypotheses-intro}
We will often require the following hypotheses on $f \in \scO_X$:
\begin{enumerate}[label=\alph*)]
    \item $f$ is \emph{Euler homogeneous} (on $X$);
    \item $f$ is \emph{parametrically prime}, i.e. $\gr^{\sharp}(\ann_{\scD_X[s]} f^{s-1})$ is prime locally everywhere;
    \item the zeroes of the Bernstein-Sato polynomial live in $(-2,0)$ locally everywhere.
\end{enumerate}
\end{hyp}
Item a) means $f$ is in the ideal generated by its partial derivatives (Definition \ref{def-EulerHom}). The paper's title comes from item b)'s \emph{parametrically prime} condition (Definition \ref{def-parametricallyPrime}), which demands the primality of the initial ideal of $\ann_{\scD_{X}[s]} f^{s-1} \subseteq \scD_X[s]$ with respect to the \emph{total order filtration} on $\scD_X[s]$ (Definition \ref{def-totalOrderFiltration}). Item b) is motivated by Question 5.9 of \cite{Wal17}, where Walther essentially asks if this holds without any assumptions on $f$. In Example \ref{ex:uliQuestion} we answer Walther in the negative.

The difficulty in extending the techniques in Theorem \ref{thmformula-intro} from the zeroeth piece of the Hodge filtration to higher pieces of the Hodge filtration involves a sort of compatability style result (Proposition \ref{propcomp-new}) between three filtrations on $i_{f,+} \scO_{X,\hspace{1pt}\fx}(*f)$: the Kashiwara-Malgrange $V$-filtration; the \emph{induced $V$-filtration} (Definition \ref{def-inducedVfiltration}), given by fixing a generator of $i_{f,+} \scO_{X,\hspace{1pt}\fx}(*f)$ and looking at the action of elements $\scD_{X \times \bC_t, (\fx, 0)}$ of fixed order with respect to the $V$-filtration along $\{t=0\}$; the \emph{induced order filtration}\footnote{We have induced order filtrations on both $\scO_{X,\hspace{1pt}\fx}(*f)$ and $i_{f,+}\scO_{X,\hspace{1pt}\fx}(*f)$, arising in similar ways.} on $i_{f,+} \scO_{X,\hspace{1pt}\fx}(*f)$, which is defined similarly except now fixing the order of elements of $\scD_{X \times \bC_t, (\fx, 0)}$ with respect to canonical order filtration. Our assumptions enable a subtle sort of iterated Gr{\"o}bner style argument (Proposition \ref{propcomp-new}), further reducing understanding the Kashiwara-Malgrange $V$-filtration to understanding all our various induced filtrations.

Theorem \ref{thmmain-new} is our first main result for these divisors, refining the coarse containment of Hodge filtration inside the pole order filtration due to Saito \cite{MSai93} and completing the story started by Theorem \ref{thm-rf-intro}. There is also a non-globally defined version of this result.

\begin{thm} \label{thmmain-new-intro}
Assume that $f \in \scO_X$ satisfies Hypotheses \ref{hyp-mainHypotheses-intro}: $f$ is Euler homogeneous; $f$ is parametrically prime; the roots of the Bernstein-Sato polynomial of $f$ are contained in $(-2,0)$ locally everywhere. Then, for every non-negative integer $k$, 
\begin{equation*}
    F_k^H\scO_X(*f)\subseteq F_k^{\ord} \scO_X(*f).
\end{equation*}

Moreover, consider a reduced effective divisor $D$ (not necessarily globally defined) that is: strongly Euler homogeneous, $\gr^{\ord} (\ann_{\scD_X} f^{s-1}) \subseteq \gr^{\ord} \scD_X$ is prime (locally everywhere, for some choice of defining equation), and the roots of the Bernstein-Sato polynomial of $D$ are contained in $(-2,0)$ locally everywhere. Then for every non-negative integer $k$, 
\begin{equation*}
    F_k^H\scO_X(*D)\subseteq F_k^{\text{\emph{ord}}}\scO_X(*D).
\end{equation*}
\end{thm}

The second main result Theorem \ref{thmgamma2-new} gives an explicit local equation for the Hodge filtration exactly in the spirit of Theorem \ref{thm-cor-zeroHodgePiece-intro}:

\begin{thm} \label{thmgamma2-new-intro}
Assume that $f \in \scO_{X,\hspace{1pt}\fx}$ satisfies Hypotheses \ref{hyp-mainHypotheses-intro} at $\fx$: $f$ is Euler homogeneous; $f$ is parametrically prime at $\fx$; the zeroes of the Bernstein-Sato polynomial are contained in $(-2,0)$ . Let $\phi_0 : \scD_{X,\hspace{1pt}\fx}[s] \to \scD_{X,\hspace{1pt}\fx}$ be the $\scD_{X,\hspace{1pt}\fx}$-map induced by $s \mapsto 0$. Then for all $k \in \mathbb{Z}_{\geq 0}$,
\begin{align*}
    F_k^H\scO_{X,\hspace{1pt}\fx}(*f) 
    = \phi_0 \bigg( \Gamma_f\cap F_k^{\sharp} \scD_{X,\hspace{1pt}\fx}[s] \bigg) \cdot f^{-1}
\end{align*}

In other words, let $\specialize: \scD_{X,\hspace{1pt}\fx}[s]\cdot f^{s-1} \to \scD_{X,\hspace{1pt}\fx}\cdot f^{-1}$ be the specialization map, cf. Definition \ref{def-specialization}. Then $F_k^H \scO_{X,\hspace{1pt}\fx}(*f)$ is the $\scO_{X,\hspace{1pt}\fx}$-module generated by $\{\specialize(P(s)\cdot f^{s-1})\}$ where $P(s)$ ranges over all elements of $F_k^\sharp \scD_{X,\hspace{1pt}\fx}[s]$ satisfying the ``functional equation''
\begin{equation*}
    P(s) \beta_f^\prime(-s) \cdot f^{s-1} \in \scD_{X,\hspace{1pt}\fx}[s]\cdot f^{s}.
\end{equation*}
\end{thm}

\noindent Theorem \ref{thmmain-new-intro} and Theorem \ref{thmgamma2-new-intro} should be contrasted with Theorem 4.4 and Theorem 5.2 \cite{CNS22}, where: both results in loc. cit. require their divisors to be strongly Koszul and free; the latter result is recursive, unlike our ``closed formula'' Theorem \ref{thmgamma2-new-intro}. And again, our argument systematically differs due to our usage of $\scD_{X,\hspace{1pt}\fx}[s]\cdot f^{s-1}$ as well as our more robust tackling of the compatability issue in Proposition \ref{propcomp-new}.

\vspace{5mm}

We conclude, in Section 6, by computing the Hodge filtration for several examples fitting Hypotheses \ref{hyp-mainHypotheses-intro}. This is done by: figuring out the Bernstein-Sato polynomial using Macaulay2 or other means; extracting $\beta_f(s)$; asking Macaulay2 for a Gr{\"o}bner basis of $\Gamma_f \subseteq \scD_X[s]$ with respect to the total order filtration. Note that our assumptions, as evidenced by some of our examples, do not imply that $\ann_{\scD_X[s]} f^{s-1}$ is generated by total order one differential operators, unlike the case of strongly Koszul free divisors considered in \cite{CNS22}.

In Section $5$, we note that divisors of \emph{linear Jacobian type} (Definition \ref{def-linearJacType}) satisfy items a) and b) of Hypotheses \ref{hyp-mainHypotheses-intro}, see Corollary \ref{cor-LJTsatisfiesHyp}. In \cite{Wal17}, Walther identified a large class of divisors of this type, transcribed as Theorem \ref{thmOfUli}. In this case $\ann_{\scD_X[s]} f^{s-1}$ is generated by total order one differential operators, expediting many computational difficulties in Macaulay2. In the case of positively weighted homogeneous locally everywhere divisors in $\mathbb{C}^3$, we note (Theorem \ref{thm-PWHomLENiceClass}) that these are of linear Jacobian type and remind the reader of Corollary 7.10 \cite{Bath24} (generalizing Proposition 1 \cite{BS23}) showing that checking item (a) of Hypotheses \ref{hyp-mainHypotheses-intro}, i.e. that $Z(b_{f,0}(s)) \subseteq (-2,0)$, equates to computing local cohomology of the Milnor algebra.

\vspace{5mm}

The structure of the paper is as follows. In section $2$, we prove Theorem \ref{thmformula-intro}/Theorem \ref{thmformula}, relating the Hodge filtration to the Kashiwara-Malgrange filtration. In section $3$ we pick (having assumed it is possible) the generator $f^{-1}$ of $\scO_{X,\hspace{1pt}\fx}(*f)$, introduce induced filtrations, and prove many useful identities. Then we prove our first two theorems, Theorem \ref{thm-cor-zeroHodgePiece-intro}/Corollary \ref{cor-zeroHodgePiece} and Theorem \ref{thm-rf-intro}/Corollary \ref{corrf} about $f \in \scO_{X,\hspace{1pt}\fx}$ only required to satisfy $Z(b_f(s)) \subseteq (-2,0)$. In Section $4$ we introduce Hypotheses \ref{hyp-mainHypotheses-intro} and prove our main results under these assumptions: Theorem \ref{thmmain-new-intro}/Theorem \ref{thmmain-new} and Theorem \ref{thmgamma2-new-intro}/Theorem \ref{thmgamma2-new}. In Section $5$ we discuss the linear Jacobian type condition; in Section $6$ we compute several examples.

The first-named author thanks Guillem Blanco and Christian Sevenheck for several helpful conversations, and the latter for, additionally, introducing the authors. The second-named author would like to thank his supervisor Christian Sevenheck for introducing him to the subject and for numerous helpful discussions concerning the material of this paper.

\section{A general formula for the Hodge filtration} \label{formula}

\noindent Throughout the paper, let $X$ be a complex manifold of dimension $n$ and $f \in \scO_X$ \hspace{-3pt}\footnote{Throughout, this shorthand notation will be used to denote a global section of the sheaf in question.} a reduced non-invertible non-zero holomorphic function on $X$. Write \[i_f:X\to X\times \bC_t\; ;\; \fx \mapsto (\fx,f(\fx))\] for the graph embedding of the function $f$, with $t$ denoting the coordinate on $\bC$. Mustaţă and Popa (see Theorem A' \cite{MP20}) proved a (local) formula for the Hodge filtration on $\scO_X(*f)$ in terms of the Kashiwara-Malgrange $V$-filtration. In this section, we recall some basic facts about the Kashiwara-Malgrange $V$-filtration and then rewrite this result in a manner that resembles more closely the formula obtained in \cite{CNS22} (Theorem 4.4). 


Let $\scM$ be a left $\scD_X$-module. The left $\scD_{X\times \bC_t}$-module $i_{f,+}\scM$ is isomorphic to $\scM[\partial_t]\delta$ (with $\delta$ being the generator of the localization), whose $\scD_{X\times \bC_t}$-action is given on an open chart $U\subseteq X$ with coordinates $x_1,\ldots,x_n$ by 
\begin{itemize}
    \item $g(x,t) \cdot(m \delta) = (g(x,f)\cdot m) \delta$;
    \item $\partial_{x_i}\cdot (m \partial_t^k\delta) = (\partial_{x_i} \cdot m) \partial_t^k\delta - (\partial_{x_i}(f)\cdot m)\partial_t^{k+1}\delta$;
    \item $t\cdot (m\partial_t^k\delta)=(f\cdot m)\partial_t^k\delta-km\partial_t^{k-1}\delta$;
    \item $\partial_t\cdot (m\partial_t^k\delta)=m\partial_t^{k+1}\delta$.
\end{itemize}

\vspace{4pt}\noindent where $g(x,t) \in \scO_{U\times\bC_t}$ and $m \in \scM|_U$. It is easy to check that this action is well-defined up to coordinate change and that this does indeed uniquely determine an action of $\scD_{U\times\bC_t}$ on $\scM|_U[\partial_t]\delta$ \footnote{In particular the action of $g(x,t)$ on $m\partial_t^k\delta$ is equal to the action of $\partial_t^kg(x,t)+[g(x,t),\partial_t^k]$ on $m\delta$, and is thus determined via recursion on $k$.}.

\begin{defn}

We may now define a filtration on $i_{f,+}\scM$, denoted $F_k^{\Torder}i_{f,+}\scM$, by pulling back $F_k^{\Torder}(\scM[\partial_t]\delta):=\sum_{i=0}^k\scM\partial_t^i\delta$ (defined to be 0 for $k<0$) along the isomorphism $i_{f,+}\mathscr{M} \simeq \mathscr{M}[\partial_t] \delta$ stated above. We call this filtration the \emph{t-order filtration} on $i_{f,+}\scM$. It is well behaved with respect to coordinate change since the isomorphism $i_{f,+} \scM \simeq \scM[\partial_t]\delta$ is.

\label{defntord}

\end{defn}

\begin{conv}\label{convlocal1}

Throughout the paper, we abuse notation and frequently refer to, for $f \in \scO_{X,\hspace{1pt}\fx}$, the $\scD_{X\times\bC,(\fx,0)}$-module $i_{f,+}\scO_{X,\hspace{1pt}\fx}(*f)$. By this we mean
\[i_{f,+}\scO_{X,\hspace{1pt}\fx}(*f):=(i_{f,+}\scO_U(*f))_{(\fx,0)}\]
where $U$ is an open neighborhood of $\fx$ on which $f$ extends to a global holomorphic function $f\in\scO_U$. In particular, we have
\begin{equation*}
    i_{f,+} \scO_{X,\hspace{1pt}\fx}(*f) \simeq \scO_{X,\hspace{1pt}\fx}(*f)[\partial_t].
\end{equation*}
    
\end{conv}

\begin{defn}

It is well known that if $\scM$ is regular and holonomic as a $\scD_X$-module, with quasi-unipotent monodromy (for example, when $ \scM = \scO_X$ or $\scM = \scO_X(*f)$), then $i_{f,+}\scM$ is specializable along the smooth divisor $\{t=0\}\subseteq X\times\bC_t$ (see Section 2.a \cite{MSai90}). That is, there exists a unique rational, decreasing, exhaustive, discrete, left-continuous filtration, denoted $V^{\bullet}i_{f,+}\scM$ (which we shorten to $V^{\bullet}$ below for ease of notation), satisfying 
\begin{enumerate}[label=\roman*)]
\item For each $\gamma\in\bQ$, $V^{\gamma}$ is a coherent module over the subring \[\{P\in\scD_{X\times\bC_t}\;|\;P\cdot t^i\in(t^i) \;\forall\; i\geq 0\} \subseteq \scD_{X\times\bC_t},\]
\item For each $\gamma\in\bQ$, $t\cdot V^{\gamma}\subseteq V^{\gamma+1}$, with equality if $\gamma>0$,
\item For each $\gamma\in\bQ$, $\partial_t\cdot V^{\gamma}\subseteq V^{\gamma-1}$,
\item For each $\gamma\in\bQ$, $\partial_tt-\gamma$ acts nilpotently on $\text{gr}_V^{\gamma}:=V^{\gamma}/ V^{>\gamma}$ \footnote{Here, $V^{>\gamma}:=\bigcup_{\gamma'>\gamma}V^{\gamma'}$.}.
\end{enumerate}
This filtration is called the \emph{Kashiwara-Malgrange $V$-filtration} for $\scM$ along $\{t=0\}$. 

\label{defnVfilt}

\end{defn}

We state firstly some basic properties of the Kashiwara-Malgrange $V$-filtrations for $\scO_X$ and $\scO_X(*f)$ and the relations between them.

\begin{prop}
With notation as above, we have:
\begin{enumerate}[label=\roman*)]

\item For all $\gamma\in\bQ$, \[t\cdot V^{\gamma}i_{f,+}\scO_X(*f)=V^{\gamma+1}i_{f,+}\scO_X(*f).\]

\item For all $\gamma\in\bQ_{>0}$, \[V^{\gamma}i_{f,+}\scO_X = V^{\gamma}i_{f,+}\scO_X(*f),\] where we identify $i_{f,+}\scO_X$ with its image under the canonical injective map $i_{f,+}\scO_X\to i_{f,+}\scO_X(*f)$.

\item For all $\gamma \in\bQ_{>-1}$, \[V^{\gamma+1}i_{f,+}\scO_X = t\cdot V^{\gamma}i_{f,+}\scO_X(*f).\]

\end{enumerate}

\label{propVfilt}

\end{prop}

\begin{proof}

See Remarks 2.1-2.3 \cite{MP20}.

\begin{enumerate}[label = \roman*)]

\item The inclusion $\subseteq$, as well as equality for $\gamma>0$, follows from the axiomatic definition of the Kashiwara-Malgrange $V$-filtration. So assume that $\gamma \leq 0$ and that $w \in i_{f,+}\scO_X(*f)$ satisfies $t\cdot w\in V^{\gamma+1}$ (it suffices to consider such a situation since the action of $t$ on $i_{f,+}\scO_X(*f)$ is bijective). Let $\delta\in\mathbb{Q}$ be maximal such that $w \in V^{\delta}$ (such a $\delta$ exists by the left-continuity of the $V$-filtration). By axiom iv) of Definition \ref{defnVfilt}, the map 
\[t\,\cdot : \text{gr}_V^{\gamma'}\to \text{gr}_V^{\gamma'+1}\]
is an isomorphism for any $\gamma' \neq 0$. Indeed, if $u \in \text{gr}_V^{\gamma'}$ such that $t\cdot u =0$, then 
\[0 = (\partial_tt-\gamma')^l\cdot u = (-\gamma')^lu\]
for $l$ sufficiently large, so $u=0$. If instead $u \in \text{gr}_V^{\gamma'+1}$, then 
\[(t\partial_t-\gamma')^l\cdot u =(\partial_tt-(\gamma'+1))^l\cdot u =0\]
for $l$ sufficiently large, implying that
\[(-\gamma')^lu \in t\partial_t\bC[t\partial_t]\cdot u \subseteq t\partial_t\cdot \text{gr}_V^{\gamma'+1}\subseteq t \cdot \text{gr}_V^{\gamma'},\]
implying that $u \in t\cdot \text{gr}_V^{\gamma'}$.

Therefore, if $\delta <\gamma$, since $t\cdot w \in V^{>\delta +1}$, we must have that $w \in V^{>\delta}$ (as $\delta \neq 0$ here), contradicting the definition of $\delta$. Thus $\delta \geq \gamma$, so $w \in V^{\gamma}$, as required.

\item Any element in the cokernel of the canonical injective $\scO_{X\times \bC}$-linear map 
\[i_{f,+}\scO_X\longrightarrow i_{f,+}\scO_X(*f).\]
is annihilated by some power of $t$, so the map becomes an isomorphism when restricted to the open subset $\{t \neq 0\}$ of $X\times\bC$. By Lemma 3.1.7 of \cite{MSai88}, this thus induces isomorphisms
\[V^{\gamma}i_{f,+}\scO_X\longrightarrow V^{\gamma}i_{f,+}\scO_X(*f).\]
for every $\gamma >0$, as required.

\item This follows immediately by combining parts i) and ii) of the proposition.

\end{enumerate}

\end{proof}

\begin{defn} \label{defn-HodgeFiltration}
Let $f \in \scO_X$. Set $D = \text{div}(f)$ and $i: D \to X$ the natural closed embedding. Set $U = X \setminus D$ and $j : U \to X$ the natural open embedding. We write $F_{\bullet}^H \scO_X(*f)$ for the Hodge filtration on $\scO_X(*f)$ underlying the mixed Hodge module $j_* \bQ_U^H[n]$. Identically, we write $F_{\bullet}^Hi_{f,+}\scO_X(*f)$ for the Hodge filtration on $i_{f,+}\scO_X(*f)$ underlying the mixed Hodge module $i_{f,*}j_*\bQ_U^H[n] \in \text{MHM}(X\times\bC)$. 
\end{defn} 

\begin{rem}\label{remlocal2}

Similarly as to in Convention \ref{convlocal1}, throughout the paper we also wish to consider the $\scO_{X,\hspace{1pt},\fx}$-modules $F_{\bullet}^H\scO_{X,\hspace{1pt}\fx}(*f)$ as well as the $\scO_{X\times\bC,\hspace{1pt},(\fx,0)}$-modules $F_{\bullet}^Hi_{f,+}\scO_{X,\hspace{1pt}\fx}(*f)$, when $f\in\scO_{X,\hspace{1pt}\fx}$. Again, these are defined by choosing an extension of $f$ to some neighborhood of $\fx$ and then localizing the Hodge filtrations defined on this neighborhood. 
    
\end{rem}

Instead of using the above definitions, we will focus on algebraic characterizations using the aforementioned isomorphism $i_{f,+}\scO_X(*f) \simeq \scO_X(*f) [\partial_t]$ and the Kashiwara-Malgrange filtration. The reader could safely regard the subsequent Theorem and Lemma as replacements for Definition \ref{defn-HodgeFiltration}.

Mustaţă and Popa proved the following (local) formula for the $k$-th step of the Hodge filtration on $\scO_X(*f)$ in terms of the $1$-st step of the Kashiwara-Malgrange $V$-filtration on $i_{f,+}\scO_X$.

\begin{thm}[Theorem A' \cite{MP20}]

Let $f \in \scO_X$. For $k \geq 0$, \[F_k^H\scO_X(*f) = \left\{\sum_{j=0}^kj!f^{-j-1}v_j \; \Big| \; v_j \in \scO_X, \;\sum_{j=0}^kv_j\partial_t^j\delta \in V^1i_{f,+}\scO_X\right\}.\]

\label{thmMP}

\end{thm}

\begin{rem} \label{rmk-MPHodgeDiscussion}

This theorem can be justified as follows: in the category $D^b\text{MHM}(X)$ of derived mixed Hodge modules on $X$, we have an exact triangle
\[i_*i^!\bQ_X^H[n]\rightarrow\phi_{f,1}\bQ_X^H[n] \xrightarrow{\text{var}} \psi_{f,1}\bQ_X^H[n](-1)\xrightarrow{+1},\]
where $i : D \to X$ is the natural closed embedding, $\bQ_X^H[n]$ is the trivial Hodge module on $X$ and $\psi_{f,1}\bQ_X^H[n]$, $\phi_{f,1}\bQ_X^H[n]$ are the mixed Hodge modules overlying the (perverse) nearby and vanishing cycles of $f$ applied to the perverse sheaf $\underline{\bQ}_X[n]$.

We also have that $\scH^0i_*i^!\bQ_X^H[n]=0$ and that there is a short exact sequence in $\text{MHM}(X)$
\[0 \to \bQ_X[n] \to j_*\bQ_U^H[n] \to \scH^1i_*i^!\bQ_X^H[n] \to 0.\]

\noindent Then it may be shown (see Proposition 5.3 \cite{Olano23}) that the $\scD_X$-module map underlying the resulting surjection of mixed Hodge modules 
\[\psi_{f,1}\bQ_X^H[n](-1) \to \scH^1i_*i^!\bQ_X^H[n]\]
equals (up to multiplication by some invertible function at least)
\[\psi_{-1} : \text{gr}_V^1i_{f,+}\scO_X \to \frac{\scO_X(*f)}{\scO_X},\]
which is the map induced by
\[\psi_{-1} : i_{f,+}\scO_X \to \scO_X(*f) \,\,;\,\, \sum_{j=0}^kv_j\partial_t^j\delta \mapsto \sum_{j=0}^kj!f^{-j-1}v_j.\]

\noindent From this one can derive the expression given in Theorem \ref{thmMP}, using strictness of morphisms of mixed Hodge modules with respect to Hodge filtrations (the Hodge filtration on $\text{gr}_V^1i_{f,+}\scO_X$ being the one induced by the $t$-order filtration on $i_{f,+}\scO_X$).
\end{rem}

\begin{lem} \label{lem:HtoVonGraph}
    Let $f \in \scO_X$. On $i_{f,+}\scO_X(*f)$, and for all integers $k$, we have the following relations between the Hodge and $t$-order filtrations:
    \begin{align}\label{Hfilt}
    F_{k+1}^Hi_{f,+}\scO_X(*f) =\hspace{-2pt}\sum_{j\geq 0} F_{k-j}^H\scO_X(*f)\partial_t^j\delta \subseteq F_k^{\Torder}i_{f,+}\scO_X(*f).
    \end{align}
\end{lem}

\begin{proof}
The final containment is by definition of $F_{k}^{\Torder} i_{f,+}\scO_X(*f)$. The equality involving the Hodge filtration on $i_{f,+}\scO_X(*f)$ follows since $ (i_{f,+}\scO_X(*f),F_{\bullet}^H)=i_{f,+}(\scO_X(*f),F_{\bullet}^H)$ as filtered  $\scD_X$-modules. See \cite{Lau83} as well as Section 1.8 of \cite{MSai93}. 
\end{proof}

\begin{rem} \label{rmk-0HodgeGraph}
Note (as can be seen from the argument in Remark \ref{rmk-MPHodgeDiscussion} for instance) that $F_{-1}^H\scO_X(*f)=0$, from which it follows also that $F_0^Hi_{f,+}\scO_X(*f)=0$.
\end{rem}


The remainder of this section will be devoted to proving a certain rewording of Mustaţă and Popa's result Theorem A' \cite{MP20}. We will give a formula for the $k$-th step of the Hodge filtration on $\scO_X(*f)$ in terms of the zeroeth step of the Kashiwara-Malgrange $V$-filtration on $i_{f,+}\scO_X(*f)$, instead of the first. A simple lemma that aids us in transferring between these two steps is the following.


\begin{lem} \label{lem_HelpCompute}
    Let $f \in \scO_X$. Write $v, u \in i_{f,+} \scO_X(*f)$ as $v = \sum_{j \geq 0} v_j \partial_t^j \delta$ and $u = \sum_{j \geq 0} u_j \partial_t^j \delta$. Suppose that $v = t \cdot u$. Then
    \begin{enumerate}[label=\roman*)]
        \item $v_j = f u_j - (j+1) u_{j+1}$ for all $j \geq 0$.
        \item $u_0 = \sum_{j \geq 0} j! f^{-j-1} v_j $.
    \end{enumerate}
\end{lem}

\begin{proof}
    Both are computations. By assumption,
    \begin{align*}
        \sum_{j \geq 0} v_{j} \partial_t^j \delta = t \cdot  \sum_{j \geq 0} u_j \partial_t^j \delta  = \sum_{j \geq 0} ( f u_j \partial_t^j - j u_j \partial_t^{j-1} ) \delta  = \sum_{j \geq 0} ( f u_j - (j+1) u_{j+1}) \partial_t^j \delta.
    \end{align*}
    Comparing $\{\partial_t^\ell \delta\}_{\ell \geq 0}$ coefficients between $u$ and $v$ proves i). 
    
   Claim ii) follows by plugging i) into the proposed formula:
    \begin{align*}
        \sum_{j \geq 0} j! f^{-j-1} v_j = \sum_{j \geq 0} j! f^{-j-1} (f u_j - (j+1) u_{j+1}) = \sum_{j \geq 0} (j! f^{-j} u_j - (j+1)! f^{-j-1} u_{j+1}) = u_0.
    \end{align*}
\end{proof}

We obtain the following useful relationship between certain restrictions of $F_{\bullet + 1}^H i_{f,+} \scO_X(*f)$ and $F_\bullet^{\Torder} \scO_X(*f)$.

\begin{conv}    
Let $G_\bullet$ be a filtration on $i_{f,+} \scO_X(*f)$ and $\gamma \in \mathbb{Q}$. Throughout the article, we use $G_\bullet V^{\gamma} i_{f,+} \scO_X(*f)$ to denote the restriction of $G_\bullet$ to $V^{\gamma} i_{f,+} \scO_X(*f)$. 
\end{conv}

\begin{lem} \label{lemHtordV}
Let $f \in \scO_X$. Then
\[F_{\bullet+1}^HV^0i_{f,+}\scO_X(*f) = F_{\bullet}^{\Torder}V^0i_{f,+}\scO_X(*f),\]

\end{lem}

\begin{proof}
It clearly suffices to check this statement locally. So we may assume $X$ is a sufficiently small open subset of $\bC^n$ and pick a coordinate system. The result is trivially true for negative indices, cf. Remark \ref{rmk-0HodgeGraph}. Let $k \in \bZ_{\geq 0}$.
    
The $\subseteq$ containment is read off of Lemma \ref{lem:HtoVonGraph}. Now we prove the reverse. Firstly, throughout the argument we convene that
\begin{equation*}
    u = \sum_{j=0}^ku_j\partial_t^j\delta \in F_k^{t\text{-ord}}V^0i_{f,+}\scO_X(*f)
\end{equation*}
is our arbitrary choice of element in $F_k^{t\text{-ord}}V^0i_{f,+}\scO_X(*f).$  By Proposition \ref{propVfilt}.iii), $t \cdot u \in V^1i_{f,+}\scO_X$. For the rest of the proof we also convene that 
\begin{equation*}
    v := t \cdot u \text{ and we write } v = \sum_{j\geq 0} v_j \partial_t^j\delta.
\end{equation*}
Employing Lemma \ref{lem_HelpCompute}.ii) and Theorem \ref{thmMP} yields 
\begin{equation} \label{eqn-u_0Fact}
    u_0 \in F_k^H\scO_X(*f). 
\end{equation}

We are ready to prove that $u \in F_{k+1}^H i_{f,+} \scO_X(*f)$. We use induction on the index $k$. In the case $k=0$, we have 
\begin{equation*}
    u = u_0 \delta \in F_0^H \scO_X(*f) \delta = F_1^H i_{f,+} \scO_X(*f),
\end{equation*}
where: the membership is \eqref{eqn-u_0Fact}; the equality is Lemma \ref{lem:HtoVonGraph}.

For $k > 0$, recall our conventions of $v = t \cdot u$. So,
\begin{equation*}
    v=t \cdot \sum_{j\geq 0} u_j \partial_t^j\delta = fu- \sum_{j= 1}^kju_j\partial_t^{j-1}\delta \in V^1i_{f,+}\scO_X,
\end{equation*}
where $V^1i_{f,+}\scO_X =t\cdot V^0i_{f,+}\scO_X(*f)\subseteq V^0i_{f,+}\scO_X(*f)$, c.f. Proposition \ref{propVfilt}.iii). Counting $\partial_t^{\ell}$ powers in $fu-v$ reveals $fu-v \in F_{k-1}^{\Torder} i_{f,+}\scO_X(*f)$, which in turn shows 
\begin{equation*}
    \sum_{j= 1}^kju_j\partial_t^{j-1}\delta =fu-v \in F_{k-1}^{t\text{-ord}}V^0i_{f,+}\scO_X(*f).
\end{equation*}
By the inductive hypothesis and Lemma \ref{lem:HtoVonGraph}, for all $j \geq 1$ each $j u_j \in F_{k-1-(j-1)}^H \scO_X(*f) = F_{k-j}^H \scO_X(*f)$. By \eqref{eqn-u_0Fact}, $u_0 \in F_k^H \scO_X(*f)$. Using Lemma \ref{lem:HtoVonGraph} again terminates the induction:
\begin{equation*}
    u = \sum_{j=0}^ku_j\partial_t^j\delta \in \sum_{j\geq 0}F_{k-j}^H\scO_X(*f)\partial_t^j\delta = F_{k+1}^Hi_{f,+}\scO_X(*f).
\end{equation*}
\end{proof}



We deduce our variant/re-framing of Mustaţă and Popa's Theorem \ref{thmMP}:

\begin{thm}

Let $f \in \scO_X$. For each non-negative integer $k$,
\begin{center} \vspace{-0.4cm} \begin{align*}F_k^H\scO_X(*f) & = \left(\partial_tF_k^Hi_{f,+}\scO_X(*f)+V^0i_{f,+}\scO_X(*f)\right)\cap\left(\scO_X(*f)\delta\right) \\ &= \left(\partial_tF_{k-1}^{t-\text{\emph{ord}}}i_{f,+}\scO_X(*f)+V^0i_{f,+}\scO_X(*f)\right)\cap\left(\scO_X(*f)\delta\right),\end{align*}\end{center} where the two intersections take place in $\scO_X(*f)[\partial_t]\delta\simeq i_{f,+}\scO_X(*f)$.

\label{thmformula}

\end{thm}

\begin{proof}
Again we may assume without loss of generality that $X$ is an open subset of $\bC^n$. 

Hereafter, fix $k \in \mathbb{Z}_{\geq 0}$. We first prove the second equality promised. The $\subseteq$ containment is Lemma \ref{lem:HtoVonGraph}. For the $\supseteq$ containment, take $u_0 \in \scO_X(* f)$ such that
\begin{equation*}
    u_0\delta \in (\partial_tF_{k-1}^{\Torder}i_{f,+}\scO_X(*f)+V^0i_{f,+}\scO_X(*f))\cap\left(\scO_X(*f)\delta\right).
\end{equation*}
As we are entitled to, set $u_0\delta = \partial_t u^\prime + u$ with $u^\prime \in F_{k-1}^{\Torder}i_{f,+}\scO_X(*f)$ and $u \in V^0i_{f,+}\scO_X(*f)$. By the definition of the $\partial_t \cdot$ action, $\partial_t u^\prime$ lies in $F_k^{\Torder} i_{f,+} \scO_X(*f)$ as does $u = u_0 \delta - \partial_t u^\prime$. Therefore
\begin{equation*}
    u = u_0 \delta - \partial_t u^\prime \in F_k^{\Torder}V^0i_{f,+}\scO_X(*f) = F_{k+1}^H V^0 i_{f,+} \scO_X(*f),
\end{equation*}
where the last claim is Lemma \ref{lemHtordV}. By considering the $\{\partial_t^\ell \delta\}_{\ell \geq 0}$ coefficients of $u = u_0 \delta - \partial_t u^\prime$ and using Lemma \ref{lem:HtoVonGraph}, it is clear that $u_0 \in F_k^H \scO_X(*f)$ and that $u^\prime \in F_k^H i_{f,+} \scO_X(*f)$. So
\begin{equation*}
    u_0 \delta = \partial_t u^\prime + u \in \big( \partial_t F_k^H i_{f,+} \scO_X(*f) + V^0 i_{f,+} \scO_X(*f) \big) \cap \big( \scO_X(*f) \delta \big).
\end{equation*}
This establishes the second promised equality, by confirming $\supseteq$. Our observations about $u_0$ also proved that
\begin{equation} \label{eqn-thmFormula-1}
    \left( \partial_tF_{k-1}^{\Torder} i_{f,+}\scO_X(*f)+V^0i_{f,+}\scO_X(*f) \right) \cap \left(\scO_X(*f)\delta\right) \subseteq F_k^H\scO_X(*f)\delta.
\end{equation}

Now we prove the first ``$=$''. With \eqref{eqn-thmFormula-1} in hand, our task is to certify
\begin{equation} \label{eqn-thmFormula-2}
F_k^H \scO_X(*f)\delta \subseteq \left( \partial_t F_{k-1}^{\Torder} i_{f,+} \scO_X(*f) + V^0i_{f,+}\scO_X(*f)\right) \cap \left(\scO_X(*f)\delta\right).
\end{equation}
So select $q \in F_k^H \scO(*f).$ By Theorem \ref{thmMP}, $q$ may be written 
\begin{equation*}
q = \sum_{j=0}^kj!f^{-j-1}v_j \;\;\;\; \text{for some} \;\;\;\; v = \sum_{j=0}^kv_j\partial_t^j\delta \in V^1i_{f,+}\scO_X.
\end{equation*}
Proposition \ref{propVfilt}.iii) asserts the existence of some $u = \sum_{j \geq 0} u_j \partial_t^j \delta \in V^0i_{f,+}\scO_X(*f)$ such that $v = t \cdot u$. Lemma \ref{lem_HelpCompute}.ii) computes $q = u_0$. Because $v$ has $t$-order at most $k$, Lemma \ref{lem_HelpCompute}.i) demonstrates that $u$ also has $t$-order at most $k$. Manipulating the $\partial_t \cdot$ action yields
\begin{align*}
    q \delta = u_0 \delta 
    = (- \sum_{j \geq 1}^{k} u_j \partial_t^j \delta ) + u 
    \in \big( \partial_t F_{k-1}^{\Torder} i_{f,+} \scO_X(*f) + V^0 i_{f,+} \scO_X(*f) \big) \cap \big(\scO_X(*f) \delta \big).
\end{align*}
So \eqref{eqn-thmFormula-2} holds and the proof is complete.
\end{proof}

\section{Hypotheses on the Bernstein-Sato polynomial} 

In the previous section, we reduced understanding the Hodge filtration on $\scO_X(*f)$ into the following three ingredients: the $t$-order filtration on $i_{f,+} \scO_X(*f)$; $V^0 i_{f,+} \scO_X(*f)$; the behavior of intersecting $i_{f,+} \scO_X(*f)$ with $\scO_X(*f)$. 

Our goal is to simplify our understanding: we would like to understand the (respective) Hodge filtrations on $\scO_X(*f)$, $i_{f,+} \scO_X(*f)$, and $V^0 i_{f,+} \scO_X(*f)$ in terms of more tractable filtrations. Namely: the order filtration on $\scD_X$; the order filtration on $\scD_{X \times \bC_t}$; the $V$-filtration on $\scD_{X \times \bC_t}$ along $\{t=0\}$; the total order filtration on $\scD_X[s]$. This introduces some costs. 
 
Firstly, $\scO_X(*f)$ is not $\scD_X$ and so it does not naturally have the order filtration. We deal with this by selecting a generator $q$ for $\scO_X(*f)$ and letting the \emph{induced order $k$} elements of $\scO_X(*f)$ be the elements $P q$ where $P \in \scD_X$ has order at most $k$, cf. Definition \ref{def-orderFilLocalization}. In a similar way we can select a generator for $i_{f,+} \scO_X(*f)$ and define the \emph{induced $V$-filtration} to be the result of applying the aforementioned $V$-filtration on $\scD_{X \times \bC_t}$ to this generator, cf. Definition \ref{def-inducedVfiltration}.

Secondly, we will have to work with $\scD_X[s]$-modules and Bernstein-Sato constructions. This is not unexpected: $\scD_X[s]$-modules and the Bernstein-Sato polynomial are familiar actors in the story of $i_{f,+} \scO_X(*f)$ and $\scO_X(*f)$. Lemma \ref{lem-commDiagram} hints at how to pass filtration data on one object to filtration data on another. Eventually, we will focus on the case of $f \in \scO_X$ whose Bernstein-Sato polynomial's roots are contained in $(-2,0)$. This forces $-1$ to be the only integer root of the Bernstein-Sato polynomial, which is equivalent to the equality $\scD_X\cdot f^{-1} = \scO_X(*f)$, cf. Proposition \ref{prop-specialization}.

While much of the machinations here are for later use, this section concludes with Corollary \ref{cor-zeroHodgePiece}. This is an algebraic formula for $F_0^H \scO_X(*f)$, provided the zeroes of the Bernstein-Sato polynomial of $f$ are contained in $(-2,0)$.

\subsection{Order Filtrations and Picking a Generator}

Recall the \emph{order filtration} $F_\bullet^{\ord} \scD_X$ on $\scD_X$ (resp. $F_\bullet^{\ord} \scD_{X \times \bC_t})$ is a good, increasing, and exhaustive filtration by coherent $\scO_X$-modules (resp. $\scO_{X \times \bC_t}$-modules). As promised, we will use canonical filtrations to induce ``simple'' filtrations on $\scO_X(*f)$ and, in a similar way, on $i_{f,+} \scO_X(*f)$. 

To do this we must fix a generator of $\scO_X(*f)$. This is possible: $\scO_X(*f)$ is cyclically generated by $f^\ell$ provided $\ell \ll 0$ (see a natural generalization of Proposition \ref{prop-specialization}). Hereafter we \emph{always} assume that $f^{-1}$ is a generator, i.e. $\scD_X\cdot f^{-1} = \scO(*f)$, and we will diligently remind the reader so. (Different generator choices mutate the following definitions appropriately.)

\begin{defn} \label{def-orderFilLocalization}
    Assume that $f \in \scO_X$ satisfies $\scD_X\cdot f^{-1} = \scO_X(*f).$ For $k \in \mathbb{Z}$ set
    \begin{equation*}
        F_k^{\ord} \scO_X(*f) := F_k^{\ord} \scD_X\cdot f^{-1}.
    \end{equation*}
    We call the filtration $F_\bullet^{\ord} \scO_X(*f)$ the \emph{order filtration} on the $\scD_X$-module $\scO_X(*f)$. It is a good, increasing, and exhaustive filtration by $\scO_X$-modules. 
\end{defn}

If the $\scD_X$-module $\scO_X(*f)$ is generated by $f^{-1}$, then the $\scD_{X \times \bC_t}$-module $\scO_X(*f)[\partial_t]\delta$ is generated by $f^{-1} \delta$. Using
\begin{equation} \label{eqn-GraphIso}
    i_{f,+}\scO_X(*f) \simeq \scO_X(*f)[\partial_t] \delta,
\end{equation}
the $\scD_{X \times \bC_t}$-module $i_{f,+} \scO_X(*f)$ is generated by the preimage of $f^{-1} \delta$.

\begin{defn} \label{def-orderFilGraphLocalization}
    Assume that $f \in \scO_X$ satisfies $\scD_X\cdot f^{-1} = \scO_X(*f).$ For $k \in \mathbb{Z}$ set
    \begin{equation*}
        F_k^{\ord} \scO_X(*f)[\partial_t] \delta := F_k^{\ord} \scD_{X \times \bC_t} \cdot f^{-1} \delta,
    \end{equation*}
    where here $F_k^{\ord} \scD_{X \times \bC_t}$ is the canonical order filtration on $\scD_{X \times \bC_t}$. We call the filtration $F_\bullet^{\ord} \scO_X(*f)[\partial_t]\delta $ the \emph{order filtration} on the $\scD_{X \times \bC_t}$-module $\scO(*f)[\partial_t] \delta$. We also obtain the \emph{order filtration} $F_\bullet^{\ord} i_{f,+} \scO_X(*f)$ on the $\scD_{X \times \bC_t}$-module $i_{f,+} \scO_X(*f)$ by pulling back $F_\bullet^{\ord} \scO_X(*f)[\partial_t] \delta$ along the isomorphism \eqref{eqn-GraphIso}. Both are good, increasing, and exhaustive filtrations by $\scO_{X \times \bC_t}$-modules.
\end{defn}

There is one more ``simple'' filtration we introduce. This still involves the generator choice/assumption $f^{-1}$ of $\scO_X(*f)$, but does not involve the canonical order filtration on $\scD_X$ or $\scD_{X \times \bC_t}$. Instead we exploit the $V$-filtration on $\scD_{X \times \bC_t}$:

\begin{defn} \label{def-inducedVfiltration}
    Recall the $V$-filtration of $\scD_{X \times \bC_t}$ along $\{t = 0\}$, denoted as $V^\bullet \scD_{X \times \bC_t}$, and characterized by
    \begin{equation*}
        V^k \scD_{X \times \bC_t} := \{ P \in \scD_{X \times \bC_t} \mid P \cdot t^i \in (t^{i+k}) \text{ for all } i \geq 0 \}.
    \end{equation*}
    This is a decreasing integer filtration by coherent $V^0 \scD_{X \times \bC_t}$-modules. 

    Now assume that $f \in \scO_X(*f)$ satisfies $\scD_X\cdot f^{-1} = \scO_X(*f)$. Then we define the \emph{induced $V$-filtration} $V_{\ind}^\bullet \scO_X(*f)[\partial_t]\delta$ on the $\scD_{X \times \bC_t}$-module $\scO_X(*f)[\partial_t]\delta$ by 
    \begin{equation*}
        V_{\ind}^k \scO_X(*f)[\partial_t] \delta := V^k \scD_{X \times \bC_t} \cdot f^{-1} \delta.
    \end{equation*}
    We define the \emph{induced $V$-filtration} $V_{\ind}^\bullet i_{f,+} \scO_X(*f)$ on the $\scD_{X \times \bC_t}$-module $i_{f,+} \scO_X(*f)$ by pulling back $V_{\ind}^\bullet \scO(*f)[\partial_t] \delta$ along the isomorphism \eqref{eqn-GraphIso}. Both induced $V$-filtrations are decreasing integer filtrations by coherent $V^0\scD_{X\times\bC_t}$-modules. 
\end{defn}

\begin{conv} \label{conv-PullBack}
    We emphasize that both the induced order and induced $V$-filtration on $i_{f,+} \scO_X(*f)$ are obtained by defining these filtrations on $\scO_X(*f)[\partial_t] \delta$ and pulling them back along the isomorphism \eqref{eqn-GraphIso}, just like $F_\bullet^{\Torder} i_{f,+} \scO_X(*f)$. When we give formulas like $F_k^{\ord} i_{f,+} \scO_X(*f) = \cdots$, the ``$=$'' sign implicitly uses this isomorphism.
\end{conv}

The order filtrations on $\scO_X(*f)$ and $i_{f,+} \scO_X(*f)$ are related by:

\begin{lem} \label{lem-OrderFilGraph}
    Assume that $f\in \scO_X$ satisfies $\scO_X(*f)=\scD_X\cdot f^{-1}$. Let $k \in \mathbb{Z}$. Then
    \begin{equation*}
        F_k^{\ord}i_{f,+}\scO_X(*f) = \sum_{j\geq 0} \bigg[ F_{k-j}^{\ord}\scO_X(*f) \bigg] \partial_t^j\delta.
    \end{equation*}
\end{lem}

\begin{proof}
    The statement for $k \in \mathbb{Z}_{\leq 0}$ is trivial. So we must show equality for $k \in \mathbb{Z}_{> 0}$.

    First, take $u \in F_k^{\ord} \scO_X(*f)[\partial_t]\delta$. By rule, we find $P \in F_k^{\ord} \scD_{X \times \bC_t}$ such that $u = P \cdot f^{-1} \delta$. Since $t$ acts on $\delta$ via multiplication by $f$, we may choose $P$ in a way for which we can write $P = \sum_{i = 0}^{k} P_{k-i} \partial_t^i$ for suitable $P_{k-i} \in F_{k-i}^{\ord} \scD_X$. Recalling the graph embedding action, we compute:
    \begin{equation} \label{eqn-orderFiltrations-1}
        (P_{k-i} \partial_t^i) \cdot f^{-1} \delta = \sum_{j = i}^{k} u_{k-i,j} \partial_t^j \delta \quad \text{where} \quad u_{k-i,j} \in F_{k-j}^{\ord} \scO_X(*f).
    \end{equation}
    So
    \begin{align*}
        P \cdot f^{-1} \delta = \sum_{i=0}^{k} \sum_{j=i}^k u_{k-i,j} \partial_t^j \delta = \sum_{j=0}^{k} \bigg[ \sum_{i=j}^k u_{k-i,j} \bigg] \partial_t^j \delta \in \sum_{j=0}^k \bigg[ F_{k-j}^{\ord} \scO_X(*f) \bigg] \partial_t^j \delta,
    \end{align*}\
    proving the containment $\subseteq$. 

    For the reverse containment $\supseteq$, it suffices to prove that
    \begin{equation} \label{eqn-orderFiltrations-2}
        \bigg[ F_{k-i}^{\ord} \scO_X(*f) \bigg] \partial_t^i \delta \subseteq F_k^{\ord} \scO_X(*f)[\partial_t] \delta
    \end{equation}
    holds for all $i \in \mathbb{Z}$ with $0 \leq i \leq k$. We do this by decreasing induction on $i$. When $i = k$,
    \begin{equation*}
        \bigg[ F_{0}^{\ord} \scO_X(*f) \bigg] \partial_t^k = (\scO_X \partial_t^k) \cdot f^{-1} \delta \subseteq F_k^{\ord} \scO_X(*f)[\partial_t] \delta,
    \end{equation*}
    confirming \eqref{eqn-orderFiltrations-2}. 

    For the inductive step, fix $i < k$ and assume that \eqref{eqn-orderFiltrations-2} holds for larger choices of $i$. Select $P_{k-i} \in F_{k-i}^{\ord} \scD_X$ and set $v_{k-i}= P_{k-i} \cdot f^{-1} \in \scO_X(*f)$. Use \eqref{eqn-orderFiltrations-1} to compute $(P_{k-i} \partial_t^i) \cdot f^{-1} \delta$. By comparing $\{\partial_t^\ell \delta\}$ coefficients of $(P_{k-i} \partial_t^i) \cdot f^{-1} \delta$, one checks that $v_{k-i} = u_{k-i,i}$. So
    \begin{equation*}
        v_{k-i} \partial_t^i \delta = \big( (P_{k-i} \partial_t^i) \cdot f^{-1} \delta \big) - \sum_{j \geq i+1}^{k} u_{k-i, j} \partial_t^j \delta \in  F_{k}^{\ord} \scO_X(*f)[\partial_t]\delta
    \end{equation*}
    where ``$\in$'' follows by the inductive hypothesis plus $u_{k-i,j} \in F_{k-j}^{\ord}\scO_X(*f)$ (cf. \eqref{eqn-orderFiltrations-1}).
\end{proof}

We obtain equivalent containments relating the Hodge and order filtrations on $\scO_X(*f)$ and $i_{f,+} \scO_X(*f)$:

\begin{lem} \label{lemorderfilt2-new}
Assume that $f\in\scO_X$ satisfies that $\scO_X(*f)=\scD_X\cdot f^{-1}$. Let $k$ be a non-negative integer. Then the following conditions are equivalent:
\begin{enumerate}[label=\roman*)]
    \item $F_\ell^H\scO_X(*f) \subseteq F_\ell^{\text{\emph{ord}}}\scO_X(*f)$ for all $0 \leq \ell \leq k$.
    \item $F_{\ell+1}^Hi_{f,+}\scO_X(*f) \subseteq F^{\text{\emph{ord}}}_\ell i_{f,+}\scO_X(*f)$ for all $0 \leq \ell \leq k$.
    \item $F_{\ell+1}^HV^0i_{f,+}\scO_X(*f) = F_\ell^{\text{\emph{ord}}}V^0i_{f,+}\scO_X(*f)$ for all $0 \leq \ell \leq k$.
\end{enumerate}
\end{lem}

\begin{proof}
    First we prove $ii) \implies iii)$. The inclusion $\subseteq$ follows by hypothesis and definition. The reverse inclusion follows from Lemma \ref{lemHtordV} and from the fact that the $t$-order filtration is coarser than the order filtration. As for $i) \implies ii)$, observe:
    \begin{equation*}
        F_{\ell+1}^H i_{f,+} \scO_X(*f) = \sum_{j \geq 0} \big[ F_{\ell-j}^{H} \scO_X(*f) \big] \partial_t^j \delta \subseteq \sum_{j \geq 0} \big[ F_{\ell-j}^{\ord} \scO_X(*f) \big] \partial_t^j \delta = F_{\ell}^{\ord} i_{f,+} \scO_X(*f).
    \end{equation*}
    Here: the first ``$=$'' is Lemma \ref{lem:HtoVonGraph}; the ``$\subseteq$'' is our assumption $i)$; the last ``$=$'' is Lemma \ref{lem-OrderFilGraph}. 

    Validating $iii) \implies i)$ remains. So select $u_0 \in F_\ell^H \scO_X(*f)$. Invoking Lemma \ref{lemHtordV} and Theorem \ref{thmformula}, we find $u = u_0 \delta + v \in F_{\ell + 1}^{H} V^0 i_{f,+} \scO_X(*f)$, where $v = \sum_{j \geq 1} v_j \partial_t^j \delta$ for $v_j \in \scO_X(*f)$. By assumption $iii)$, we have $u \in F_\ell^{\ord} i_{f,+} \scO_X(*f)$. So we can write $u = P \cdot f^{-1} \delta$ where $P = \sum_{j\geq 0} P_{\ell-j} \partial_t^j$ and $P_{\ell-j} \in F_{\ell-j} \scD_X$. (Again, we have used the definition of the $t$-action on $\scO_X(*f)\delta$ to remove $t$-terms.) By the definition of the action of $\scD_{X \times \bC_t}$ on $\scO_X(*f)[\partial_t]\delta$ we deduce $i)$, since then $u_0 = P_{\ell} \cdot f^{-1} \in F_{\ell}^{\ord} \scO_X(*f).$
\end{proof}

\subsection{Bernstein-Sato Preliminaries}

Let $\scD_X[s] = \scD_X \otimes_\bC \bC[s]$, for $s$ a new variable. We often denote elements of $\scD_X[s]$ by $P(s) = \sum_{j \geq 0} P_j s^j$ where $P_j \in \scD_X$. First, another filtration:

\begin{defn} \label{def-totalOrderFiltration}
    The \emph{total order filtration} $F_{\bullet}^\sharp \scD_X[s]$ on $\scD_X[s]$ is exhaustive, increasing, indexed by $\mathbb{Z}$, and characterized by 
    \begin{equation*}
        F_k^\sharp \scD_X[s] := \sum_{i = 0}^{k} F_{k-i}^{\ord} \scD_X \otimes_{\bC} \bC[s]_{\leq i}
    \end{equation*}
    where $\bC[s]_{\leq i}$ is the set of elements of degree at most $i$ in the variable $s$.
\end{defn}

Let $f \in \scO_X$. Consider the ring $\scO_X(*f)[s] = \scO_X(*f) \otimes_\bC \bC[s]$ and denote the cyclic $\scO_X(*f)[s]$-module generated by the symbol $f^s$ by $\scO_X(*f)[s] f^s$. This module carries a natural $\scD_X[s]$-module structure. The action of $\scO_X[s]$ is inherited by construction. As for a derivation $\delta \in \text{Der}_{\bC}(\scO_X)$, the action is given by formal application of the chain and product rules:
\begin{equation*}
    \delta \cdot \left( \frac{g s^i}{f^\ell} f^s \right) = s^i \delta \cdot \left( \frac{g}{f^{\ell}} \right) f^{s} + s^{i+1} \frac{ g \big( \delta \cdot f \big)}{f^{\ell+1}} f^s
\end{equation*}
where $g \in \scO_X$, $i, \ell \in \mathbb{Z}_{\geq 0}$. One checks that this extends to a $\scD_X$-action, and to the level of sheaves.

It is well known that $\scO(*f)[s] f^s$ is closely related to $i_{f,+}\scO_X(*f)$. We cite a result of Mustaţă and Popa:

\begin{prop}[Proposition 2.5 \cite{MP20}] \label{prop-MP-graphIsofs}
    Let $\scD_{X} \langle t,s \rangle$ be the subsheaf of $\scD_{X \times \bC_t}$ generated by $\scD_{X}$, $t$, and $s = - \partial_t t$. Let $\scD_{X} \langle t, t^{-1}, s \rangle = \scD_{X} \langle t, t^{-1}, \partial_t \rangle$ be the localization at $t$ (i.e. the pushforward of the sheaf of differential operators along $X \times \bC^* \to X \times \bC$). Then $\scO_X(*f)[s]f^s$ is a $\scD_{X} \langle t, t^{-1}, s \rangle $-module, given by extending the $\scD_{X}$-action, using the obvious $s$-action, decreeing that $t \cdot u s^j f^s = f u (s+1)^j f^s$, and letting $t^{-1}$ act in an inverse manner. Moreover, we have an isomorphism of $\scD_{X} \langle t, t^{-1}, s \rangle$-modules
    \begin{equation} \label{eqn-varphiMap}
        \varphi_f : \scO_X(*f)[s] f^s \to i_{f,+} \scO_X(*f)
    \end{equation}
    given by $us^j f^s \mapsto u (-\partial_t t)^j \cdot \delta$.
\end{prop}

Clearly $\scO_X(*f)[s]f^s$ is unwieldy as a $\scD_X[s]$-module: for one, coherence over $\scD_X[s]$ fails thanks to $\{f^\ell f^s\}_{\ell \in \mathbb{Z}}$. So we will work with coherent, and even cyclic, submodules. To facilitate, for all $\ell \in \mathbb{Z}$ we denote $f^\ell f^s$ as $f^{s+\ell}$. Furthermore, set
\begin{equation*}
    \scD_X[s]\cdot f^{s+\ell} = \text{the cyclic $\scD_X[s]$-submodule of $\scO_X(*f)[s] f^s$ generated by $f^{s+\ell}$},
\end{equation*}
which is canonically isomorphic to $\scD_X[s] / \ann_{\scD_X[s]} f^{s+\ell}$.

Proposition \ref{prop-MP-graphIsofs} showed a tight relationship between $\scO_X(*f)[s]f^s$ and $i_{f,+}\scO_X(*f)$. There is a looser relationship between $\scD_X[s]\cdot f^{s + \ell}$ and $\scD_X\cdot f^j \subseteq \scO_X(*f)$:

\begin{defn} \label{def-specialization}
    Let $\ell, j \in \mathbb{Z}$. The \emph{specialization} map
    \begin{equation*}
        \spe_{s+\ell \mapsto j} : \scD_X[s]\cdot f^{s+\ell} \twoheadrightarrow \scD_X\cdot f^j
    \end{equation*}
    is induced by $\sum u_t (s+\ell)^t f^s+\ell \mapsto u_t j^t f^j$ for $u_t \in \scO_X(*f)$ (or equivalently, $P(s)\cdot f^{s + \ell} \mapsto P(j)\cdot f^j$). It is a surjective map of $\scD_X$-modules naturally factoring as
    \begin{equation*}
        \overline{\spe_{s+\ell \mapsto j}} : \frac{ \scD_X[s]\cdot f^{s+\ell}}{\scD_X[s] \cdot (s+\ell - j) f^{s+\ell}} \twoheadrightarrow \scD_X\cdot f^j ,
    \end{equation*}
    which is also surjective and $\scD_X$-linear. We also call $\overline{\spe_{s+\ell \mapsto j}}$ the \emph{specialization} map, trusting context to disambiguate.
\end{defn}

Proposition \ref{prop-MP-graphIsofs} relates $\scO_X(*f)[s]$ and $\scO_X(*f)$; Definition \ref{def-specialization} relates $\scO_X f^{s-1} \subseteq \scO_X(*f)[s]$ to $\scO_X(*f)$. The following connects all these objects simultaneously. Note that all objects in Lemma \ref{lem-commDiagram} are $\scD_X$-modules.

\begin{lem} \label{lem-commDiagram}
    Let $f \in \scO_X$. Define the $\scD_X$-linear map $\phi_{0} : \scD_X[s] \to \scD_X$ by $P(s) \mapsto P(0)$; define the $\scD_X$-linear map $\widetilde{\phi}_{f,0}: \scD_X[s] \to \scO_X(*f)$ by $P(s) \mapsto \phi_{0}(P(s)) \cdot f^{-1}$. Then $\widetilde{\phi}_{f,0}$ factors as in the commutative diagram of $\scD_X$-linear maps below:
    \begin{equation}
        \begin{tikzcd}
            \widetilde{\phi}_{f,0}: \scD_X[s] \arrow[rr, bend left = 20, "\pi_f"] \arrow[r, twoheadrightarrow] 
                & \scD_X[s]\cdot f^{s-1} \arrow[r, hookrightarrow, "\varphi_f"] \arrow[rr, bend right = 20, swap, "\specialize"]
                    & i_{f,+}\scO_X(*f) \arrow[r, "\psi_{f,0}"]
                        & \scO_X(*f) .
        \end{tikzcd}
    \end{equation}
    Here: 
    \begin{itemize}
        \item $\pi_f(P(s)) = P(-\partial_t t) \cdot f^{-1} \delta$; 
        \item $\varphi_f$ is restriction of the map of Proposition \ref{prop-MP-graphIsofs} to $\scD_X[s]\cdot f^{s-1}$;
        \item $\specialize$ is the specialization map;
        \item $\psi_{f,0}$ is induced by $\partial_t \mapsto 0$ and $u \delta \mapsto u$ for $u \in \scO_X(*f)$;
    \end{itemize}
    Moreover, $\ker \pi_f = \ann_{\scD_X[s]} f^{s-1}$.
\end{lem}

\begin{proof}
    The commutativity follows by construction. To check the claim about $\ker \pi_f$, use that $\ann_{\scD_X[s]} f^{s-1}$ is the kernel of $\scD_X[s] \twoheadrightarrow \scD_X[s]\cdot f^{s-1}$ and the injectivity of $\varphi_f$. 
\end{proof}

We next define the Bernstein-Sato polynomial of $f \in \scO_{X,\hspace{1pt}\fx}$. (We consider analytic germs as otherwise the Bernstein-Sato polynomial may not exist.) This is one of our main characters: not only does it inform the relationship between $\overline{\spe_{s-1 \mapsto -1}}$, $\scD_X[s]\cdot f^{s-1}$, $\scD_X\cdot f^{-1}$ and $\scO_X(*f)$ as well as the diagram of Lemma \ref{lem-commDiagram}, but we eventually use it to give a description of the $V$-filtration on $i_{f,+} \scO_X(*f)$ in terms of the induced $V$-filtration. 

\begin{defn} \label{def-BSpoly}
    Let $f \in \scO_{X,\hspace{1pt}\fx}$. The \emph{Bernstein-Sato polynomial} $b_f(s)$ of $f$ is the monic, minimal degree polynomial $b_f(s) \in \bC[s]$ satisfying the \emph{functional equation}
    \begin{equation*}
        b_f(s) f^s \in \scD_{X,\hspace{1pt}\fx}[s]\cdot f^{s+1}.
    \end{equation*}
    Equivalently, $b_f(s)$ is the monic generator of the $\bC[s]$-ideal 
    \begin{equation*}
        \ann_{\bC[s]} \frac{ \scD_{X,\hspace{1pt}\fx}[s]\cdot f^{s}}{ \scD_{X,\hspace{1pt}\fx}[s]\cdot f^{s+1}} = \ann_{\bC[s]} \frac{\scD_{X,\hspace{1pt}\fx}[s]}{\ann_{\scD_{X,\hspace{1pt}\fx}[s]} f^s + \scD_{X,\hspace{1pt}\fx}[s] \cdot f } = \bC[s] \cap \big( \ann_{\scD_{X,\hspace{1pt}\fx}[s]} f^s + \scD_{X,\hspace{1pt}\fx}[s] \cdot f \big).
    \end{equation*}
    We denote the zeroes of $b_{f}(s)$ by $Z(b_{f}(s))$.
\end{defn}

\begin{rem}
    The proof Bernstein-Sato polynomials exist is nontrivial. The first proof in the algebraic case is due to Bernstein \cite{Bernstein72}; the first proof in the local analytic case is due to Kashiwara \cite{Kashiwara77}. We refer to the survey \cite{WaltherSurvey} for general properties and connections to other singularity invariants.
\end{rem}

We conclude the subsection by using the Bernstein-Sato polynomial to characterize not only when $\specialize$ is a $\scD_X$-isomorphism, but also when its image is all of $\scO_X(*f)$. This fulfills part of our promise to clarify the precise relationship between $\scD_X[s]\cdot f^{s-1}$ and $\scO_X(*f)$. The following seems to be well-known to specialists, essentially appearing in Appendix A of \cite{BS23}, some pieces appearing in the proof of Lemma 6.21 \cite{Ka01}, and other pieces elsewhere, cf. Proposition 6.3.15 \cite{Bj93} and Section 3 \cite{OT99}.

\begin{prop} \label{prop-specialization}
    Consider $f \in \scO_{X,\hspace{1pt}\fx}$. Then the following are equivalent:
    \begin{enumerate}[label=\roman*)]
        \item $Z(b_{f}(s)) \cap \mathbb{Z} = \{-1\}$;
        \item $f^{-1}$ generates $\scO_{X,\hspace{1pt}\fx}(*f)$ as a $\scD_{X,\hspace{1pt}\fx}$-module, i.e. $\scD_{X,\hspace{1pt}\fx}\cdot f^{-1} = \scO_{X,\hspace{1pt}\fx}(*f)$;
        \item the $\scD_{X,\hspace{1pt}\fx}$-linear specialization map below is an isomorphism:
        \begin{equation*}
            \overline{\specialize} : \frac{\scD_{X,\hspace{1pt}\fx}[s]\cdot f^{s-1}}{\scD_{X,\hspace{1pt}\fx}[s] \cdot s f^{s-1}} \twoheadrightarrow \scD_{X,\hspace{1pt}\fx}\cdot f^{-1}.
        \end{equation*}
    \end{enumerate}
\end{prop}

\begin{proof}
    For $\ell \in \mathbb{Z}$, consider the commutative diagram of $\scD_{X,\hspace{1pt}\fx}$-maps
    \begin{equation} \label{eqn-specializationCommDiag}
    \begin{tikzcd}
        \frac{\scD_{X,\hspace{1pt}\fx}[s]\cdot f^s}{\scD_{X,\hspace{1pt}\fx}[s] \cdot (s-\ell) f^s} \arrow[twoheadrightarrow]{r}{\overline{\spe_{s \mapsto \ell}}} \arrow{d}{\nabla_\ell}
            & \scD_{X,\hspace{1pt}\fx}\cdot f^\ell \arrow[hookrightarrow]{d} \\
        \frac{\scD_{X,\hspace{1pt}\fx}[s]\cdot f^s}{\scD_{X,\hspace{1pt}\fx}[s] \cdot (s-(\ell-1)) f^s} \arrow[twoheadrightarrow]{r}{\overline{\spe_{s \mapsto \ell-1}}}
            & \scD_{X,\hspace{1pt}\fx}\cdot f^{\ell - 1}
    \end{tikzcd}
    \end{equation}
    where the leftmost vertical map $\nabla_\ell$ is induced by the $\scD_{X,\hspace{1pt}\fx}$-linear map $\scD_{X,\hspace{1pt}\fx}[s]\cdot f^s \to \scD_{X,\hspace{1pt}\fx}[s]\cdot f^{s+1}$ given by $s \mapsto s+1$. Proposition A \cite{BS23} shows $\spe_{s \mapsto \ell}$ is bijective if and only if $Z(b_f(s)) \cap \{\ell - 1, \ell - 2, \dots \} = \emptyset$. This shows $i) \iff iii)$, after the appropriate language change. As for $i) \implies ii)$, Lemma A and Proposition A of \cite{BS23} along with an inductive application of \eqref{eqn-specializationCommDiag} show $i)$ forces $\scD_{X,\hspace{1pt}\fx}\cdot f^{-1} = \scD_{X,\hspace{1pt}\fx}\cdot f^{-2} = \cdots = \scO_{X,\hspace{1pt}\fx}(*f)$.
    
    Now we prove that $ii) \implies i)$. We can use Proposition A \cite{BS23} to find an integer $j \gg 0$ such that $\overline{\spe_{s \mapsto - j-1}}$ is an isomorphism. By assumption, $\scD_{X,\hspace{1pt}\fx}\cdot f^{- j} = \scO_{X,\hspace{1pt}\fx}(*f)$. Rewriting the square \eqref{eqn-specializationCommDiag} with the map $\overline{\spe_{s \mapsto - j -1}}$ as the lower horizontal map, the rightmost and lowermost maps in the reframed \eqref{eqn-specializationCommDiag} are bijections. This forces the leftmost map $\nabla_{-j}$ to be a surjection, which, by Lemma A \cite{BS23}, forces $\nabla_{-j}$ to be an isomorphism. Whence the uppermost map $\overline{\spe_{s \mapsto -j}}$ is an isomorphism and (Proposition A \cite{BS23}) $Z(b_{f}(s)) \cap \{-j - 1 , -j -2 , \dots \} = \emptyset$. (In sum, the set-up forces all maps to be isomorphisms.) Continuing inductively verifies $ii) \implies i)$. 
\end{proof}

\subsection{Restricting Zeroes of the Bernstein-Sato Polynomial}

Here we consider a nontrivial restriction on $f \in \scO_X$: we consider such $f$ whose Bernstein-Sato polynomial (locally everywhere) only has roots in $(-2,0)$, i.e. $Z(b_f(s)) \subseteq (-2,0)$. (In particular, $\scD_X\cdot f^{-1} = \scO_X(*f)$, cf. Proposition \ref{prop-specialization}.) This assumption is relatively strong, but is integral for our formula for the $V$-filtration on $i_{f,+} \scO_X(*f)$ to hold. Namely, we use the formula appearing as Proposition 3.3 \cite{CNS22}, transcribed as Proposition \ref{prop-funnyBSpolyInducedV} here. On the other hand, the condition frequently occurs in nature: for example, hyperplane arrangements \cite{MSai16}, strongly Koszul free divisors \cite{Nar15}, and the examples in Section \ref{calns}.

The goal of this subsection is to give a $\scD_X[s]$-theoretic description of both $V^0 i_{f,+} \scO_X(*f)$ (Proposition \ref{propcondnBS-new}) and $F_0^H \scO_X(*f)$ (Corollary \ref{cor-zeroHodgePiece}). 

When $Z(b_f(s)) \subseteq (-2,0)$, we can cut $b_f(s)$ into its $(-1,0)$ ``half'' and its $(-2,-1]$ ``half.'' Modulo a $\bC$-involution, we do this now: 

\begin{defn} \label{def-funnyBSpoly}
    Assume that $f \in \scO_{X,\hspace{1pt}\fx}$ satisfies $Z(b_f(s)) \subseteq (-2,0)$. Define
    \begin{equation} \label{eqn-funnyBSpoly1}
        \beta_f(s) = \prod_{\lambda \in Z(b_f(s)) \cap (-1,0)} (s + \lambda + 1)^{\ell_\lambda}
    \end{equation}
    where $\ell_\lambda$ is the multiplicity of the root $\lambda$ in $b_f(s)$. Similarly, define
    \begin{equation} \label{eqn-funnyBSpoly2}
        \beta_f^\prime(s) = \prod_{\lambda \in Z(b_f(s)) \cap (-2,-1]} (s + \lambda + 1)^{\ell_\lambda} .
    \end{equation}
\end{defn}

\begin{rem} \label{rmk-funnyBSpoly} \enspace
    \begin{enumerate}
        \item The definition of $\beta_f(s)$ is technical, the idea intuitive. Take $b_f(s)$ and remove all the roots (and their multiplicities) living in $(-2,-1]$. Exchange $s \mapsto -s - 1$ in the resulting polynomial. This produces \emph{exactly} $\beta_f(s)$. 
        \item Here is another recipe for $\beta_f(s)$: take $b_f(-s-1)$ (whose roots lie in $(-1,1)$); remove all the roots of $b_f(-s-1)$ (and their multiplicities) living in $[0,1)$.
        \item By construction, $\beta_f(s) \beta_f^\prime (s) = b_f(-s-1)$ with $\beta_f(s)$, $\beta_f^\prime(s) \in \bC[s]$ coprime.
    \end{enumerate}
\end{rem}

The following result from \cite{CNS22} is the reason the hypothesis ``$Z(b_f(s)) \subseteq (-2,0)$'' will be so useful to us. It is also why we introduced $\beta_f(s)$. 

\begin{prop}[Proposition 3.3 \cite{CNS22}] \label{prop-funnyBSpolyInducedV}
    Assume that $f \in \scO_{X,\hspace{1pt}\fx}$ satisfies $Z(b_f(s)) \subseteq (-2,0)$. Then, for all $k \in \mathbb{Z}$,
    \begin{equation*}
        V^k i_{f,+} \scO_{X,\hspace{1pt}\fx}(*f) = V_{\ind}^{k+1} \scO_{X,\hspace{1pt}\fx}(*f) + \beta_f(\partial_t t - k) V_{\ind}^k i_{f,+} \scO_{X,\hspace{1pt}\fx}(*f).
    \end{equation*}
    In particular, $V^k i_{f,+} \scO_{X,\hspace{1pt}\fx}(*f) \subseteq V_{\ind}^k i_{f,+} \scO_{X,\hspace{1pt}\fx}(*f)$. 
\end{prop}

\begin{rem}
    The proof goes as follows: restricting to the integer steps of the Kashiwara-Malgrange $V$-filtration $V^{\bullet}i_{f,+}\scO_{X,\hspace{1pt}\fx}(*f)$, the resulting integral $V$-filtration is (by definition) the unique good $V$-filtration such that the Bernstein-Sato polynomial of this good $V$-filtration has all of its roots in the interval $[0,1)$ (see Proposition 4.2-6, Proposition 4.3-5 \cite{MM04}). 

    Then one shows manually that the Bernstein-Sato polynomial of the right hand side in the expression given in the proposition divides
    \[\beta_f(s-1)\beta^{\prime}_f(s).\]
    whose roots lie in the interval $[0,1)$, thus proving the proposition by the uniqueness mentioned above. (If one defines $\beta_f^{\prime}(s) = b_f(-s-1) / \beta_f(s)$, see Remark \ref{rmk-funnyBSpoly}, then the property $Z(\beta_f(s-1) \beta_f^\prime(s)) \subseteq [0,1)$ is equivalent to $Z(b_f(s)) \subseteq (-2,0)$.)
\end{rem}

We will eventually want to use Proposition \ref{prop-funnyBSpolyInducedV} to connect $\scD_X[s]\cdot f^{s-1}$ with the Hodge filtration on $\scO_X(*f)$. We introduce a key player:

\begin{defn} \label{def-Gammaf}
    Let $f \in \scO_{X,\hspace{1pt}\fx}$ such that $Z(b_f(s)) \subseteq (-2,0)$. Define the left $\scD_{X,\hspace{1pt}\fx}[s]$-ideal 
    \begin{equation*}
        \Gamma_f = \scD_{X,\hspace{1pt}\fx}[s] \cdot f + \scD_{X,\hspace{1pt}\fx}[s] \cdot \beta_f(-s) + \ann_{\scD_{X,\hspace{1pt}\fx}[s]} f^{s-1} \subseteq \scD_{X,\hspace{1pt}\fx}[s].
    \end{equation*}
\end{defn}

Membership in $\Gamma_f$ has an functional equation style interpretation:

\begin{prop} \label{prop-funnyBSFunctional}
    Suppose that $f \in \scO_{X,\hspace{1pt}\fx}$ such that $Z(b_f(s)) \subseteq (-2,0)$. Then the following are equivalent:
    \begin{enumerate}[label=\roman*)]
        \item $P(s) \in \Gamma_f$;
        \item $P(s) \in \scD_{X,\hspace{1pt}\fx}[s]$ satisfies the functional equation
        \begin{equation*}
            P(s) \beta_f^\prime(-s) \cdot f^{s-1} \in \scD_{X,\hspace{1pt}\fx}[s]\cdot f^{s}.
        \end{equation*}
    \end{enumerate}
\end{prop}

\begin{proof}
    $i) \implies ii)$. Take $P(s) \in \Gamma_f$. We may write $P(s) = C(s) f + D(s) \beta_f(-s) + E(s)$ for suitable $C(s), D(s) \in \scD_{X,\hspace{1pt}\fx}[s]$ and $E(s) \in \ann_{\scD_{X,\hspace{1pt}\fx}[s]} f^{s-1}$. Then
    \begin{align*}
        P(s) \beta_f^{\prime}(-s) \cdot f^{s-1} 
        &= C(s) \beta_f^{\prime}(-s) \cdot f^{s} + D(s) \beta_f^{\prime}(-s) \beta_f(-s) \cdot f^{s-1} \\ 
        &= C(s) \beta_f^{\prime}(-s) \cdot f^{s} + D(s) b_f(s-1) \cdot f^{s-1} 
        \in \scD_{X,\hspace{1pt}\fx}[s]\cdot f^s
    \end{align*}
    by the definition of $\beta_f(s)$, $\beta_f^\prime(s)$ and $b_f(s)$.

    $ii) \implies i)$: Suppose $P(s) \beta_f^\prime (-s) \cdot f^{s-1} = Q(s) \cdot f^s.$ Since $\beta_f(-s)$ and $\beta_f^\prime(s)$ are coprime we may pick $\alpha(s), \alpha^\prime(s) \in \bC[s]$ so that $\alpha(s) \beta_f(-s) + \alpha^\prime(s) \beta_f^\prime (-s) = 1$. Then
    \begin{align*}
        \alpha^\prime(s) Q(s) \cdot f^s 
        = \alpha^\prime(s) \beta_f^\prime(-s) P(s) \cdot f^{s-1} 
        = \big[ 1 - \alpha(s) \beta_f(-s) \big] P(s) \cdot f^{s-1}.
    \end{align*}
    In particular, $P(s) - \alpha^\prime(s) Q(s) f - \alpha(s) P(s)\beta_f(-s) \in \ann_{\scD_{X,\hspace{1pt}\fx}[s]} f^{s-1}$, verifying that $P(s) \in \Gamma_f$.
\end{proof}

We demonstrate the utility of assuming $Z(b_f(s)) \subseteq (-2,0)$, and consequently of Proposition \ref{prop-funnyBSpolyInducedV}. With no other assumptions, we obtain an algebraic description of $V^0 i_{f,+} \scO_{X,\hspace{1pt}\fx}(*f)$ in terms of $\scD_{X,\hspace{1pt}\fx}[s]\cdot f^{s-1}$ data. We also include some special cases.

\begin{prop} \label{propcondnBS-new}

    Assume that $f \in \scO_{X,\hspace{1pt}\fx}$ satisfies $Z(b_f(s)) \subseteq (-2,0)$. Then
    \begin{enumerate}[label=\roman*)]
        \item Recall (Definition \ref{def-Gammaf}) the $\scD_{X,\hspace{1pt}\fx}[s]$-ideal
            \begin{equation*}
                \Gamma_f:=\scD_{X,\hspace{1pt}\fx}[s]f +\scD_{X,\hspace{1pt}\fx}[s]\beta_f(-s)+ \text{\emph{ann}}_{\scD_{X,\hspace{1pt}\fx}[s]}f^{s-1}\subseteq \scD_{X,\hspace{1pt}\fx}[s]
            \end{equation*}
            and (Lemma \ref{lem-commDiagram}) the $\scD_{X,\hspace{1pt}\fx}$-linear map
            \begin{equation*}
                \pi_f: \scD_{X,\hspace{1pt}\fx}[s] \to i_{f,+}\scO_{X,\hspace{1pt}\fx}(*f) \; ; \; P(s) \mapsto P(-\partial_tt)\cdot(f^{-1}\delta).
            \end{equation*}
            Then
            \begin{equation*}
                \pi_f(\Gamma_f)=V^0i_{f,+}\scO_{X,\hspace{1pt}\fx}(*f).
            \end{equation*}
        \item If $Z(b_f(s)) \subseteq (-2,-1]$ then $\beta_f(s)=1$ and thus, for all $k\in\bZ$,
            \begin{equation*}
                V^ki_{f,+}\scO_{X,\hspace{1pt}\fx}(*f) = V_{\text{\emph{ind}}}^ki_{f,+}\scO_{X,\hspace{1pt}\fx}(*f).
            \end{equation*}
        \item If $Z(b_f(s)) \subseteq (-2,-1]$ then $\Gamma_f=\scD_{X,\hspace{1pt}\fx}[s]$ and thus
            \begin{equation*}
                \pi_f(\scD_{X,\hspace{1pt}\fx}[s])=V^0i_{f,+}\scO_{X,\hspace{1pt}\fx}(*f).
            \end{equation*}
    \end{enumerate}
\end{prop}

\begin{proof}
    For item $ii)$, $Z(b_f(s)) \subseteq (-2,-1] \implies \beta_f(s) = 1$. The rest is Proposition \ref{prop-funnyBSpolyInducedV}. For item $iii)$, $\beta_f(s) = 1 \implies \Gamma_f = \scD_X[s]$. The rest is $i)$.

    It remains to prove $i)$. Lemma \ref{lem-commDiagram} describes the commutative diagram of $\scD_{X,\hspace{1pt}\fx}$-maps
    \begin{equation*}
        \begin{tikzcd}
            \scD_{X,\hspace{1pt}\fx}[s] \arrow[rr, bend left = 20, "\pi_f"] \arrow[r, twoheadrightarrow] & \scD_{X,\hspace{1pt}\fx}[s]\cdot f^{s-1} \arrow[r, hookrightarrow, "\varphi_f"] & i_{f,+}\scO_{X,\hspace{1pt}\fx}(*f).
        \end{tikzcd}
    \end{equation*}
So $\ker\pi_f = \text{ann}_{\scD_{X,\hspace{1pt}\fx}[s]}f^{s-1}$, and
\begin{align*}
\pi_f(\Gamma_f)&=\scD_{X,\hspace{1pt}\fx}[\partial_tt]\cdot \delta + \scD_{X,\hspace{1pt}\fx}[\partial_tt]\beta_f(\partial_tt)\cdot f^{-1}\delta\\
&= V^1_{\text{ind}}i_{f,+}\scO_{X,\hspace{1pt}\fx}(*f)+\beta_f(\partial_tt) V^0_{\text{ind}}i_{f,+}\scO_{X,\hspace{1pt}\fx}(*f) \\
&= V^0i_{f,+}\scO_{X,\hspace{1pt}\fx}(*f) ,
\end{align*}
where the second ``$=$'' is the definition of $V_{\ind}^\bullet$ and the last ``$=$'' is Proposition \ref{prop-MP-graphIsofs}.
\end{proof}

Finally we conclude with our algebraic description of $F_0^H \scO_{X,\hspace{1pt}\fx}(*f)$ promised in Theorem \ref{thm-cor-zeroHodgePiece-intro}. In the \emph{much} more restrictive case of ``strongly Koszul free divisors,'' a subset of Corollary \ref{cor-zeroHodgePiece} appears as Corollary 5.3 \cite{CNS22}. 

\begin{cor} \label{cor-zeroHodgePiece}
Assume that $f \in \scO_{X,\hspace{1pt}\fx}$ satisfies $Z(b_f(s)) \subseteq (-2,0)$. Then
\begin{align*}
    F_0^H\scO_{X,\hspace{1pt}\fx}(*f) 
    =(\Gamma_f\cap\scO_{X,\hspace{1pt}\fx})f^{-1}.
\end{align*}
That is, $F_0^H \scO_{X,\hspace{1pt}\fx}(*f)$ is the $\scO_{X,\hspace{1pt}\fx}$-submodule generated by $\{g f^{-1}\}$ where $g$ ranges over all elements of $\scO_{X,\hspace{1pt}\fx}$ satisfying the ``functional equation''
\begin{equation*}
    g \beta_f^\prime(-s) f^{s-1} \in \scD_{X,\hspace{1pt}\fx}[s]\cdot f^{s}.
\end{equation*}
\end{cor}

\begin{proof}

We prove the first claimed ``$=$'', as the ``that is'' claim is Proposition \ref{prop-funnyBSFunctional}. Note: 
\begin{equation} \label{eqn-corZeroHodgePiece1}
F_0^H\scO_{X,\hspace{1pt}\fx}(*f) 
\simeq V^0i_{f,+}\scO_{X,\hspace{1pt}\fx}(*f)\cap\scO_{X,\hspace{1pt}\fx}(*f)\delta
=\pi_f(\Gamma_f)\cap\scO_{X,\hspace{1pt}\fx}(*f)\delta,
\end{equation}
where the ``$\simeq$'' is Theorem \ref{thmformula}, the ``$=$'' Proposition \ref{propcondnBS-new}.

So if $u\in F_0^H\scO_{X,\hspace{1pt}\fx}(*f)$, we may write $u\delta=\pi_f(P)$ with $P\in \Gamma_f$. Since $F_0^H \scO_{X,\hspace{1pt}\fx}(*f) \subseteq \scO_{X,\hspace{1pt}\fx} \cdot f^{-1}$ (cf. \cite{MSai93}), we know that $uf \in \scO_{X,\hspace{1pt}\fx}$ and that $P - uf \in \scD_{X,\hspace{1pt}\fx}[s]$. Then $\pi_f(P-uf)=0$, implying that $uf \in P+\ker\pi_f \subseteq \Gamma_f$, since $\ker\pi_f=\ann_{\scD_{X,\hspace{1pt}\fx}[s]}f^{s-1}\subseteq\Gamma_f$ as seen in the proof of Proposition \ref{propcondnBS-new}.i). Also, $u = (uf)f^{-1} \in (\Gamma_f \cap \scO_{X,\hspace{1pt}\fx}) f^{-1}$, proving that $F_0^H \scO_{X,\hspace{1pt}\fx}(*f) \subseteq (\Gamma_f \cap \scO_{X,\hspace{1pt}\fx}) f^{-1}$ as required.

Conversely, take $u \in \Gamma_f \cap \scO_{X,\hspace{1pt}\fx}$. Then $\pi_f(u) = u f^{-1} \delta$, implying that $u f^{-1} \in F_0^H \scO_{X,\hspace{1pt}\fx}(*f)$, cf. \eqref{eqn-corZeroHodgePiece1}.
\end{proof}

\begin{eg}   
We record a curiosity. Let $\mathscr{A} \subseteq \bC^n$ define a (central, indecomposable, and irreducible) hyperplane arrangement defined by reduced polynomial $f \in R = \bC[x_1, \dots, x_n]$. Let $\mathscr{L}(\mathscr{A})$ denote the intersection lattice of $\mathscr{A}$. For each flat $F \in \mathscr{L}(\mathscr{A})$, let $m_F$ denote the number of hyperplanes containing $F$, let $\rank(F)$ denote the rank (codimension of $F \subseteq \mathbb{A}^n$) of $F$, and let $I_F \subseteq R$ be the defining ideal of $F$. 

By a formula of Mustaţă \cite{Mustata06} (see also Teitler's \cite{Teitler08}) we have that
    \begin{equation} \label{eqn-MustataArrangementMultiplier}
        F_0^H \scO_{\mathbb{A}^n}(* \mathscr{A}) = \bigcap_{F \in \mathscr{L}(\mathscr{A})} I_F^{m_F - \rank(F)} \cdot \scO_{\mathbb{A}^n} (\mathscr{A}).
    \end{equation}
\end{eg}
On the other hand, it is well known that the global Bernstein-Sato polynomial of $f$ equals the local one at zero. Moreover, $Z(b_f(s)) \subseteq (-2,0)$, cf. \cite{MSai16}. Corollary \ref{cor-zeroHodgePiece} then implies
\begin{equation} \label{eqn-OurArrangementMultiplierFml}
    F_0^H \scO_{\mathbb{A}^n}(* \mathscr{A}) = \{g \in \scO_{\mathbb{A}^n} \mid g \beta_f^\prime (-s) \in f^{s-1} \in \scD_X[s]\cdot f^{s} \} \cdot \scO_{\mathbb{A}^n}(\mathscr{A}).
\end{equation}

We deduce the RHS's of \eqref{eqn-MustataArrangementMultiplier} and \eqref{eqn-OurArrangementMultiplierFml} agree, giving a novel formula for the multiplier ideals of $\mathscr{A}$. We have no idea how to prove the RHS's of \eqref{eqn-MustataArrangementMultiplier} and \eqref{eqn-OurArrangementMultiplierFml} agree directly. Note that computing $b_f(s)$ or $\beta_f^\prime(s)$ is currently intractable for arbitrary arrangements.

\vspace{10pt}

We finish by proving Theorem \ref{thm-rf-intro} from the Introduction, now significantly generalizing Corollary 4.8 \cite{CNS22}. Our result describes what shift $F_{k - \ell}^{\ord} \scO_{X,\hspace{1pt}\fx}(*f)$ of the induced order filtration will be contained in the Hodge filtration $F_k^H \scO_{X,\hspace{1pt}\fx}(*f)$. Equivalently, we find a bound for the first step of the Hodge filtration that contains $f^{-1}$.

\begin{cor} \label{corrf}

Assume that $f \in \scO_{X,\hspace{1pt}\fx}$ satisfies $Z(b_f(s)) \subseteq (-2,0)$. Write $r_f$ for the degree of the polynomial $\beta_f(s)\in\bC[s]$, cf. Definition \ref{def-funnyBSpoly}. Then, for all $k \in \mathbb{Z}$, 
\[F_{k-r_f}^{\text{\emph{ord}}}\scO_{X,\hspace{1pt}\fx}(*f)\subseteq F_k^H\scO_{X,\hspace{1pt}\fx}(*f).\]

\end{cor}

\begin{proof}

Recall from Proposition \ref{propcondnBS-new} that our assumption $Z(b_f(s))\subseteq (-2,0)$ implies the equality
\[\pi_f(\Gamma_f)=V^0i_{f,+}\scO_{X,\hspace{1pt}\fx}(*f).\]
Moreover, the definition of $\pi_f$ and of the module action on $\scO_{X,\hspace{1pt}\fx}(*f)[\partial_t]\delta$ implies that we have an inclusion
\[\pi_f(F_k^{\sharp}\scD_{X,\hspace{1pt}\fx}[s])\subseteq F_k^{\Torder}i_{f,+}\scO_{X,\hspace{1pt}\fx}(*f).\]
Also, see that Theorem \ref{thmformula} may be reformulated as saying (we work locally here for consistency) that 
\begin{equation} \label{eqn-reformulation-0}
    F_k^H\scO_{X,\hspace{1pt}\fx}(*f) = \psi_{f,0}(F_k^{\Torder}V^0i_{f,+}\scO_{X,\hspace{1pt}\fx}(*f)).
\end{equation}
Indeed, Theorem \ref{thmformula} says that $u_0 \in F_k^H \scO_{X,\hspace{1pt}\fx}(*f)$ if and only if there exists $u_1, \ldots , u_q \in \scO_{X,\hspace{1pt}\fx}(*f)$ such that
\begin{equation} \label{eqn-reformulation-help}
    V^0 i_{f,+} \scO_{X,\hspace{1pt}\fx}(*f) \ni u_0\delta + \underbrace{u_1 \partial_t \delta + \cdots + u_q \partial_t^q \delta}_{\in \enspace \partial_t F_{k-1}^{\Torder} i_{f,+} \scO_{X,\hspace{1pt}\fx}(*f)}.
\end{equation}
In particular, $q \leq k$. Now we check \eqref{eqn-reformulation-0}. If $u_0 \in F_k^H \scO_{X,\hspace{1pt}\fx}(*f)$, we find $u_1, \ldots, u_k \in \scO_{X,\hspace{1pt}\fx}(*f)$ such that $u_0 \delta + u_1 \partial_t \delta + \cdots u_k \partial_t^k \delta$ is in the form of \eqref{eqn-reformulation-help}. And $\psi_{f,0}(u_0 \delta + u_1 \partial_t \delta + \cdots + u_k \partial_t^k \delta) = u_0$. On the other hand, if $u \in F_k^{\Torder} V^0 i_{f,+} \scO_{X,\hspace{1pt}\fx}(*f)$, then $u = u_0 \delta + u_1 \partial_t \delta + \cdots + u_k \partial_t^k \delta$ is trivially of the form \eqref{eqn-reformulation-help}, meaning $u_0 \in F_k^H \scO_{X,\hspace{1pt}\fx}(*f)$. And $\psi_{f,0}(u) = u_0 $.

Combining the first three display equations and applying the map $\psi_{f,0}$ yields:
\begin{align*}
\phi_{f,0}(\Gamma_f\cap F_k^{\sharp}\scD_{X,\hspace{1pt}\fx}[s])\cdot f^{-1} &= (\psi_{f,0}\circ\pi_f)(\Gamma_f\cap F_k^{\sharp}\scD_{X,\hspace{1pt}\fx}[s]) \\ &\subseteq \psi_{f,0}(F_k^{\Torder}V^0i_{f,+}\scO_{X,\hspace{1pt}\fx}(*f)) \\ &= F_k^H\scO_{X,\hspace{1pt}\fx}(*f).
\end{align*}
By definition $\beta_f(s)\in\Gamma_f\cap F_{r_f}^{\sharp}\scD_{X,\hspace{1pt}\fx}[s]$, and thus
\[f^{-1}\in \phi_{f,0}\left(\frac{\beta_f(s)}{\beta_f(0)}\right)\cdot f^{-1} \subseteq F_{r_f}^H\scO_{X,\hspace{1pt}\fx}(*f),\]
which implies in turn that
\[F_k^{\ord}\scO_{X,\hspace{1pt}\fx}(*f) = F_k^{\ord}\scD_{X,\hspace{1pt}\fx}\cdot f^{-1}\subseteq F_k^{\ord}\scD_{X,\hspace{1pt}\fx}\cdot F_{r_f}^H\scO_{X,\hspace{1pt}\fx}(*f) \subseteq F_{k+r_f}^H\scO_{X,\hspace{1pt}\fx}(*f)\]
for all $k \in\bZ$ as required.

\end{proof}

\section{Main Results}

Here we prove our main theorems Theorem \ref{thmmain-new} and Theorem \ref{thmgamma2-new} on $F_\bullet^H \scO_{X,\hspace{1pt}\fx}(*f)$. Under their assumptions, the first strengthens Saito's \cite{MSai93} result that the Hodge filtration is contained in the pole order filtration, by showing it is contained in the finer induced order filtration. The second gives a simple algebraic formula for \emph{all} pieces $F_k^H \scO_{X,\hspace{1pt}\fx}(*f)$. The formula is a natural generalization of the expression for $F_0^H \scO_{X,\hspace{1pt}\fx}(*f)$ from Corollary \ref{cor-zeroHodgePiece}.

Let us roughly describe the path the proofs chart. We continue our approach of understanding the Hodge filtration on $\scO_X(*f)$ via the interactions between the induced order and induced $V$-filtration on $i_{f,+} \scO_X(*f)$. Under the assumptions of Theorem \ref{thmmain-new} and Theorem \ref{thmgamma2-new} we will have explicit presentations of $\scO_X(*f)$ and $i_{f,+} \scO_X(*f)$, thanks to:

\begin{prop} \label{prop-presentations}
    Assume that there exists $E \in \Der_X$ such that $E \cdot f = f$ and that $\scD_X\cdot f^{-1} = \scO_X(*f)$. Then
    \begin{equation} \label{eqn-presentationComplement}
        \frac{\scD_X}{\ann_{\scD_X}f^{s-1} + \scD_X\cdot (E+1)} = \frac{\scD_X}{\ann_{\scD_X} f^{-1}} \simeq \scD_X\cdot f^{-1} \xhookrightarrow{=} \scO_X(*f)
    \end{equation}
    Moreover, 
    \begin{align} \label{eqn-presentationGraphEmbed}
        i_{f,+} \scO_X(*f) \simeq \scO_X(*f)[\partial_t]\delta 
        &= \scD_{X \times \bC_t} \cdot f^{-1} \delta \\
        &\simeq \frac{\scD_{X \times \bC_t}}{\scD_{X \times \bC_t} \cdot (t-f, E + f \partial_t + 1) + \scD_{X \times \bC_t} \cdot \ann_{\scD_{X}} f^{s-1}}. \nonumber
    \end{align}
\end{prop}

\begin{proof}
    The only thing to prove in \eqref{eqn-presentationComplement} is the first equality. Because $E - (s-1) \in \ann_{\scD_{X}[s]} f^{s-1}$, one checks that $\ann_{\scD_{X}[s]} f^{s-1} = \scD_X[s]\cdot \ann_{\scD_X} f^{s-1} + \scD_X[s]\cdot (E - (s-1))$. By Proposition \ref{prop-specialization} we know the $\scD_X$-linear specialization map $\specialize : \scD_X[s]\cdot f^{s-1} / \scD_{X}[s] \cdot s f^s \to \scD_X\cdot f^{-1}$ is an isomorphism. This validates \eqref{eqn-presentationComplement}.

    As for \eqref{eqn-presentationGraphEmbed}, since $\scD_X\cdot f^{-1} = \scO_X(*f)$, by applying the graph embedding and using \eqref{eqn-presentationComplement} just as in Lemma 3.5 \cite{CNS22}, we have
    \begin{align*}
        i_{f,+} \scO_X(*f) = i_{f,+} \scD_X\cdot f^{-1} 
        &\simeq i_{f,+} \left( \frac{\scD_X}{\ann_{\scD_X}f^{s-1} + \scD_X \cdot( E + 1)} \right) \\
        &= \frac{\scD_{X \times \bC_t}}{\scD_{X \times \bC_t} \cdot (t-f, E + f \partial_t + 1) + \scD_{X \times \bC_t} \cdot \ann_{\scD_X} f^{s-1}}.
    \end{align*}
    The validity of \eqref{eqn-presentationGraphEmbed} follows.  
\end{proof}

We will use the explicit presentations of $\scO_X(*f)$ and $i_{f,+}\scO_X(*f)$ to study the induced order and $V$-filtrations, and, in turn, the Hodge filtration. This requires understanding generating sets of the $\scD_X$-ideal (resp. $\scD_{X \times \bC_t}$-ideal) whose cokernel gives $\scO_X(*f)$ (resp. $i_{f,+} \scO_X(*f)$) and the behavior of these generating sets with respect to the $\scD_X$-order filtration (resp. $\scD_{X \times \bC_t}$-order filtration and $\scD_{X \times \bC_t}-V$-filtration along $\{t=0\}$). In other words: it is helpful to have Gr{\"o}bner basis type properties for these ideals. We will see that the full package of hypotheses assumed in Theorem \ref{thmmain-new} and Theorem \ref{thmgamma2-new} allow us to make Gr{\"o}bner inspired arguments. With this in hand we will be able to strengthen the results from the previous section.

The first subsection establishes some general algebraic facts used to set up the Gr{\"o}bner basis type variant we require. The second subsection introduces the hypotheses on $f \in \scO_X$ our main theorems demand, including the titular \emph{parametrically prime} condition. The hypotheses include our earlier requirement that $Z(b_f(s)) \subseteq (-2,0)$. The third subsection proves Theorem \ref{thmmain-new}; the fourth proves Theorem \ref{thmgamma2-new}. The last subsection collects stronger results when, additionally, $Z(b_{f}(s)) \subseteq (-2,-1]$.

\subsection{Gr{\"o}bner Type Constructions and Lemmas}

In this subsection we first study the relationship between ideals in a filtered ring $A$ and the ideal of principal symbols in the associated graded ring $\gr A$. The prize of our labor is the crucial Lemma \ref{lem-GRespNZDs}. We finish with the useful Lemma \ref{lem-IndVRespBasisonGrOrd} which lets us understand, in a special case, the behavior of the induced $V^\bullet$ filtration on $\gr^{\ord} \scD_{X \times \bC_t}$. First, generalities. 

We define an \emph{integral increasing filtration} $G_\bullet A$ on a ring $A$ to be a $\mathbb{Z}$-indexed set of subgroups $G_k A$ of $A$ such that: $G_k A \subseteq G_{k+1} A$; $(G_k A) \cdot (G_j A) \subseteq G_{k+j} A$. We say the filtration is \emph{exhaustive} when $\cup_{k \in \mathbb{Z}} G_k A = A$, we say the filtration is \emph{separated} when $\cap_{j \in\bZ} G_j A = 0$.

For integral, exhaustive, separated, and increasing filtrations $G_\bullet A$, we define the \emph{order} of $a \in A \setminus 0$ with respect to $G_\bullet A$ by $\ord_G (a) = \min\{ k \in \mathbb{Z} \mid a \in G_k A \}$. We set $\ord_G(0) = -\infty$. (Note that exhaustive and separated ensures $\ord_G : A \setminus 0 \to \mathbb{Z}$ is well-defined.) When $G_\bullet A$ is increasing we have the associated graded ring $\gr_G A = \bigoplus_{k \in \mathbb{Z}} G_k A / G_{k-1} A$ which is naturally $\mathbb{Z}$-graded. The $\emph{initial symbol}$ of $a \in A$ will be denoted by $\gr_G (a) \in \gr_G A$ which, for $a \neq 0$, is defined by $\gr_G(a) = \overline{a} \in G_{\ord_G(a)} A / G_{\ord_G(a) - 1} A \subseteq \gr_G A$. (We always set $\gr_G(0) = 0 \in \gr_G A.$) For $I$ a left $A$-ideal, we call $\gr_G (I) := \gr_G A \cdot (\{ \gr_G(a) \mid a \in I \})$ the \emph{ideal of principal symbols of $I$}.

\begin{rem}

An \emph{integral decreasing fltration} $G^{\bullet}A$ is a sequence of subgroups such that $G_{\bullet}^+A:=G^{-\bullet}A$ is an integral increasing filtration. From here on out, whenever we say that a decreasing filtration $G^{\bullet}A$ has a certain property, we mean that the corresponding increasing filtration $G_{\bullet}^+A$ has said property. When $G^{\bullet}A$ is an integral, exhaustive, separated and decreasing filtration, we write $\ord^G(a):=\text{max}\{k\in\bZ\mid a\in G^kA\}$ (with $\ord^G(0)=\infty$) and $\gr^GA = \bigoplus_{k\in\bZ}G^kA/G^{k+1}A$ for the associated graded ring, and $\gr^G(a)\in G^{\ord^G(a)}A/G^{\ord^G(a)+1}A\subseteq \gr^GA$ for the initial symbol of $a\in A$.
    
\end{rem}

\begin{conv}
For readability's sake, in this subsection we state most definitions and results for increasing filtrations. We convene that all definitions and results also hold for decreasing filtrations even when only stated for increasing ones, since symmetric versions hold due to the natural correspondence between increasing and decreasing filtrations. For example, Lemma \ref{lem-IndVRespBasisonGrOrd} uses the decreasing version of Definition \ref{def-GRespBasis}.
\end{conv}

Now suppose $a_1, \ldots, a_r$ generate $I \subseteq A$. In general, the gulf between the (left) ideal of principal symbols $\gr_G(I)$ and the ideal $\gr_G A \cdot (\gr_G(a_1), \ldots \gr_G(a_r))$ generated by the principal symbols of the $\{a_k\}_{1 \leq k \leq r}$ is vast. So it will be useful for us to restrict to generating sets of our ideals that play well with our filtrations. Thus we make two definitions:

\begin{defn} \label{def-GGrobnerBasis}
    Let $A$ be a ring and $G_\bullet A$ an integral, exhaustive, separated, and increasing filtration such that $\gr_G A$ is a commutative domain. Let $I \subseteq A$ be a left ideal. We say that $a_1 \ldots, a_r \in A$ is a \emph{Gr{\"o}bner basis of $I$ with respect to $
    G_\bullet A$} provided that:
    \begin{enumerate}[label=\roman*)]
        \item $I = A \cdot (a_1, \ldots , a_r)$;
        \item $\gr_G(I) = (\gr_G(a_1) , \ldots, \gr_G(a_r))$.
    \end{enumerate}
\end{defn}

\begin{defn} \label{def-GRespBasis}
    Let $A$ be a ring and $G_\bullet A$ an integral, exhaustive, separated, and increasing filtration such that $\gr_G A$ is a commutative domain. For $I \subseteq A$ a left ideal, we say that $a_1, \ldots, a_r \in A$ form a \emph{$G_\bullet A$-respecting basis} for $I$ if $I = A \cdot (a_1, \ldots, a_r)$ and if for all $b \in I$ we can find a \emph{$G_\bullet A$-standard representation} of $b$; that is, $k_1, \ldots, k_r \in A$ satisfying:
    \begin{enumerate}[label=\roman*)]
        \item for each $1 \leq i \leq r$ we have either $\ord_G (k_i) + \ord_G (a_i) \leq \ord_G (b)$ or $k_i= 0$;
        \item $b = \sum_i k_i a_i$.
    \end{enumerate}
\end{defn}

In general, every $G_\bullet A$-respecting basis is a Gr{\"o}bner basis:

\begin{lem} \label{lem-GRespImpliesGrobner}
    Let $A$ be a ring and $G_\bullet A$ an integral, exhaustive, separated, and increasing filtration. If $a_1, \ldots, a_r$ is $G_\bullet A$-respecting basis, then $a_1, \ldots, a_r$ is a Gr{\"o}bner basis of $A \cdot (a_1, \ldots, a_r)$ with respect to $G_\bullet A$.
\end{lem}

\begin{proof}
    Let $b\in A\cdot(a_1,\dots,a_r)$. Then, by hypothesis, we may choose a $G_{\bullet}A$-standard representation $b=\sum_{i=1}^rk_ia_i$. In particular, $b,k_1a_1,\dots,k_ra_r\in G_{\ord_G(b)}$, and thus
    \[\gr_Gb=\sum_{i=1}^r\gr_G(k_ia_i)=\sum_{i\in\Upsilon}\gr_G(k_i)\gr_G(a_i)\in(\gr_G(a_1),\dots,\gr_G(a_r)),\]
    where $\Upsilon:=\{i\in\{1,\dots,r\}\mid \ord_G(k_i)+\ord_G(a_i)=\ord_G(b)\}.$ 
\end{proof}

We say an integral increasing filtration $G_{\bullet}A$ is \emph{bounded} when $G_k A = 0$ for $k \ll 0$. When bounded, there is no difference between a $G_\bullet A$-Gr{\"o}bner basis and $G_\bullet A$-respecting basis:

\begin{lem} \label{lem-BoundedFiltration}
    Let $A$ be a ring and $G_\bullet A$ an integral, exhaustive, separated, and increasing filtration such that $\gr_G A$ is a commutative domain. Assume that $G_\bullet A$ is bounded. Then $a_1, \ldots , a_r$ is a Gr{\"o}bner basis for the left $A$-ideal $I = A\cdot (a_1 \ldots, a_r)$ if and only if $a_1, \ldots, a_r$ is a $G_\bullet A$-respecting basis. 
\end{lem}

\begin{proof}
    Thanks to Lemma \ref{lem-GRespImpliesGrobner}, all we must show is that if $a_1, \ldots, a_r$ is a Gr{\"o}bner basis for the left $A$-ideal $I$, then it is also a $G_\bullet A$-respecting basis. So select $b_0 \in I$. By the Gr{\"o}bner property, there exists $c_{1,1}, \ldots, c_{1,r} \in \gr_G A$ such that $c_{1,k}$ is homogeneous of degree $\ord_G(b)-\ord_G(a_k)$ (or is zero) and such that $\gr_G(b_0) = \sum c_{1,k} \gr_G (a_k)$. Select lifts $C_{1,k} \in A$ of each $c_{1,k}$, i.e. $C_{1,k} \in A$ satisfies $\gr_G(C_{1,k}) = c_{1,k}$. So $\ord_G(C_{i,k}) + \ord_G(a_k) \leq \ord_G(b)$. Moreover,
    \begin{equation*}
        b_1 := b_0 - \sum_{k = 1}^r C_{1,k} a_k \in G_{\ord_G(b_1)} A
    \end{equation*}
    by the choice of the $c_{1,k}$. Moreover, $b_1 \in I$ and $\ord_G(b_1) < \ord_G(b_0)$. We may repeat this procedure on $b_1$, obtaining $C_{2,1}, \ldots, C_{2,r} \in A$ such that $b_2 := b_1 - \sum_{k=1}^r C_{2,k} a_{k} \in G_{\ord_G(b_1) - 1} A$, $\ord_G(C_{2,k}) + \ord_G(a_k) \leq \ord_G(b_1)$ and $b_2 \in I$. Repeating, eventually we find $b_m = 0$ because $G_\bullet A$ is bounded. Unraveling the equations gives 
    \begin{align} \label{eqn-BoundedFiltration-1}
        b = b_1 + \sum_{k=1}^{r} C_{1,k} a_k 
        &= b_2 + \bigg[\sum_{k=1}^r C_{2,k} a_k \bigg] + \bigg[ \sum_{k=1}^{r} C_{1,k} a_k \bigg] \\
        &= \bigg[\sum_{k=1}^r C_{2,k} a_k \bigg] + \cdots + \bigg[ \sum_{k=1}^r C_{m-1,k} a_k \bigg]
        = \sum_{k=1}^r \bigg( \sum_{j = 1}^{m-1} C_{j,k} \bigg) a_k . \nonumber
    \end{align}
    Since $\ord_G (C_{j,k}) \leq \ord_G(b_{j-1}) - \ord_G(a_k) \leq \ord_G(b_0) - \ord_G(a_k)$ for all $k$, we see that $\sum_{j=1}^{m-1} C_{j,k} \in G_{\ord_G(b_0) - \ord_G(a_k)} A$. That is, $\ord_G ( \sum_{j=1}^{m-1} C_{j,k}) + \ord_G(a_k) \leq \ord_G(b_0)$ and \eqref{eqn-BoundedFiltration-1} exhibits a $G_\bullet A$-standard representation of $b_0$ with respect to $a_1, \ldots, a_r \in A$. 
\end{proof}

\begin{rem} \label{rmk-FiltrationsGen} \enspace
    \begin{enumerate}
        \item It is not immediately obvious that Lemma \ref{lem-BoundedFiltration} is true if the bounded hypothesis is removed. It is true in certain situations (cf. Theorem 10.6 \cite{OT01}), but as we do not need a stronger version of Lemma \ref{lem-BoundedFiltration} we do not attempt to establish one.
        \item Lemma \ref{lem-BoundedFiltration} exploits the well-behavedness of standard representations with respect to addition. Namely, if $b = \sum k_j a_j$ and $\widetilde{b} = \sum \widetilde{k_j} a_j$ are $G_\bullet A$ standard representations with respect to $a_1, \dots, a_r$, then $b + \widetilde{b} = \sum (k_j + \widetilde{k_j}) a_j$ is a $G_\bullet A$ standard representation of $b + \widetilde{b}$.
    \end{enumerate}
\end{rem}

\begin{lem} \label{lem-GRespNZDs}
    Let $A$ be a ring and $G_\bullet A$ an integral, exhaustive, separated, and increasing filtration with a $G_\bullet A$-respecting basis $a_1, \ldots, a_r$. Assume that $a_{r+1} \in A$ is such that $\gr_G(a_{r+1})$ is a non-zero divisor on $\gr_{G} A / \gr_G(A\cdot(a_1, \ldots, a_r))$. Then $a_1, \ldots, a_r, a_{r+1}$ is a $G_\bullet A$-respecting basis.
\end{lem}

\begin{proof}
    Take $b \in A \cdot (a_1, \ldots, a_r, a_{r+1})$, and write $b = \sum_{i=1}^{r+1} b_i a_i$. We show that $b$ has a $G_{\bullet}A$-standard representation with respect to $a_1,\ldots,a_{r+1}$.
    
    \emph{Case 1:} $\ord_{G}(b_{r+1}) + \ord_{G}(a_{r+1}) \leq \ord_G(b)$ or $b_{r+1} = 0$. 

    \noindent Set $b^{\prime} = b - b_{r+1} a_{r+1}$. If $b_{r+1} \neq 0$, then $b^{\prime} \in G_{\leq \ord_G(b)}$. Since $b^{\prime} \in A \cdot (a_1, \ldots, a_r)$, it has a $G_\bullet A$-standard representation with respect to $a_1, \ldots, a_r$ and we are done. If $b_{r+1} = 0$, we are again done since $b = b^{\prime}$.

    \emph{Case 2:} $\ord_G(b_{r+1}) + \ord_G(a_{r+1}) > \ord_G(b)$ and $b_{r+1} \neq 0$. 

    \noindent We will construct a new expression $b = \sum_{i = 1}^{r+1} \widetilde{b_i} a_{i}$ where $\ord_G (\widetilde{b_{r+1}}) < \ord_G (b_{r+1})$ or $\widetilde{b_{r+1}} = 0$. This will finish the proof: by proceeding inductively we will arrive at \emph{Case 1}. 

    Set $b^{\prime} = b - b_{r+1} a_{r+1}$. Since $\ord_G(b) < \ord_G(b_{r+1} a_{r+1}) = \ord_G(b_{r+1}) + \ord_G(a_{r+1})$, we deduce that $\gr_G(b^{\prime}) = -\gr_G(b_{r+1} a_{r+1}) = -\gr_G(b_{r+1}) \cdot \gr_G(a_{r+1})$. By Lemma \ref{lem-GRespImpliesGrobner}, $a_1, \dots, a_r$ is a Gr{\"o}bner basis for $A \cdot (a_1, \ldots, a_r)$ and so
    \begin{equation*}
        0 = \overline{\gr_G(b_{r+1})} \cdot \overline{\gr_G(a_{r+1})} = -\overline{\gr_G(b^{\prime})} \in \frac{\gr_G A}{\gr_G( A \cdot (a_1, \ldots, a_r))} = \frac{\gr_G A}{(\gr_G(a_1), \ldots, \gr_G(a_r))},
    \end{equation*}

    Since $\gr_G(a_{r+1})$ is a non-zero divisor on this module (and in particular $\overline{\gr_G(a_{r+1})} \neq 0$), we deduce that $\gr_G(b_{r+1}) \in \gr_G(A \cdot (a_1, \ldots, a_r))$. Using that $a_1, \ldots, a_r$ is a $G_{\bullet}A$-respecting basis, we find $k_1, \ldots, k_r \in A$ such that $\ord_G(k_i)+\ord_G(a_i)\leq\ord_G(b_{r+1})$ for all $i$ and such that 
    \begin{equation} \label{eqn-RespBasis-1}
        \gr_G(b_{r+1}) = \gr_G\left( \sum_{i = 1}^{r} k_i a_i \right).
    \end{equation}
    Now we rearrange:
    \begin{align} \label{eqn-RespBasis-2}
        b 
        = \sum_{i = 1}^{r+1} b_i a_i 
        &= \sum_{i = 1}^{r} b_i a_i  + \bigg[ (b_{r+1} - \sum_{i = 1}^{r} k_i a_i) + \sum_{i=1}^r k_i a_i \bigg] a_{r+1} \\
        &= \bigg[ \sum_{i=1}^{r} (b_i + k_i a_{r+1}) a_i \bigg] + \bigg[ b_{r+1} - \sum_{i=0}^r k_i a_i \bigg] a_{r+1}. \nonumber
    \end{align}
    By \eqref{eqn-RespBasis-1}, we know that $b_{r+1} = \sum_{i=1}^r k_i a_i$ or $\ord_G(b_{r+1} - \sum_{i=1}^r k_i a_i) < \ord_G(b_{r+1})$. So either the expression \eqref{eqn-RespBasis-2} for $b$ has an $a_{r+1}$ coefficient equal to zero or it has an $a_{r+1}$ coefficient with strictly smaller $G_\bullet A$-order than $\ord_G(b_{r+1})$. So we have constructed our desired expression for $b$; iterating the procedure finishes the proof.
\end{proof}

The second technical lemma returns to the realm of the paper at large:

\begin{conv}
    When working with $F_{\bullet}^{\ord} \scD_X$ we will often write $\gr^{\ord} \scD_X$ for the associated graded object and $\gr^{\ord} : \scD_X \to \gr^{\ord} \scD_X$ for the principal symbol map. Similarly for $F_\bullet^{\ord} \scD_{X \times \bC_t}$.
\end{conv}

$V^{\bullet} \scD_{X \times \bC_t}$ is the $V$-filtration on $\scD_{X \times \bC_t}$ along $\{t=0\}$, cf. Definition \ref{def-inducedVfiltration}. One checks that $V^{\bullet}$ induces an exhaustive, separated, and decreasing filtration on $\gr^{\ord} \scD_{X \times \bC_t}$, cf. the prelude to Lemma 3.7 \cite{CNS22}.

\begin{defn} \label{defn-widetildeVfiltration}
    We define $\widetilde{V}^{\bullet} \gr^{\ord} \scD_{X \times \bC_t}$ to be the filtration described above, i.e. the one induced by $V^{\bullet}$ on $\gr^{\ord} \scD_{X \times \bC_t}$. 
\end{defn}

Recall (see the first display equation after Lemma 3.6 of \cite{CNS22}), in local coordinates, we have
\begin{equation} \label{eqn - incudedVonGrexplicit}
\widetilde{V}^k \gr^{\ord} \scD_{X \times \bC_t, (\fx, 0)}  \simeq \left\{ \sum_q \sum_I \bigg( \sum_{p = 0}^\infty g_{I,q,p} (x) t^p \bigg) \xi^I T^q \mid g_{I,q,p}(x) \neq 0 \implies p-q \geq k \right\}
\end{equation}
where $\sum_{p=0}^\infty g_{I,p,q}(x) t^p \in \scO_{X \times \bC_t, (\fx, 0)}$ and $g_{I,p,q}(x) \in \scO_{X,\hspace{1pt}\fx}$ and $\xi^I$ is standard multi-index notation. Note that $\sum_q$ and $\sum_I$ are finite sums.

We establish a useful criterion for constructing certain $\widetilde{V}^\bullet \gr^{\ord} \scD_{X \times \bC_t, (\fx, 0)}$-respecting bases:

\begin{lem} \label{lem-IndVRespBasisonGrOrd}
    Suppose that $A_1, \ldots, A_m \in \gr^{\ord} \scD_{X,\hspace{1pt}\fx}$. Then $A_1, \ldots, A_m \in \gr^{\ord} \scD_{X \times \bC_t, (\fx, 0)}$ is a $\widetilde{V}^\bullet \gr^{\ord} \scD_{X \times \bC_t, (\fx, 0)}$-respecting basis.
\end{lem}

\begin{proof}
    Take $B \in \gr^{\ord} \scD_{X \times \bC_t, (\fx, 0)} \cdot (A_1, \ldots, A_m)$. Since $T$ is a new indeterminant in a polynomial ring extension, we may find $Q \in \mathbb{Z}_{\geq 0}$ such that
    \begin{align*}
        B = \sum_{q = 0}^Q  \bigg[ \underbrace{ \sum_\ell \bigg[ \sum_I \bigg( \sum_{p=0}^\infty g_{I,q,p}^\ell(x) t^p \bigg) \xi_I \bigg] A_\ell }_{:= M_q} \bigg]  T^q
    \end{align*}
    Fix $0 \leq q \leq Q$ and assume that $M_q \neq 0$. Let $\tau_q$ be the largest power such that $M_q \in t^{\tau_q} \cdot \scO_{X \times \bC_t, (\fx,0)}[\zeta_1, \dots, \zeta_n]$. By definition of $\widetilde{V}^\bullet \scO_{X \times \bC_t, (\fx, 0)}$, we must have $\tau_q - q \leq \ord^{\widetilde{V}}(B)$. By reducing modulo $t^{\tau_q}$ we see that
    \begin{equation*}
        M_{q, < \tau_q} = \sum_{p=0}^{\tau_q - 1} \bigg[ \sum_\ell  \bigg( \sum_I  g_{i,q,p}^\ell(x)  \xi_I \bigg) A_\ell \bigg] t^p \in \bigoplus_{v = 0}^{\tau_q - 1} \scO_{X,\hspace{1pt}\fx} [\xi_1, \dots, \xi_n] \cdot t^v
    \end{equation*}
    equals $0$. That is,
    \begin{equation*}
        M_q = \sum_\ell \bigg[ \sum_I \bigg( \sum_{p \geq \tau_q}^\infty g_{I,q,p}^\ell (x) t^p \bigg) \xi_I \bigg] A_\ell.
    \end{equation*}
    In other words, $M_q T^q$ has a $\widetilde{V}^\bullet \gr^{\ord} \scD_{X \times \bC_t, (\fx, 0)}$-standard representation with respect to $A_1, \ldots, A_m$, since $\tau_q - q \geq \ord^{\widetilde{V}}(B)$ and $\ord^{\widetilde{V}}(A_\ell) = 0$. (When $M_q = 0$, there is trivially a standard representation of $M_q T^q$.) Since standard representations are well-behaved with respect to sums, we have exhibited a $\widetilde{V}^\bullet \gr^{\ord} \scD_{X \times \bC_t, (\fx,0)}$-standard representation of $B$.
\end{proof}

\subsection{Hypotheses for Main Theorems}

Our main theorems require three hypotheses on $f \in \scO_X$:

\begin{hyp} \label{hyp-mainHypotheses}
We will often require the following hypotheses on $f \in \scO_X$:
\begin{enumerate}[label=\alph*)]
    \item $f$ is Euler homogeneous (on $X$);
    \item $f$ is \emph{parametrically prime}, i.e. $\gr^{\sharp}(\ann_{\scD_{X,\fx}[s]} f^{s-1})$ is prime locally everywhere;
    \item the zeroes of the Bernstein-Sato polynomial of $f$ live in $(-2,0)$ locally everywhere.
\end{enumerate}
\end{hyp}

We have already considered condition c) extensively. For us, its potency lies in Proposition \ref{prop-funnyBSpolyInducedV} and Proposition \ref{propcondnBS-new}. Conditions a) and b) have different philosophical roles: we will shortly see their utility in executing Gr{\"o}bner basis arguments. As they are new, they must be defined. Condition b) involves principal symbols of $\ann_{\scD_{X}[s]} f^{s-1}$:

\begin{defn} \label{def-parametricallyPrime}
    Let $f \in \scO_X$ and consider $\ann_{\scD_X[s]} f^{s-1}$,  the $\scD_X[s]$-annihilator of $f^{s-1}$. Recall that $\gr^\sharp (\ann_{\scD_X[s]} f^{s-1}) \subseteq \gr^\sharp \scD_X[s]$ denotes the ideal of principal symbols of $\ann_{\scD_X[s]} f^{s-1}$ with respect to the total order filtration $F^\sharp_{\bullet} \scD_X[s]$ (Definition \ref{def-totalOrderFiltration}). We say that $f$ is \emph{parametrically prime at $\fx \in X$} when $\gr^\sharp (\ann_{\scD_{X,\fx}[s]} f^{s-1})$ is a prime $(\gr^{\sharp} \scD_{X,\fx}[s]$)-ideal. We say that $f \in \scO_X$ is \emph{parametrically prime} when it is parametrically prime at all $\fx \in X$. 
\end{defn}

Condition a) evokes the criterion of homogeneity for standard graded polynomials:

\begin{defn} \label{def-EulerHom} \enspace 
    \begin{enumerate}[label=\roman*)]
        \item Let $f \in \scO_{X,\hspace{1pt}\fx}$ be a function germ. An \emph{Euler vector field} or \emph{Euler homogeneity} $E$ is a derivation $E \in \Der_{\bC}(\scO_{X,\hspace{1pt}\fx})$ such that $E \cdot f = f$. A \emph{strong Euler vector field} or \emph{strong Euler homogeneity $E$} is a Euler homogeneity $E$ such that $E$ vanishes at $\fx$ (i.e. $E \in \mathfrak{m}_{X,\fx} \cdot \Der_{\bC}(\scO_{X,\hspace{1pt}\fx})$). We say that $f$ is \emph{(strongly) Euler homogeneous} if it admits a (strong) Euler homogeneity.
        \item We say $f \in \scO_X$ is \emph{(strongly) Euler homogeneous at $\fx \in X$} if the function germ is $f \in \scO_{X,\hspace{1pt}\fx}$ is (strongly) Euler homogeneous at $\fx$. We say $f \in \scO_X$ is \emph{(strongly) Euler homogeneous} if it is Euler homogeneous at all $\fx \in \text{Var}(f)$. We say \emph{$f$ is (strongly) Euler homogeneous with (strong) Euler homogeneity $E$} when the localization of $E$ is a (strong) Euler homogeneity for all $\fx \in \text{Var}(f)$.
        \item We say a (reduced, effective) divisor $D \subseteq X$ is \emph{(strongly) Euler homogeneous at $\fx \in X$} if it admits a local defining equation $f \in \scO_{X,\hspace{1pt}\fx}$ such that $f$ is (strongly) Euler homogeneous. We say $D$ is (strongly) Euler homogeneous if it is (strongly) Euler homogeneous at all $\fx \in D$. 
    \end{enumerate}
\end{defn}

\begin{rem} \label{rmk-EulerHomBasics} \enspace
    \begin{enumerate}
        \item Euler homogeneity is an open condition: if $f \in \scO_X$ is Euler homogeneous at $\fx \in X$ with Euler homogeneity $E$, then $E$ is an Euler homogeneity for $f$ on an open $U \ni \fx$. Strong Euler homogeneity is not open.
        \item Euler homogeneity at $\fx$ is nothing more than requiring that $f \in (\partial_1 \cdot f, \ldots, \partial_n \cdot f) \subseteq \scO_{X,\hspace{1pt}\fx}$. We have used local coordinates to describe the Jacobian ideal, but this membership is not coordinate dependent.
        \item \label{item-StrongEulerChoice} Let $f \in \scO_{X,\hspace{1pt}\fx}$ and $u \in \scO_{X,\hspace{1pt}\fx}$ a unit. Then it is possible for $f$ to be Euler homogeneous while $uf$ is not Euler homogeneous, cf. Example 2.8.(2) \cite{Wal17}. In contrast, $f$ is strongly Euler homogeneous if and only if $uf$ is strongly Euler homogeneous. Indeed, if $E$ is a strong Euler homogeneity of $f$, then $uE / (u + E \cdot u)$ is a strong Euler homogeneity of $uf$, cf. Example 2.8.(3) \cite{Wal17}.
        \item \label{item-OtherDivMainHyp} The $f \in \scO_X$ considered in \cite{CNS22}, namely strongly Koszul free divisors, are both parametrically prime and strongly Euler homogeneous. As are $f \in \scO_X$ of \emph{linear Jacobian type}, see Definition \ref{def-linearJacType} and Proposition \ref{propLJ0-new}. The latter includes the class of $f \in \scO_X$ considered in the main theorems of \cite{Wal17} Section 3 and in particular those of \cite{Bath24}.
        \item One can check that for $f \in \scO_{X,\hspace{1pt}\fx}$, being parametrically prime at $\fx$ equates to $\gr^\sharp(\ann_{\scD_{X,\fx}[s]} f^{as + b})$ being prime for any $a \in \bC^\star, b \in \bC$.
        \item Macaulay2 experiments suggest that \emph{most} Euler homogeneous germs are parametrically prime.
    \end{enumerate}
\end{rem}

When $f \in \scO_X$ is Euler homogeneous with Euler homogeneity $E$, we saw in Proposition \ref{prop-presentations}'s proof that $\ann_{\scD_{X}[s]} f^{s-1} = \scD_X[s] \cdot \ann_{\scD_X} f^{s-1} + \scD_X[s] \cdot (E - (s-1))$. This suggests being parametrically prime is linked to the primality of $\gr^{\ord} (\ann_{\scD_X}f^{s-1}) \subseteq \gr^{\ord} \scD_X$, the ideal of principal symbols of $\ann_{\scD_X}f^{s-1}$ with respect to the order filtration. In fact, it is the primality of $\gr^{\ord} (\ann_{\scD_X}f^{s-1})$ that is most useful to us. We prove:

\begin{prop} \label{prop-ParaPrimeEquivalent}
    Suppose $f \in \scO_{X,\hspace{1pt}\fx}$ is Euler homogeneous. Then $\gr^\sharp (\ann_{\scD_{X,\fx}[s]} f^{s-1})$ is prime if and only if $\gr^{\ord} (\ann_{\scD_{X,\fx}} f^{s-1})$ is prime.
\end{prop}

\begin{proof}
    Throughout, let $E$ be the Euler homogeneity of $f$ at $\fx$. 
    
    We first prove the harder fact: $\gr^\sharp (\ann_{\scD_{X,\fx}[s]} f^{s-1})$ is prime $\implies \gr^{\ord} (\ann_{\scD_{X,\fx}} f^{s-1})$ is prime. So presume the antecedent. As $E - (s-1) \in \ann_{\scD_{X,\fx}[s]} f^{s-1}$, certainly $\gr^{\ord}(E) - s \in \gr^\sharp (\ann_{\scD_{X,\fx}[s]} f^{s-1})$. So if $Q(s) = \sum Q_t s^t \in \gr^\sharp (\ann_{\scD_{X,\fx}[s]} f^{s-1})$ where $Q_t \in \gr^{\ord} \scD_{X,\fx}$, we may write $Q(s) = (\sum Q_t \gr^{\ord}(E)^t ) + R(s) (\gr^{\ord}(E) - s)$ for some $R(s) \in \gr^{\sharp} \scD_{X,\fx}[s]$. Therefore we may find a $(\gr^{\ord} \scD_{X,\fx})$-ideal $J$ such that
    \begin{equation*}
        \gr^{\sharp} (\ann_{\scD_{X,\fx}[s]} f^{s-1}) = \gr^{\sharp} \scD_{X,\fx}[s] \cdot J + \gr^{\sharp} \scD_{\scD_{X,\fx}[s]} \cdot (\gr^{\ord}(E) - s).
    \end{equation*}

    \emph{Claim:} $J \subseteq \gr^{\ord} (\ann_{\scD_{X,\fx}} f^{s-1})$.

    If we prove \emph{Claim} we will be done. Indeed, since we have assumed that $\gr^{\sharp} (\ann_{\scD_{X,\fx}[s]} f^{s-1})$ is prime, necessarily of height $n$ (consider smooth points near $\fx$), we deduce that $J \subseteq \gr^{\ord} \scD_{X,\fx}$ is prime of height $n-1$. Since $\gr^{\ord}(\ann_{\scD_{X,\fx}} f^{s-1})$ is generically, i.e. at smooth points near $\fx$, prime of height $n-1$, and since $J \subseteq \gr^{\ord}(\ann_{\scD_{X,\fx}} f^{s-1})$, the primality plus a dimension count forces $J = \gr^{\ord}(\ann_{\scD_{X,\fx}} f^{s-1})$.
    
    \emph{Proof of Claim:} Take $p \in J$ of degree $r$ and select a lift $A(s) \in \ann_{\scD_{X}[s]} f^{s-1}$ such that $\gr^{\sharp}(A(s)) = p$. Since $A(s)$ has total order $r$ we may write $A(s) = \sum_{k = 0}^r P_{r-k} s^k$ where $P_{r-k} \in F_{r-k}^{\ord} \scD_{X,\fx}$; since $\gr^{\sharp}(A(s)) = p \in \gr^{\ord} \scD_{X,\fx}$ is (homogeneous) of degree $r$, we deduce both $\gr^{\sharp}(A(s)) = \gr^{\ord}(P_r) = p$ and $P_{r-k} \in F_{r-k-1}^{\ord} \scD_{X,\fx}$ for all $0 \leq k \leq r-1$. 
    
    Now consider $B := \sum_{k=0}^r P_{r-k}(E + 1)^k \in \scD_{X,\fx}$. Observe:
    \begin{equation*}
        B \cdot f^{s-1} = \bigg( \sum_{k=0}^r P_{r-k} (E+1)^k \bigg) \cdot f^{s-1}= \bigg( \sum_{k=0}^r P_{r-k} s^k \bigg) \cdot f^{s-1} = A(s) \cdot f^{s-1} = 0.
    \end{equation*}
    So $B \in \ann_{\scD_{X,\fx}} f^{s-1}$. Because $P_{r-k} (E + 1)^k \in (F_{r-k-1}^{\ord} \scD_{X,\fx}) \cdot (F_{k}^{\ord} \scD_{X,\fx}) \subseteq F_{r-1}^{\ord} \scD_{X,\fx}$ for all $1 \leq k \leq r$, we know that $\gr^{\ord}(B) = \gr^{\ord}(P_r) = p$. This demonstrates that $J \subseteq \gr^{\ord} \ann_{\scD_{X,\fx}} f^{s-1}$. $\square$

        \vspace{5mm}
    
    Now we prove that $\gr^{\ord} (\ann_{\scD_{X,\fx}} f^{s-1})$ is prime $\implies \gr^\sharp (\ann_{\scD_{X,\fx}[s]} f^{s-1})$ is prime. Assuming $\gr^{\ord} (\ann_{\scD_{X,\fx}} f^{s-1}) \subseteq \gr^{\ord} \scD_{X,\fx}$ is prime, necessarily of height $n-1$, we see that the $(\gr^\sharp \scD_{X,\fx}[s])$-ideal on right hand side of
    \begin{equation} \label{eqn-parametricPrime-2}
        \gr^{\sharp} (\ann_{\scD_{X,\fx}[s]} f^{s-1}) \supseteq \gr^{\sharp} \scD_{X,\fx}[s] \cdot \gr^{\ord} (\ann_{\scD_{X,\fx}} f^{s-1}) + \gr^{\sharp} \scD_{X,\fx}[s] \cdot (\gr^{\ord}(E) - s)
    \end{equation} 
    is prime of height $n$. To prove equality in \eqref{eqn-parametricPrime-2} it suffices to confirm that $\gr^{\sharp} (\ann_{\scD_{X,\fx}[s]} f^{s-1})$ is generically prime of height $n$. But this can be checked at a smooth points near $\fx$, as in the discussion after \emph{Claim}.
\end{proof}

\begin{eg} \label{ex:uliQuestion}
    In Question 5.9 of \cite{Wal17}, Walther asked: for an arbitrary $f \in \scO_{X,\hspace{1pt}\fx}$, is $\gr^{\ord}(\ann_{\scD_{X,\fx}} f^{s-1}) \subseteq \gr^{\ord} \scD_{X,\fx}$ always prime? The answer is no. Consider $f = (x_1 x_3 + x_2)(x_1^4 - x_2^4) \in \mathbb{C}[x_1,x_2,x_3]$, an example studied in subsection 3.3 of \cite{LinearityNarvaezMacarro}. Macaulay2 claims that $\gr^{\ord} (\ann_{\scD_{X,0}} f^{s-1}) \subseteq \gr^{\ord} \scD_{X,0} \simeq \scO_{X,0}[y_1, y_2, y_3] $ has $(x_1, x_2, y_3) \subseteq \scO_{X,0}[y_1, y_2, y_3]$ as an embedded prime. Here $f$ is Euler homogeneous and free but it is not Saito-holonomic: all logarithmic derivations of $f$ vanish along $\{x_3 = 0\}$.  (Note that $\gr ^{\ord} \ann_{\scD_{X,0}} f^{s-1}$ is always generically prime and in the Euler homogeneous case $\text{rad} (\gr^{\ord} \ann_{\scD_{X,0}} f^{s-1})$ is always prime and its zero locus is given explicitly by, for example, Proposition 21 of \cite{MaisonobeFiltrationRel}. So here Walther's question is about the existence of embedded primes.) 
\end{eg}

\subsection{Hodge versus Order}

This subsection is devoted to proving the following theorem:

\begin{thm} \label{thmmain-new}
Assume that $f \in \scO_X$ satisfies Hypotheses \ref{hyp-mainHypotheses}: $f$ is Euler homogeneous; $f$ is parametrically prime; the roots of the Bernstein-Sato polynomial of $f$ are contained in $(-2,0)$ locally everywhere. Then, for every non-negative integer $k$, 
\begin{equation*}
    F_k^H\scO_X(*f)\subseteq F_k^{\ord} \scO_X(*f).
\end{equation*}

Moreover, consider a reduced effective divisor $D$ (not necessarily globally defined) that is: strongly Euler homogeneous, $\gr^{\ord} (\ann_{\scD_X} f^{s-1}) \subseteq \gr^{\ord} \scD_X$ is prime (locally everywhere, for some choice of defining equation), and the roots of the Bernstein-Sato polynomial of $D$ are contained in $(-2,0)$ locally everywhere. Then for every non-negative integer $k$, 
\begin{equation*}
    F_k^H\scO_X(*D)\subseteq F_k^{\text{\emph{ord}}}\scO_X(*D).
\end{equation*}
\end{thm}

\vspace{5mm}

This is a significant improvement to Theorem 4.4 \cite{CNS22}, where Hypotheses \ref{hyp-mainHypotheses} are replaced with the much more restrictive condition ``strongly Koszul free.'' (See \cite{Nar15} for proof that ``strongly Koszul free divisors'' satisfy Hypotheses \ref{hyp-mainHypotheses}.) Morally, apart from the Bernstein--Sato condition, when $D$ is given by a global defining equation $f \in \scO_X$, the change from loc. cit. to our Theorem \ref{thmmain-new} is: in loc. cit. they require strongly Euler homogeneous and that $\gr^{\ord} \ann_{\scD_X} f^{s-1}$ is a complete intersection and prime; here we only require Euler homogeneous and primality. In the case $D$ is not globally defined, the change from loc. cit. to our Theorem \ref{thmmain-new} is as above minus weakening of homogeneity type. 

The proof is presented through a sequence of lemmas and propositions. We start with two technical Gr{\"o}bner basis style lemmas. Note that since $F_{\bullet}^{\ord} \scD_{X}$ and $F_{\bullet}^{\ord} \scD_{X \times \bC_t}$ are bounded filtrations, the notions of a Gr{\"o}bner basis with respect to $F_{\bullet}^{\ord}$ and of a $F_{\bullet}^{\ord}$-respecting basis coincide (Lemma \ref{lem-BoundedFiltration}).

\begin{lem} \label{propgroe-primeHyp(lemma)}
    Assume that $f \in \scO_{X,\hspace{1pt}\fx}$ has an Euler homogeneity $E$ and that $\gr^{\ord} \ann_{\scD_{X,\hspace{1pt}\fx}} f^{s-1}$ is prime. Let $\zeta_1, \ldots, \zeta_m$ be a Gr{\"o}bner basis of $\ann_{\scD_{X,\hspace{1pt}\fx}} f^{s-1}$ with respect to the order filtration $F_\bullet^{\ord} \scD_{X,\hspace{1pt}\fx}$. Then the following hold:
    \begin{enumerate}[label=\roman*)]
        \item $t-f, E + \partial_t t + 1, \zeta_1, \dots, \zeta_m$ form a Gr{\"o}bner basis with respect to the order filtration $F_\bullet^{\ord} \scD_{X \times \bC_t, (\fx,0)}$; 
        \item $t - f, \gr^{\ord}(E) + Tt, \gr^{\ord}(\zeta_1), \ldots, \gr^{\ord}(\zeta_m) \in \gr^{\ord} \scD_{X \times \bC_t, (\fx, 0)}$ form a $\widetilde{V}^\bullet \gr^{\ord} \scD_{X \times \bC_t, (\fx, 0)}$-respecting basis.
    \end{enumerate}
\end{lem}

\begin{proof}
    $i)$: Consider the $\scD_{X \times \bC_t, (\fx, 0)}$-ideal 
    \begin{equation*}
        J = \gr^{\ord} \scD_{X \times \bC_t, (\fx, 0)} \cdot \gr^{\ord} \ann_{\scD_{X,\hspace{1pt}\fx}} f^{s-1} = \gr^{\ord} \scD_{X \times \bC_t, (\fx, 0)} \cdot (\gr^{\ord} (\zeta_1), \ldots \gr^{\ord} (\zeta_m)).
    \end{equation*}
    By Lemma \ref{lem-GRespNZDs}, it suffices to show that $t-f, \gr^{\ord}(E + \partial_t t + 1)$ is a regular sequence on $ \gr^{\ord} \scD_{X \times \bC_t, (\fx, 0)} / J$. 
    
    We claim $t-f \notin J$. Indeed, if not, then modding out by $t$ and $T:= \gr^{\ord}(\partial_t)$, implies that $f \in \gr^{\ord} \ann_{\scD_{X,\hspace{1pt}\fx}} f^{s-1}$, which in turn implies the impossible membership $f \in \ann_{\scD_{X,\hspace{1pt}\fx}} f^{s-1}$. Since $J$ is prime by hypothesis, $t-f$ is a non-zero divisor on the cokernel of $J$.
    
    Now we must show that $\gr^{\ord}(E) + Tt$ is a non-zero divisor on $ \gr^{\ord} \scD_{X \times \bC_t, (\fx, 0)} / ((t-f) + J)$. This quotient ring is a domain by hypothesis, being isomorphic to $(\gr^{\ord} \scD_{X,\hspace{1pt}\fx} / J)[T].$ So it suffices to show that $\gr^{\ord}(E) + Tt \notin (t-f) + J$. If this membership held, then $\gr^{\ord}(E) + fT \in \gr^{\ord} \scD_{X,\hspace{1pt}\fx}[T] \cdot \gr^{\ord} \ann_{\scD_{X,\hspace{1pt}\fx}} f^{s-1}$, which, by modding out by $T$, yields $\gr^{\ord}(E) \in \gr^{\ord} \ann_{\scD_{X,\hspace{1pt}\fx}} f^{s-1}$. Since $\gr^{\ord}(E) \in \gr^{\ord} \scD_{X,\hspace{1pt}\fx}$ has degree $1$, this means we can find a $g \in \scO_{X,\hspace{1pt}\fx}$ such that 
    \begin{equation*}
        0 = (E + g) \cdot f^{s-1} = ((s-1) + g) f^{s-1},
    \end{equation*}
    which is clearly impossible. 

    $ii)$: Recall \eqref{eqn - incudedVonGrexplicit} the canonical isomorphism 
\begin{equation} \label{eqn-doubleGrIso}
    \gr_{\widetilde{V}}\gr^{\ord} \scD_{X\times \bC_t, (\fx, 0)} \simeq (\gr^{\ord}\scD_{X,\hspace{1pt}\fx})[t,T].
\end{equation}
The symbols of $t-f, \gr^{\ord}(E)+Tt,\gr^{\ord}(\zeta_1),\ldots, \gr^{\ord}(\zeta_m) \in \gr^{\ord}\scD_{X\times \bC_t,(\fx,0)}$ with respect to $\widetilde{V}^{\bullet} \gr^{\ord}\scD_{X \times \bC_t, (\fx, 0)}$ are, under the above isomorphism, 
\begin{equation*}
    -f, \gr^{\ord}(E)+Tt,\gr^{\ord}(\zeta_1),\ldots,\gr^{\ord}(\zeta_m) \in (\gr^{\ord} \scD_{X,\hspace{1pt}\fx})[t,T].
\end{equation*}
We must show this set is a $\widetilde{V}^\bullet \gr^{\ord} \scD_{X \times \bC_t, (\fx, 0)}$-respecting basis in $(\gr^{\ord} \scD_{X,\hspace{1pt}\fx})[t,T].$

First, since $\gr^{\ord}(\zeta_1), \ldots , \gr^{\ord}(\zeta_m) \in \gr^{\ord} \scD_{X,\hspace{1pt}\fx}$, we know that $\gr^{\ord}(\zeta_1), \ldots , \gr^{\ord}(\zeta_m) \in \gr^{\ord} \scD_{X \times \bC_t, (\fx, 0)}$ is a $\widetilde{V}^\bullet \gr^{\ord} \scD_{X \times \bC_t, (\fx, 0)}$-respecting basis thanks to Lemma \ref{lem-IndVRespBasisonGrOrd}. Now set 
\begin{equation*}
    R = \gr^{\ord} \scD_{X,\hspace{1pt}\fx} / \ann_{\scD_{X,\hspace{1pt}\fx}} f^{s-1}.
\end{equation*}
which is a domain by hypothesis. Lemma \ref{lem-GRespNZDs} shows that to prove $ii)$ it suffices to show that $\overline{\gr^{\ord}(E) + Tt}$, $\overline{f} \in R[t,T]$ is a regular sequence in 
\begin{equation*}
    \gr^{\ord} \scD_{X,\hspace{1pt}\fx}[t,T] / (\gr^{\ord}(\zeta_1), \ldots, \gr^{\ord}(\zeta_m)) \simeq R[t,T].
\end{equation*}

Since $R[tT] \xhookrightarrow{} R[t,T]$ is a flat morphism and flat morphisms preserve regularity (see Proposition 9.6.7 \cite{Bourbaki}), it suffices to prove that $\overline{\gr^{\ord}(E) + Tt}, \overline{f}$ is a regular sequence in $R[tT]$. As $R[tT] / (\overline{\gr(E) + Tt}) \simeq R$, clearly $\overline{\gr^{\ord}(E) + Tt}$ is regular on $R[tT]$.

We are done once we show that $\overline{f}$ is regular on the domain $R$. This is equivalent to showing that $f \notin \gr^{\ord} \ann_{\scD_{X,\hspace{1pt}\fx}} f^{s-1} \subseteq \gr^{\ord} \scD_{X,\hspace{1pt}\fx}$. Since $f \in \gr^{\ord} \scD_{X,\hspace{1pt}\fx}$ has degree zero, this is equivalent to showing that $f \notin \ann_{\scD_{X,\hspace{1pt}\fx}} f^{s-1}$ which is clear.
\end{proof}

Let us clarify what we have just done. Under Hypotheses \ref{hyp-mainHypotheses}, we saw in Proposition \ref{prop-presentations} that if $(\zeta_1, \ldots, \zeta_m) = \ann_{\scD_X} f^{s-1}$ as in Lemma \ref{propgroe-primeHyp(lemma)} and if $E$ is an Euler homogeneity for $f$ on $X$, then
\begin{equation*} \label{eqn-I(f)Defn}
    (t-f, E + \partial_t t, \zeta_1, \ldots, \zeta_m) \subseteq \scD_{X \times \bC_t}
\end{equation*}
is the $\scD_{X \times \bC_t}$-annihilator of $f^{-1} \delta \in \scO_X(*f)[\partial_t] \delta \simeq i_{f,+}\scO_X(*f)$. Moreover, $f^{-1} \delta$ cyclically generates $\scO_X(*f)[\partial_t] \delta \simeq i_{f,+} \scO_X(*f)$. So the preceding lemma opens up Gr{\"o}bner basis techniques for understanding the induced order filtration and $V$-filtrations on $i_{f,+} \scO_X(*f)$, as well as their interactions.

This is what we exploit in the following compatability style result. This should be contrasted to Proposition 3.10 \cite{CNS22}: we mimic the structure of their proof but as we have weakened our assumptions, we augment the argumentation.

\begin{prop} \label{propcomp-new} (Compare Proposition 3.10 \cite{CNS22}) Suppose that $f \in \scO_{X,\hspace{1pt}\fx}$ satisfies: $f$ has an Euler homogeneity $E$; $f$ is parametrically prime at $\fx$; $-1$ is the smallest integral root of the (local) Bernstein-Sato polynomial. Then, for any pair of non-negative integers $k$ and $\ell$,
\begin{equation*}
F_\ell^{\ord} i_{f,+}\scO_{X,\hspace{1pt}\fx} (*f)\cap V_{\ind}^ki_{f,+}\scO_{X,\hspace{1pt}\fx}(*f) = (F_\ell^{\ord} \scD_{X\times\bC_t, (\fx, 0)} \cap V^k\scD_{X\times\bC_t, (\fx, 0)}) \cdot f^{-1}\delta.
\end{equation*}
\end{prop}

\begin{proof} 
First of all, this can be checked in local coordinates, which is what we will do. Also, in what follows $\ord(-)$ denotes the order of an element of $\scD_{X \times \bC_t, (\fx, 0)}$ with respect to the order filtration $F_\bullet^{\ord} \scD_{X \times \bC_t, (\fx, 0)}$, whereas $\ord^{\widetilde{V}}(-)$ denotes the order of an element in $\gr^{\ord} \scD_{X \times \bC_t, (\fx, 0)}$ with respect to the the filtration $\widetilde{V}^{\bullet} \gr^{\ord}\scD_{X \times \bC_t, (\fx, 0)}.$

Let $\zeta_1, \ldots, \zeta_m$ be a Gr{\"o}bner basis of $\ann_{\scD_{X,\hspace{1pt}\fx}} f^{s-1}$ with respect to the order filtration $F_\bullet^{\ord} \scD_{X,\hspace{1pt}\fx}$. Let $I(f) = (t-f, E + \partial_t t, \zeta_1, \ldots, \zeta_m) \subseteq \scD_{X \times \bC_t, (\fx, 0)}$ be the left $\scD_{X \times \bC_t, (\fx, 0)}$-ideal coinciding, thanks to our assumptions, with the $\scD_{X \times \bC_t, (\fx, 0)}$-annihilator of $f^{-1} \delta \in i_{f,+} \scO_{X,\hspace{1pt}\fx}(*f)$, cf. Proposition \ref{prop-presentations}. Using \cite{MSai88}, Lemma 1.2.14, we reduce as in the proof of \cite{CNS22}, Proposition 3.10 to showing that the principal symbol map
\begin{equation} \label{eqn-Compatibility-0}
    F_\ell^{\ord} \scD_{X\times \bC_t, (\fx, 0)} \cap V^k\scD_{X\times \bC_t, (\fx, 0)} \cap I(f) \to \widetilde{V}^k \gr^{\ord} \scD_{X \times \bC_t, (\fx, 0)} \cap \gr_{\ell}^{\ord}(I(f))
\end{equation}
is surjective. Here $\gr_{\ell}^{\ord}: \scD_{X \times \bC_t, (\fx, 0)} \to \gr_{\ell}^{\ord}(\scD_{X \times \bC_t, ({\fx, 0)}}) \subseteq \gr^{\ord}(\scD_{X \times \bC_t, ({\fx, 0)}})$ is the principal symbol map sending an operator of order $\ell$ to its principal symbol in the submodule of homogeneous degree $\ell$ elements (and otherwise sending an operator to $0$). 

Consider $P \in I(f)$ such that $\ord(P) = \ell$. We must find an appropriate lift of $P$ to validate the surjectivity of \eqref{eqn-Compatibility-0}. We have 
\begin{equation*}
    \gr_{\ell}^{\ord}(P) \in \gr^{\ord}(I(f)) = (\gr^{\ord}(G_0), \ldots, \gr^{\ord}(G_{m+1})) \subseteq \gr^{\ord}(\scD_{X \times \bC_t, (\fx, 0)})
\end{equation*}
where we write $G_0 = t-f$, $G_i = \zeta_i$ for $1 \leq i \leq m$, and $G_{m+1} = E + \partial_t t + 1$. Here we used Proposition \ref{prop-ParaPrimeEquivalent} to invoke Lemma \ref{propgroe-primeHyp(lemma)}.i) which says that $G_0, \dots, G_{m+1}$ is a Gr{\"o}bner basis for $I(f)$. As Lemma \ref{propgroe-primeHyp(lemma)}.ii) says that $\gr^{\ord}(G_0), \ldots, \gr^{\ord}(G_{m+1})$ is a $\widetilde{V}^{\bullet} \gr^{\ord} \scD_{X \times \bC_t, (\fx, 0)}$-respecting basis for $\gr^{\ord}(I(f))$, we may now find $k_0, \ldots, k_{m+1} \in \gr^{\ord} \scD_{X \times \bC_t, (\fx, 0)}$ such that 
\begin{equation} \label{eqn-Compatibility-1}
    \gr_{\ell}^{\ord}(P) = \sum_{i=0}^{m+1} k_i \gr^{\ord}(G_i) \quad \text{ where} \quad k_i \in \widetilde{V}^{\ord^{\widetilde{V}}(\gr_{\ell}^{\ord}(P)) - \ord^{\widetilde{V}}(\gr^{\ord}(G_i))} \gr^{\ord} \scD_{X \times \bC_t, (\fx, 0)}.
\end{equation}
(Recall that $\widetilde{V}^{\bullet} \gr^{\ord} \scD_{X \times \bC_t, (\fx, 0)}$ is a decreasing filtration with explicit local description \eqref{eqn - incudedVonGrexplicit}.) Since $G_1, \dots, G_m \in \scD_{X,\hspace{1pt}\fx}$, we have that $\ord^{\widetilde{V}}(\gr^{\ord}(G_i)) = 0$ for all $1\leq i \leq m$. By the definitions of $G_0$ and $G_{m+1}$ we also have $\ord^{\widetilde{V}}(\gr^{\ord}(G_0))=\ord^{\widetilde{V}} (\gr^{\ord}(G_{m+1})) = 0.$ So we may rewrite \eqref{eqn-Compatibility-1} as: there exist $k_0, \ldots, k_{m+1} \in \gr^{\ord} \scD_{X \times \bC_t, (\fx, 0)}$ such that 
\begin{equation} \label{eqn-Compatibility-2}
    \gr_{\ell}^{\ord}(P) = \sum_{i=0}^{m+1} k_i \gr^{\ord}(G_i) \quad \text{ where} \quad k_i \in \widetilde{V}^{\ord^{\widetilde{V}}(\gr_{\ell}^{\ord}(P))} \gr^{\ord} \scD_{X \times \bC_t, (\fx, 0)}.
\end{equation}  

Take one such $k_i$ and write $k_i = \sum_{j_i \geq 0} k_{i,j}$ where $k_{i,j} \in \gr_j^{\ord} \scD_{X \times \bC_t, (\fx, 0)}$ is homogeneous of degree $j$. We have 
\begin{equation*}
    k_i \in \widetilde{V}^{\ord^{\widetilde{V}} (\gr_{\ell}^{\ord}(P))} \gr^{\ord} \scD_{X \times \bC_t, (\fx,0)} = \scO_{X \times \bC_t, (\fx, 0)}[\xi_1, \ldots, \xi_n, (tT)] \cdot t^{{\ord^{\widetilde{V}}(\gr_{\ell}^{\ord}(P))}}
\end{equation*}
Write $k_i$ as a sum of monomials $g_{I,q} \xi_I (tT)^{q} t^{\ord^{\widetilde{V}}(\gr_{\ell}^{\ord}(P))}$ for suitable $g_{I,q} \in \mathscr{O}_{X \times \bC_t, (\fx, 0)}$. Each of these monomials then also belongs to $\widetilde{V}^{\ord^{\widetilde{V}}(\gr_{\ell}^{\ord}(P))} \gr^{\ord} \scD_{X \times \bC_t, (\fx, 0)}$. In particular, each $k_{i, j} \in \widetilde{V}^{\ord^{\widetilde{V}}(\gr_{\ell}^{\ord}(P))} \gr^{\ord} \scD_{X \times \bC_t, (\fx, 0)}$. Since $\gr_{\ell}^{\ord}(P)$ is homogeneous of degree $\ell$ and since each $\gr^{\ord}(G_i)$ is homogeneous of degree $\ord(G_i)$ we may simplify \eqref{eqn-Compatibility-2} as: there exist $k_{0, \ell-\ord(G_0)}, \ldots , k_{m+1, \ell-\ord(G_0)} \in \gr^{\ord} \scD_{X \times \bC_t, (\fx, 0)}$ homogeneous of degrees $\ell-\ord(G_0), \ldots, \ell-\ord(G_{m+1})$ respectively such that
\begin{equation} \label{compatibility-3}
\gr_{\ell}^{\ord}(P) = \sum_{i=0}^{m+1} k_{i, \ell-\ord(G_i)} \gr^{\ord}(G_i) \quad \text{ where} \quad k_{i, \ell-\ord(G_i)} \in \widetilde{V}^{\ord^{\widetilde{V}}(\gr_{\ell}^{\ord}(P))} \gr^{\ord} \scD_{X \times \bC_t, (\fx, 0)}.
\end{equation}

As noted in the proof of Corollary 3.9 \cite{CNS22}, the principal symbol map 
\begin{equation*}
    F_{\ell}^{\ord} \scD_{X \times \bC_t, (\fx, 0)} \cap V^k \scD_{X \times \bC_t, (\fx, 0)} \to \widetilde{V}^k \gr^{\ord} \scD_{X \times \bC_t, (\fx, 0)} \cap \gr_{\ell}^{\ord} \scD_{X \times \bC_t, (\fx, 0)}
\end{equation*}
is surjective. Thus for each $k_{i, \ell-\ord(G_i)}$ we may find $K_{i, \ell-\ord(G_i)} \in \scD_{X \times \bC_t, (\fx, 0)}$ such that both $\gr^{\ord}(K_{i, \ell-\ord(G_i)}) = k_{i, \ell-\ord(G_i)}$ and $K_{i, \ell-\ord(G_i)} \in V^{\ord^{\widetilde{V}}(\gr_{\ell}^{\ord}(P))} \scD_{X \times \bC_{t}, (\fx, 0)}$. And
\begin{equation*}
    \gr^{\ord} \bigg( \sum_{i} K_{i, \ell-\ord(G_i)} G_i \bigg) = \gr_{\ell}^{\ord}(P).
\end{equation*}
Moreover, by construction, $\sum_{i} K_{i, \ell-\ord(G_i)} G_i$ is in $I(f)$, is in $F_{\ell}^{\ord} \scD_{X \times \bC_t, (\fx, 0)}$, and is in $V^{\ord^{\widetilde{V}}(\gr_{\ell}^{\ord}(P))} \scD_{X \times \bC_t, (\fx, 0)}$. So we are done: the surjectivity of \eqref{eqn-Compatibility-0} is guaranteed by the constructed $\sum K_{i, \ell-\ord(G_i)} G_i$. 
\end{proof}

We deduce a precise relationship between the Hodge and order filtrations on $V^0 i_{f,+} \scO_X(*f)$. 

\begin{prop} \label{propiota-new}
Assume that $f \in \scO_X$ satisfies Hypotheses \ref{hyp-mainHypotheses}: it is Euler homogeneous; $f$ is parametrically prime; the roots of the Bernstein-Sato polynomial are contained in $(-2,0)$ locally everywhere. Then for any positive integer $k$,
\begin{equation*}
    F_k^H V^0i_{f,+}\scO_X(*f)= F_{k-1}^{\text{\emph{ord}}} V^0i_{f,+} \scO_X(*f).
\end{equation*}
\end{prop}

\begin{proof}
By Lemma \ref{lemHtordV} it suffices to show that 
\[F_k^{\Torder} V^0i_{f,+}\scO_X(*f)= F_k^{\text{ord}} V^0i_{f,+} \scO_X(*f)\]
for every non-negative integer $k$. This holds if and only if it holds locally, so it suffices to prove the equality in the case that $X$ is an sufficiently small open subset of $\bC^n$. 

We prove this by first proving the following claim:

\vspace{10pt}

\emph{Claim:} For any non-negative integer $k$ and $u \in F_k^{\Torder}i_{f,+}\scO_X(*f)$, there exists some $p \in \bZ_{\geq 0}$ such that $t^p\cdot u \in F_k^{\text{ord}}i_{f,+}\scO_X(*f)$.

\emph{Proof of claim:} We prove this inductively on $k$. For $k=0$, $u=u_0 \delta$ for some $u_0 \in\scO_X(*f)$. Now, there exists some $p \geq 0$ such that $f^pu_0 \in \scO_Xf^{-1}$. Therefore, $t^p\cdot u = f^pu_0\delta \in \scO_Xf^{-1}\delta =F_0^{\text{ord}}i_{f,+}\scO_X(*f)$.

For $k >0$, write $u=\sum_{i=0}^ku_i\partial_t^i\delta$, where $u_i \in\scO_X(*f)$ for each $i$. Now, there exists some $p \geq 0$ and $g \in \scO_X$ such that $f^pu_k = g f^{-1}$. Therefore, $t^p\cdot u - \partial_t^k g f^{-1} \delta \in F_{k-1}^{\Torder} i_{f,+} \scO_X(*f)$. By our inductive hypothesis, there exists some $q \geq 0$ such that $t^{p+q}\cdot u - t^q \partial_t^k g f^{-1} \delta \in F_{k-1}^{\ord} i_{f,+} \scO_X(*f)$. Rearranging yields $t^{p+q} \cdot u \in F_k^{\ord} i_{f,+} \scO_X(*f)$. $\square$

\vspace{10pt}

\noindent Now we prove the statement of the proposition. The inclusion 
\[F_k^{\text{ord}} V^0i_{f,+}\scO_X(*f)\subseteq F_k^{\Torder} V^0i_{f,+} \scO_X(*f)\]
is immediate by the definition of the order filtration. For the reverse inclusion, take any $u \in F_k^{\Torder}V^0i_{f,+}\scO_X(*f)$. Then, by the above claim, there exists some $p \geq 0$ such that $t^p \cdot u \in F_k^{\text{ord}}V^pi_{f,+}\scO_X(*f)$. 

Since $V^pi_{f,+}\scO_X(*f) \subseteq V^p_{\text{ind}}i_{f,+}\scO_X(*f)$ by Proposition \ref{prop-funnyBSpolyInducedV}, we see that 
\[t^p \cdot u \in F_k^{\text{ord}}V^p_{\text{ind}}i_{f,+}\scO_X(*f) = (V^p\scD_{X\times\bC_t}\cap F_k^{\ord}\scD_{X\times\bC_t}) \cdot f^{-1}\delta,\]
by Proposition \ref{propcomp-new}. Namely, there exists some $P\in V^p\scD_{X\times\bC_t}\cap F_k^{\ord}\scD_{X\times\bC_t}$ such that $t^p \cdot u =P\cdot f^{-1}\delta$. However, by the definition of $V^p\scD_{X\times\bC_t}$, there exists some $P'\in V^0\scD_{X\times\bC_t}\cap F_k^{\ord}\scD_{X\times\bC_t}$ such that $P=t^pP'$. Since $t$ acts invertibly on $i_{f,+}\scO_X(*f)$, this finally implies that
\[u = P'\cdot f^{-1}\delta \in F_k^{\text{ord}}i_{f,+}\cO_X(*f).\]

\end{proof}

This essentially completes the proof of Theorem \ref{thmmain-new}:

\begin{proof}[Proof of Theorem \ref{thmmain-new}] 
Let us first assume that $D$ is globally defined by $f \in \scO_X$ as in the first paragraph of the statement. By Lemma \ref{lemorderfilt2-new}, the claims of Theorem \ref{thmmain-new} and Proposition \ref{propiota-new} are equivalent. When $D$ is not globally defined, the strongly Euler homogeneous assumption along with Remark \eqref{rmk-EulerHomBasics}.\eqref{item-StrongEulerChoice} allows us to reduce to the case $f \in \scO_U$ global and Euler homogeneous.
\end{proof}

\subsection{A Formula for the Hodge Filtration}

Here we prove our second main result, giving a formula for $F_k^H \scO_{X,\hspace{1pt}\fx}(*f)$ for $f \in \scO_{X,\hspace{1pt}\fx}$ a suitable germ. For readability, we recite some iconography:
\begin{itemize}
    \item when $Z(b_f(s)) \subseteq (-2,0)$, $\beta_f(s)$ is obtained from $b_f(-s-1)$ by removing all roots (and their multiplicities) living in $(-1,0]$, cf. Definition \ref{def-funnyBSpoly};
    \item when $Z(b_f(s)) \subseteq (-2,0)$, $\beta_f^{\prime}(s) = b_f(-s-1) / \beta_f(s)$, cf. Definition \ref{def-funnyBSpoly};
    \item $\Gamma_f$ is the $\scD_{X,\hspace{1pt}\fx}[s]$-ideal $\scD_{X,\hspace{1pt}\fx}[s] \cdot f + \scD_{X,\hspace{1pt}\fx}[s] \cdot \beta_f(-s) + \ann_{\scD_{X,\hspace{1pt}\fx}[s]} f^{s-1}$, cf. Definition \ref{def-Gammaf};
    \item $\phi_0 : \scD_{X,\hspace{1pt}\fx}[s] \to \scD_{X,\hspace{1pt}\fx}$ is the $\scD_{X,\hspace{1pt}\fx}$-map given by $s \mapsto 0$, cf. Lemma \ref{lem-commDiagram};
    \item $\specialize \scD_{X,\hspace{1pt}\fx}[s]\cdot f^{s-1} \to \scD_{X,\hspace{1pt}\fx}\cdot f^{-1}$ is the specialization map, cf. Definition \ref{def-specialization};
    \item $\pi_f : \scD_{X,\hspace{1pt}\fx}[s] \to \scO_{X,\hspace{1pt}\fx}(*f)[\partial_t] \delta \simeq i_{f,+} \scO_{X,\hspace{1pt}\fx}(*f)$ is the $\scD_{X,\hspace{1pt}\fx}$-linear map given by $P(s) \mapsto P(-\partial_t t) \cdot f^{-1} \delta$.
\end{itemize}

We will prove the following extension of Corollary \ref{cor-zeroHodgePiece} to all pieces of $F_\bullet^H \scO_{X,\hspace{1pt}\fx}(*f)$:

\begin{thm} \label{thmgamma2-new}
Assume that $f \in \scO_{X,\hspace{1pt}\fx}$ satisfies Hypotheses \ref{hyp-mainHypotheses}: $f$ is Euler homogeneous; $f$ is parametrically prime at $\fx$; the zeroes of the Bernstein-Sato polynomial are contained in $(-2,0)$. Then, for all $k \in \mathbb{Z}_{\geq 0}$,
\begin{align*}
    F_k^H\scO_{X,\hspace{1pt}\fx}(*f) 
    = \phi_0 \bigg( \Gamma_f\cap F_k^{\sharp} \scD_{X,\hspace{1pt}\fx}[s] \bigg) \cdot f^{-1}
\end{align*}

In other words, let $\specialize: \scD_{X,\hspace{1pt}\fx}[s]\cdot f^{s-1} \to \scD_{X,\hspace{1pt}\fx}\cdot f^{-1}$ be the specialization map, cf. Definition \ref{def-specialization}. Then $F_k^H \scO_{X,\hspace{1pt}\fx}(*f)$ is the $\scO_{X,\hspace{1pt}\fx}$-module generated by $\{\specialize(P(s) \cdot f^{s-1})\}$ where $P(s)$ ranges over all elements of $F_k^\sharp \scD_{X,\hspace{1pt}\fx}[s]$ satisfying the ``functional equation''
\begin{equation*}
    P(s) \beta_f^\prime(-s) \cdot f^{s-1} \in \scD_{X,\hspace{1pt}\fx}[s]\cdot f^{s}.
\end{equation*}
\end{thm}

\vspace{5mm}

As in Corollary \ref{cor-zeroHodgePiece}, the proof uses the commutative diagram Lemma \ref{lem-commDiagram} to transfer Hodge data between $i_{f,+} \scO_{X,\hspace{1pt}\fx}(*f)$, $\scD_{X,\hspace{1pt}\fx}[s]\cdot f^{s-1}$, and $\scO_{X,\hspace{1pt}\fx}(*f)$. But this time Proposition \ref{propcomp-new} makes the argument more robust. First:

\begin{prop} \label{proppi-new}
Assume that the zeroes of the Bernstein-Sato polynomial of $f \in \scO_{X,\hspace{1pt}\fx}$ are contained in $(-2,0)$ and that the order filtration and the induced $V$-filtration on $i_{f,+}\scO_{X,\hspace{1pt}\fx}(*f)$ are compatible in the sense of Proposition \ref{propcomp-new}. Then, for all $k \in \mathbb{Z}_{\geq 0}$, 
\begin{equation*}
    F_k^{\ord} V^0 i_{f,+} \scO_{X,\hspace{1pt}\fx}(*f) = \pi_f(\Gamma_f\cap F_k^{\sharp}\scD_{X,\hspace{1pt}\fx}[s]).
\end{equation*}
\end{prop}

\begin{proof}
Proposition \ref{propcondnBS-new}.i) says $\pi_f(\Gamma_f)=V^0 i_{f,+} \scO_{X,\hspace{1pt}\fx}(*f)$.  So the containment $\supseteq$ in the statement of the proposition is clear since $\pi_f$ cannot increase order. We must prove the reverse inclusion. Assume that $u \in F_k^{\text{ord}}V^0i_{f,+}\scO_{X,\hspace{1pt}\fx}(*f)$. Then $u \in V^0i_{f,+}\scO_{X,\hspace{1pt}\fx}(*f)$ and so there exists $Q(s) \in \Gamma_f$ such that $u = \pi_f(Q(s))$.

By Proposition \ref{prop-funnyBSpolyInducedV}, $V^0i_{f,+}\scO_{X,\hspace{1pt}\fx}(*f)\subseteq V^0_{\ind} i_{f,+}\scO_{X,\hspace{1pt}\fx}(*f)$ and so $u \in F_k^{\ord}V^0_{\ind} i_{f,+}\scO_{X,\hspace{1pt}\fx}(*f)$. Therefore, by the hypothesis of the compatibility of the order filtration and the induced $V$-filtration, there exists some $P\in F_k^{\ord} \scD_{X\times\bC_t, (\fx,0)}\cap V^0\scD_{X\times\bC_t, (\fx,0)}$ such that $u = P\cdot f^{-1}\delta$. By using the definition of the $V$-filtration on $\scD_{X \times \bC_t, (\fx, 0)}$ along $\{t=0\}$ as well as the fact $t \cdot g f^{-1} \delta = f g f^{-1} \delta$ (for any $g \in \scO_{X,\hspace{1pt}\fx}$), we may find a $P^{\prime} \in \sum_{\ell=0}^{k} \sum_{j=0}^{k - \ell}(F_\ell^{\ord} \scD_{X,\hspace{1pt}\fx}) (\partial_t t)^j$ such that $u = P \cdot f^{-1} \delta = P^{\prime} \cdot f^{-1} \delta$. So there is some $\widetilde{P}(s) \in F_k^\sharp \scD_{X,\hspace{1pt}\fx}[s]$ such that $P^{\prime} = \widetilde{P} (-\partial_t t)$ and, in particular, $\pi_f (\widetilde{P}(s)) = P^{\prime} \cdot f^{-1} \delta = P \cdot f^{-1} \delta = u$.

So now we have that $\pi_f(\widetilde{P}(s))=\pi_f(Q(s))=u$. Hence $\widetilde{P}(s)-Q(s)\in\ker\pi_f$. We also have the commutative diagram of $\scD_{X,\hspace{1pt}\fx}$-maps from Lemma \ref{lem-commDiagram}
\begin{center}
\begin{tikzcd}
\scD_{X,\hspace{1pt}\fx}[s] \arrow[rr, bend left = 20, "\pi_f"] \arrow[r, twoheadrightarrow] & \scD_{X,\hspace{1pt}\fx}[s]\cdot f^{s-1} \arrow[r, "\varphi_f"] & i_{f,+}\scO_{X,\hspace{1pt}\fx}(*f),
\end{tikzcd}
\end{center}
and in Lemma \ref{lem-commDiagram} it is proven that $\ker \pi_f = \ann_{\scD_{X,\hspace{1pt}\fx}[s]} f^{s-1}$. As $\ann_{\scD_{X,\hspace{1pt}\fx}[s]} f^{s-1} \subseteq \Gamma_f$ by definition, $\widetilde{P}(s) - Q(s) \in \ker \pi_f \implies \widetilde{P}(s) \in \Gamma_f$. And thus we are done, since then $u = \pi_f(\widetilde{P}(s))\in \pi_f(\Gamma_f\cap F_k^{\sharp}\scD_{X,\hspace{1pt}\fx}[s])$.
\end{proof}

Now we can prove Theorem \ref{thmgamma2-new}:

\begin{proof}[Proof of Theorem \ref{thmgamma2-new}]
The justifications of ``in other words'' follows from the first formula via Lemma \ref{prop-funnyBSFunctional}. 

Recall the commutative diagram of $\scD_{X,\hspace{1pt}\fx}$-maps from Lemma \ref{lem-commDiagram}:
    \begin{equation*}
        \begin{tikzcd}
            \widetilde{\phi}_{f,0}: \scD_{X,\hspace{1pt}\fx}[s] \arrow[rr, bend left = 20, "\pi_f"] \arrow[r, twoheadrightarrow] 
                & \scD_{X,\hspace{1pt}\fx}[s]\cdot f^{s-1} \arrow[r, hookrightarrow, "\varphi_f"] \arrow[rr, bend right = 20, swap, "\specialize"]
                    & i_{f,+}\scO_{X,\hspace{1pt}\fx}(*f) \arrow[r, "\psi_{f,0}"]
                        & \scO_{X,\hspace{1pt}\fx}(*f).
        \end{tikzcd}
    \end{equation*}
Also remember that: $\psi_{f,0} : i_{f,+} \scO_{X,\hspace{1pt}\fx}(*f) \to \scO_{X,\hspace{1pt}\fx}(*f)$ is given by $\partial_t \mapsto 0$ and $u \delta \mapsto u$; $\varphi_f : \scD_{X,\hspace{1pt}\fx}[s]\cdot f^{s-1} \to i_{f,+} \scO_{X,\hspace{1pt}\fx}(*f)$ is as defined in Proposition \ref{prop-MP-graphIsofs}. 

Recall also (see the justification for \eqref{eqn-reformulation-0}) that Theorem \ref{thmformula} is equivalent to
\begin{equation} \label{eqn-reformulation}
    F_k^H\scO_{X,\hspace{1pt}\fx}(*f) = \psi_{f,0}(F_k^{\Torder}V^0i_{f,+}\scO_{X,\hspace{1pt}\fx}(*f)).
\end{equation}

Furthermore, we know, by Lemma \ref{lemHtordV} and Proposition \ref{propiota-new}, that under our hypotheses
\begin{equation} \label{eqn-VordVersusTorder-help}
    F_k^{\Torder}V^0i_{f,+}\scO_{X,\hspace{1pt}\fx}(*f)=F_k^{\text{ord}}V^0i_{f,+}\scO_{X,\hspace{1pt}\fx}(*f).
\end{equation}

Moreover, Proposition \ref{proppi-new} gives us the equality
\begin{equation} \label{eqn-VordVersusGamma}
    F_k^{\ord}V^0i_{f,+}\scO_{X,\hspace{1pt}\fx}(*f) = \pi_f(\Gamma_f\cap F_k^{\sharp}\scD_{X,\hspace{1pt}\fx}[s]).
\end{equation}

So finally, combining \eqref{eqn-reformulation}, \eqref{eqn-VordVersusTorder-help} and \eqref{eqn-VordVersusGamma} into the commutative diagram Lemma \ref{lem-commDiagram} yields:
\begin{align*} 
\phi_0\left((\Gamma_f \cap F^{\sharp}_k\scD_{X,\hspace{1pt}\fx}[s])\right)\cdot f^{-1} 
&= \specialize \left((\Gamma_f \cap F^{\sharp}_k\scD_{X,\hspace{1pt}\fx}[s])\cdot f^{s-1}\right) \\
&= \psi_{f,0}\left(\varphi_f\left((\Gamma_f \cap F^{\sharp}_k\scD_{X,\hspace{1pt}\fx}[s])\cdot f^{s-1}\right)\right) \\ 
&= \psi_{f,0}\left(\pi_f\left(\Gamma_f \cap F^{\sharp}_k\scD_{X,\hspace{1pt}\fx}[s]\right)\right) \\[0.2em]
&= \psi_{f,0}\left(F_k^{\text{ord}}V^0i_{f,+}\scO_{X,\hspace{1pt}\fx}(*f)\right)\\[0.5em]
&= \psi_{f,0}\left(F_k^{t\text{-ord}}V^0i_{f,+}\scO_{X,\hspace{1pt}\fx}(*f)\right) \\[0.5em]
&= F_k^H\scO_{X,\hspace{1pt}\fx}(*f).\end{align*}
\end{proof}

\subsection{A Special Case}

When we additionally assume that the roots of the Bernstein--Sato polynomial lie in $(-2,-1]$, we have equality of Hodge and induced order filtrations:

\begin{cor}

Assume that $f \in \scO_{X,\hspace{1pt}\fx}$ satisfies Hypotheses \ref{hyp-mainHypotheses}: it is Euler homogeneous; $f$ is parametrically prime at $\fx$; the roots of the Bernstein-Sato polynomial are contained in $(-2,0)$. Assume moreover that $Z(b_f(s))\subseteq(-2,-1]$. Then, for every non-negative integer $k$, 
\[F_k^H\scO_{X,\hspace{1pt}\fx}(*f)=F_k^{\text{\emph{ord}}}\scO_{X,\hspace{1pt}\fx}(*f).\]

\end{cor}

\begin{proof}
    Combining Corollary \ref{corrf} and Theorem \ref{thmmain-new}, we have the chain of inclusions
\[F_{k-r_f}^{\ord}\scO_{X,\hspace{1pt}\fx}(*f)\subseteq F_k^H\scO_{X,\hspace{1pt}\fx}(*f)\subseteq F_k^{\ord}\scO_{X,\hspace{1pt}\fx}(*f).\]
    Our assumption $Z(b_f(s)) \subseteq (-2,-1]$ implies $r_f = 0$, whence the claim. 
\end{proof}

\begin{cor}

Assume that $f \in \scO_X$ satisfies Hypotheses \ref{hyp-mainHypotheses}: it is Euler homogeneous; it is parametrically is prime; the roots of the Bernstein-Sato polynomial are contained in $(-2,0)$ locally everywhere. Let $D$ be the divisor arising from our global $f \in \scO_X$. Then $D$ is log-canonical, and if $\pi:(Y,E)\to(X,D)$ is a log resolution of the pair $(X,D)$, isomorphic over $Y\backslash E$, then $R^p\pi_*\Omega_Y^{n-p}(\log E)=0$ for all $p \in\bZ_{\geq 1}$.

\end{cor}

\begin{proof}

This is immediate by the previous Corollary, using the main results of \cite{MP19a}.
    
\end{proof}

\begin{rem}
    The previous two corollaries also hold for reduced effective divisors $D$ that are not necessarily globally defined, provided we add the assumption that $D$ is strongly Euler homogeneous.     
\end{rem}

\section{Linear Jacobian Type}

Our objective is to introduce a condition on $f \in \scO_X$ or a reduced effective divisor $D$ ensuring that: $\gr^{\sharp} \ann_{\scD_X[s]} f^{s-1} \subseteq \gr^{\sharp} \scD_X[s]$ is prime; $\ann_{\scD_{X}[s]} f^{s-1}$ has as ``simple'' of a generating set as possible. The former certifies the parametrically prime assumption in Hypotheses \ref{hyp-mainHypotheses}; when all these hypotheses hold, the latter helps make computations tractable. 

The condition goes by \emph{linear Jacobian type} (Definition \ref{def-linearJacType}): it is nothing more than asking the ideal defining the singular locus to be of linear type. This is an algebraic notion. When the restriction on roots of the Bernstein-Sato polynomial hold, linear Jacobian type divisors satisfy our main theorems (Corollary \ref{cor-LJTsatisfiesHyp}). In the final subsection we introduce some geometric properties on a divisor ensuring it is of linear Jacobian type, which yields a more explicit class of divisors satisfying our main theorems (Theorem \ref{thmOfUli}). 

\subsection{The definition of linear Jacobian type}

The condition of $f \in \scO_X$ being linear Jacobian type involves the defining ideal of the singular locus:

\begin{defn} \label{def-Jacobian}
    For $f \in \scO_X$, we call the \emph{true Jacobian ideal} $\scJ_f$ of $f$ the $\scO_X$-ideal generated by $f$ and its partial derivatives. In local coordinates,
    \begin{equation*}
        \scJ_f = (f, \frac{\partial f}{\partial x_1}, \ldots, \frac{\partial f}{\partial x_n} ) \subseteq \scO_X.
    \end{equation*}
    We can also consider the $\scO_X$-ideal $\widetilde{\scJ_f}$ generated by just the partial derivatives, which, in local coordinates, is given by
    \begin{equation*}
        \widetilde{\scJ_f} = (\frac{\partial f}{\partial x_1}, \ldots, \frac{\partial f}{\partial x_n} ) \subseteq \scO_X.
    \end{equation*}
    These ideals can be checked to be well-behaved with respect to coordinate changes and choice of defining equation.
\end{defn}

\begin{defn}[c.f. Section 7.2 \cite{Vas97}] \label{def-linearType}
Let $R$ be a commutative ring and $I\subseteq R$ an ideal. $I$ is of \emph{linear type} if the canonical surjective map of graded $R$-algebras
\begin{equation} \label{eqn-linearType}
    \text{Sym}_R(I) \to \text{Rees}(I):=\sum_{i=0}^{\infty}(It)^i \subseteq R[t]
\end{equation}
is an isomorphism. Here \eqref{eqn-linearType} is the unique $R$-algebra homomorphism extending the linear map
\begin{equation*}
    I \to \text{Rees}(I) \quad \text{where} \quad g \mapsto gt,
\end{equation*}
guaranteed by the universal property of the symmetric algebra $\text{Sym}_R(I)$.
\end{defn}

\begin{defn}[c.f. \cite{CN08}] \label{def-linearJacType}
$f \in \scO_X$ is of \emph{linear Jacobian type} if $\scJ_f\subseteq \scO_X$ is of linear type. A reduced effective divisor $D$ is of linear Jacobian type if it is so locally everywhere.
\end{defn}

We will make precise the relationship between $\scJ_f$ being of linear type and $\widetilde{\scJ_f}$ being of linear type. It is relatively immediate that $f \in \scO_X$ being Euler homogeneous is equivalent to $\scJ_f = \widetilde{\scJ_f}$. The same is true for the divisor version. So in the Euler homogeneous setting, $\scJ_f$ being of linear type is equivalent to $\widetilde{\scJ_f}$ being of linear Jacobian type. Amazingly we have the following, due to Narváez Macarro (Proposition 1.5 \cite{Nar15}):

\begin{prop}[Proposition 1.5 \cite{Nar15}] \label{propLJ-new}
Let $f \in \scO_X$. Then the following are equivalent:
    \begin{enumerate}[label=\roman*)]
        \item $f$ is of linear Jacobian type, i.e. $\scJ_f \subseteq \scO_X$ is of linear type.
        \item The ideal $\widetilde{\scJ}_f\subseteq\scO_X$ is of linear type and $f$ is strongly Euler homogeneous.
\end{enumerate}
\end{prop}

\begin{proof}
That $ii) \implies i)$ follows since $f$ is Euler homogeneous at all $\fx \in X$ and so $\widetilde{\scJ_{f,\fx}} = \scJ_{f,\fx}$ locally everywhere. For $i) \implies ii)$, Proposition 1.5 \cite{Nar15} shows linear Jacobian type implies strongly Euler homogeneous, from whence $\widetilde{\scJ_f} = \scJ_f$.
\end{proof}

\subsection{Liouville ideal}

We begin to relate the linear Jacobian type condition to $\ann_{\scD_X[s]} f^{s-1}$ having the simplest possible generating set.

\begin{defn} \label{def-logDer}
    Fix $f \in \scO_X$. The \emph{logarithmic derivations} $\Der_X(-\log f)$ is the coherent $\scO_X$-module of $\bC$-derivations tangent to our fixed $f$:
    \begin{equation*}
        \Der_X(-\log f) = \{\delta \in \Der_{\bC}(\scO_X) \mid \delta \cdot f \in \scO_X \cdot f \}.
    \end{equation*}
    The coherent $\scO_X$-submodule of \emph{annihilating logarithmic derivations} $\Der_X(-\log_0 f)$ are the derivations that kill $f$:
    \begin{equation*}
        \Der_X(-\log_0 f) = \{\delta \in \Der_{\bC}(\scO_X) \mid \delta \cdot f = 0 \}. 
    \end{equation*}
\end{defn}

\begin{rem} \label{rmk-logDerivations} \enspace 
    \begin{enumerate}
        \item The logarithmic derivations $\Der_X(-\log f)$ are well-behaved with respect to coordinate change and choice of a defining equation for $f$. So one can define $\Der_X(-\log D)$ for $D$ a divisor similarly.
        \item \label{item-annLogCoordChanges} The annihilating logarithmic derivations do depend on choice of defining equations (cf. Remark 2.10.(4) \cite{Wal17}). But they are well-behaved with respect to coordinate change, via multiplication with the transpose of the Jacobian matrix. Alternatively, this well-behavedness follows from the $\scO_X$-isomorphism
        \begin{equation*}
            \Der_X(-\log_0 f) \simeq \ann_{\scD_X} f^{s} \cap F_1^{\ord} \scD_X  = \ann_{\scD_X} f^s \cap \Der_{\bC}(\scO_X).
        \end{equation*}
        \item Fix $f \in \scO_X$. The annihilating logarithmic derivations $\Der_X(-\log_0 f)$ are the $\scO_X$-syzygies of the Jacobian $\widetilde{\scJ_f}$; the logarithmic derivations $\Der_X(-\log f)$ are the $\scO_X$-syzygies of the generating set $\{f, \frac{\partial f}{\partial x_1}, \ldots, \frac{\partial f}{\partial x_n}\}$ of the true Jacobian $\scJ_f$.
    \end{enumerate}
\end{rem}

Remark \ref{rmk-logDerivations}.\eqref{item-annLogCoordChanges} shows that $\scD_X\cdot \Der_X(-\log_0 f) \subseteq \ann_{\scD_X}f^{s-1}$. We will study when the annihilating logarithmic derivations generate the entire $\scD_X$-annihilator, by investigating the ideal generated by the principal symbols of $\Der_X(-\log_0f)$.

\begin{defn}
Fix $f \in \scO_X$. The \emph{Liouville ideal} $\scL_f$ of $f$ is the $(\gr^{\ord} \scD_X)$-ideal
\begin{equation*}
    \scL_f := \gr^{\ord} \scD_X\cdot \gr^{\ord}(\Der_X(-\log_0f)) \subseteq \gr^{\ord} \scD_X.
\end{equation*}

\end{defn}

\begin{rem} \label{rmk-Liouville} \enspace 

\begin{enumerate}
\item Given any biholomorphism $\varphi:U\to V$ between open subsets $U$ and $V$ of $\bC^n$,
\[\varphi(\scL_f)=\scL_{f\circ \varphi^{-1}}.\]
i.e., the definition of the Liouville ideal commutes with coordinate change (see \cite{Wal17}, Remark 3.2.(1)), implying that the Liouville ideal $\scL_f \subseteq \text{gr}^{\ord} \scD_X$ is indeed well-defined in general.
\item If rather $\fx\in X$ and $f\in\scO_{X,\hspace{1pt}\fx}$, then we choose an extension of $f$ to an open neighborhood of $\fx$ and define the Liouville ideal $\scL_f\subseteq\text{gr}^{\ord} \scD_{X,\hspace{1pt}\fx}$ of $f$ as the ideal of germs of elements in the Liouville ideal of this extension. Alternatively, this is of course just the ideal generated by the symbols of derivations of $\scO_{X,\hspace{1pt}\fx}$ that kill $f$.
\item When $f \in \scO_{X,\hspace{1pt}\fx}$ is strongly Euler homogeneous, Remark 3.2 \cite{Wal17} shows local and algebraic properties of the Liouville ideal $\scL_f$ are independent of the choice of defining equation of the divisor germ. In particular, one can define the Liouville ideal sensibly for reduced, effective, strongly Euler homogeneous divisors.
\end{enumerate}
\end{rem}

We would like to thank Luis Narváez Macarro for pointing out the equivalence of the notions stated in Proposition \ref{propLJ0-new}, and we also note a variant of $ii) \implies i)$ appears in Corollary 3.23 \cite{Wal17}. Note also that $iii)$ is equivalent to the same facts about $\gr^{\ord} \ann_{\scD_X} f^{s-1}$ by a formal substitution. 

\begin{prop} \label{propLJ0-new}

Let $f \in \scO_X$. Then the following are equivalent:

\begin{enumerate}[label=\roman*)]
    \item The ideal $\widetilde{\scJ}_f\subseteq\scO_X$ is of linear type.
    \item The Liouville ideal $\scL_f$ is prime of height $n-1$.
    \item $\scL_f$ is prime and any generating set of $\scL_f$ is a Gr{\"o}bner basis for $\gr^{\ord} \ann_{\scD_X} f^s$, i.e.
    \begin{equation*}
        \scL_f=\gr^{\ord}(\scD_X\cdot \Der_X(-\log_0f))=\gr^{\ord}(\ann_{\scD_X}f^s).
    \end{equation*}
\end{enumerate}
So if $f$ is of linear Jacobian type then all three conditions hold. 
\end{prop}

\begin{proof}
Once we prove the equivalence, the ``so if'' claim follows from Proposition \ref{propLJ-new}.

$i)\Leftrightarrow ii)$: By definition, $\widetilde{\scJ}_f$ is of linear type as an ideal of $\scO_X$ if and only if the surjective $\scO_X$-algebra homomorphism
\[\text{Sym}_{\scO_X}(\widetilde{\scJ}_f)\to \text{Rees}(\widetilde{\scJ}_f)\]
is an isomorphism. This can be checked in local coordinates and is the case if and only if the surjective $\scO_X$-algebra homomorphism
\[\varphi:\text{gr}^F\scD_X=\scO_X[\xi_1,\ldots,\xi_n]\to\text{Rees}(\widetilde{\scJ}_f)\; ;\;\xi_i \mapsto \frac{\partial f}{\partial x_i}\]
has kernel $\scL_f$, since
\[\text{Sym}_{\scO_X}(\widetilde{\scJ}_f)\simeq \scO_X[\xi_1,\ldots,\xi_n]/(\text{Syz}(\widetilde{\scJ}_f)) = \scO_X[\xi_1,\ldots,\xi_n]/\scL_f .\]
Here by $(\text{Syz}(\widetilde{\scJ}_f))$ we mean the ideal generated by $\{s \cdot (\xi_1 \ldots \xi_n)^T \; | \; s \in \text{Syz}(\widetilde{\scJ}_f)\subseteq \scO_X^n\}$, which coincides with $\scL_f$ by definition.

Since $\text{Rees}(\widetilde{\scJ}_f)$ is an integral domain of dimension $n+1$, the kernel of $\varphi$ is \emph{a priori} a prime ideal of $\text{gr}^F\scD_X$ of height $n-1$, containing the Liouville ideal $\scL_f$. Then it is clear that the Liouville ideal is prime of height $n-1$ if and only if it equals the kernel of $\varphi$, showing the equivalence required.

$ii)\Leftrightarrow iii)$: It is clear that \emph{a priori} we have that 
\begin{equation*}
    \scL_f \subseteq \gr^{\ord}(\scD_X\cdot \Der_X(-\log_0f)) \subseteq \gr^{\ord} (\ann_{\scD_X}f^s).
\end{equation*}
Therefore it suffices to prove that $\gr^{\ord} (\ann_{\scD_X}f^s)$ has height at most $n-1$, as then under the assumption that $\scL_f$ is prime, it has height $n-1$ if and only if it equals $\gr^{\ord} (\ann_{\scD_X}f^s)$.

If this wasn't the case, $\gr^{\ord} \ann_{\scD_X}f^s$ would be of height at least $n$. But, over a smooth point of the zero locus of $f$ in $X$, $\gr^{\ord} \ann_{\scD_X}f^s$ is a prime ideal with height $n-1$. Indeed, local algebraic properties of the characteristic ideal of $\scD_X\cdot f^s$ do not depend on choice of defining equation nor coordinate system, so locally at a smooth point we may write $f=x_1$ after a change of coordinates, making the characteristic ideal exactly $(\xi_2,\ldots,\xi_n)$. As any open subset intersecting $\{ f= 0 \}$ contains smooth points, this causes a contradiction.
\end{proof}

Computing $\ann_{\scD_X} f^s$ is trivial in this case because:

\begin{prop} \label{propcondnPres-new}
Assume that $f \in \scO_X$ satisfies the equivalent conditions of Proposition \ref{propLJ0-new}. Then
\begin{equation*}
    \ann_{\scD_X} f^{s} = \scD_X\cdot \Der_X(-\log_0 f).
\end{equation*}
If $f$ is additionally assumed to have a global Euler homogeneity $E$, then
\begin{equation*}
    \ann_{\scD_X[s]} f^{s} = \scD_X[s]\cdot \Der_X(-\log_0 f) + \scD_X[s] \cdot (E - s).
\end{equation*}
\end{prop}

\begin{proof}
    Argue as in Lemma 3.25 (set $\scP = \emptyset)$ and Theorem 3.26 of \cite{Wal17}.
\end{proof}

\subsection{Use Cases}

Now we establish some ``practical'' algebraic and geometric conditions certifying the property of linear Jacobian type. This gives a large class of ``nice'' divisors satisfying all but the Bernstein-Sato part of Theorem \ref{thmgamma2-new}'s Hypotheses \ref{hyp-mainHypotheses}. We also give a natural class of divisors in $\mathbb{C}^{3}$ satisfying all the Hypotheses \ref{hyp-mainHypotheses}. 

First of all, it is now immediate that $f \in \scO_X$ of linear Jacobian type condition already gives a large class of use cases of our main results:

\begin{cor} \label{cor-LJTsatisfiesHyp}
    Suppose that $f \in \scO_X$ is of linear Jacobian type. If the roots of the Bernstein-Sato polynomial of $f$ are contained in $(-2,0)$ locally everywhere, then Theorem \ref{thmmain-new} and Theorem \ref{thmgamma2-new} apply to $F_\bullet^H \scO_X(*f)$ and $F_\bullet^H \scO_{X,\hspace{1pt}\fx}(*f)$. 
\end{cor}

\begin{proof}
    Hypotheses \ref{hyp-mainHypotheses} hold by \ref{propLJ0-new}, Proposition \ref{propLJ-new}, and Proposition \ref{prop-ParaPrimeEquivalent}.
\end{proof}

Turning to easily computable assumptions guaranteeing linear Jacobian type, we first require the \emph{logarithmic differential forms}. These can be thought of as the objects in the smallest subcomplex of the de Rham complex $\Omega_X^k(*D)$ such that: every module member has pole order at most one; the subcomplex is preserved under exterior differentiation. 

\begin{defn} \label{def-logforms}
    Let $D$ be a reduced effective divisor on $X$ (or replace $D$ with $f \in \scO_X$). The sheaf of logarithmic differential $k$-forms along $D$ is the $\scO_X$-module
    \begin{equation*}
        \Omega_X^k(\log D) = \{ \omega \in \Omega_X^k(D) \mid d(\omega) \in \Omega_X^{k+1}(D) \}.
    \end{equation*}
\end{defn}

Recall that $D$ (or $f \in \scO_X$) is \emph{free} when $\Der_X(-\log D)$ is a free $\scO_X$-module. This is equivalent to $\Omega_X^1(\log D)$ being a free $\scO_X$-module, which in turns implies all $\Omega_X^k(\log D)$ are free. The following is a significant relaxation of freeness:

\begin{defn} \label{def-tame}

$X$ a complex manifold, $\fx\in X$, $f\in\scO_{X,\hspace{1pt}\fx}$. Then $f$ is \emph{tame} if
\[\text{prdim}_{\scO_{X,\hspace{1pt}\fx}}(\Omega_{X,\hspace{1pt}\fx}^k(\log f))\leq k \;\;\;\forall k \geq 0.\]
If $f\in\scO_X$, then $f$ is \emph{tame} if the germ $f\in\scO_{X,\hspace{1pt}\fx}$ is tame for every $\fx\in X$. If $D$ is a reduced effective divisor, then $D$ is \emph{tame} if it admits tame defining equations locally around every point $\fx\in X$.  
\end{defn}

Now we need a geometric condition coming from logarithmic derivations. $D$ is a reduced effective divisor as always. We define an equivalence relation on points of $X$, denoted $\approx_X$. For $\fx, \fy\in X$, $\fx\approx_D\fy$ if and only if there exists an open subset $U$ of $X$ containing both $\fx$ and $\fy$, and a logarithmic derivation $\delta \in \Der_U(-\log D \cap U)$ on $U$, such that $\delta$ is nowhere-vanishing on $U$, and such that there exists an integral curve of $\delta$ passing through $\fx$ and $\fy$. 

\begin{defn}[Saito, \cite{KSai80}]
    The \emph{logarithmic stratification of $X$ with respect to $D$} is the stratification of $X$ whose strata are precisely the irreducible components of the equivalence classes of $\approx_Y$. We say that $D$ is \emph{Saito holonomic} if the logarithmic stratification of $X$ induced by $D$ is locally finite, i.e. if for any point $\fx\in X$, there exists an open neighborhood $U$ of $\fx$ such that $U$ has a non-empty intersection with only finitely many logarithmic strata.
\end{defn}

\begin{rem}
    An equivalent definition is that 
    \begin{equation*}
        \text{dim}_{\bC}\{\fx \in X\, :\, \text{dim}_{\bC}(\Der_X(-\log D)\otimes\scO_{X,\hspace{1pt}\fx}/\fm_{X,\hspace{1pt}\fx})=i\}\leq i
    \end{equation*} 
    for all $i\geq 0$.
\end{rem}

\begin{rem} \enspace

\begin{enumerate}  

\item If $\dim X = 3$, then any reduced effective divisor $D \subseteq X$ is tame, thanks to the reflexivity of logarithmic differential forms.

\item The zero set of a hyperplane arrangement $\scA \subseteq \bC^n$ is Saito holonomic. The stratification is given by flats of the intersection lattice.

\item A free divisor is Saito holonomic if and only if it is Koszul (\cite{GMNS09}, Theorem 7.4). So strongly Koszul free divisors are a subset of tame, Saito holonomic, and strongly Euler homogeneous divisors.

\end{enumerate}

\end{rem}

At last, we arrive at the natural and easily verifiable conditions on divisors guaranteeing the assumptions necessary in our main theorems. The nontrivial part of the below is due to Walther \cite{Wal17}; the trivial part is Corollary \ref{cor-LJTsatisfiesHyp}.

\begin{thm}[Walther, \cite{Wal17}] \label{thmOfUli}
Let $f \in \scO_X$ for $X$ or $D$ a reduced effective divisor. If $f$ (resp. $D$) is tame, Saito holonomic, and strongly Euler homogeneous, then $f$ is of linear Jacobian type. In particular, if in addition the roots of the Bernstein-Sato polynomial of $f$ (resp. $D$) are contained in $(-2,0)$ locally everywhere, then Theorem \ref{thmmain-new} and Theorem \ref{thmgamma2-new} apply.

\end{thm}

To use Theorem \ref{thmOfUli}, one must still verify the Bernstein-Sato polynomial's roots are (locally) in $(-2,0)$. The result at the end of this section gives a class where this can be checked easily using local cohomology. 

\begin{defn} \label{defn-PWHomLe}
    We say $f \in \scO_X$ or  a reduced effective divisor $D \subseteq X$ is \emph{positively weighted homogeneous at $\fx$} if there exists local coordinates $x_1, \ldots, x_n$ at $\fx$, a choice of local defining equation $h \in \bC[x_1, \dots, x_n]$, and positive weights $w_1, \ldots, w_n\in\bR_{>0}$ such that
    \begin{equation*}
        \left( \sum_{1 \leq i \leq n} w_i x_i \partial_i \right) \cdot h = (\deg_{\textbf{w}} h)h
    \end{equation*}
    for some $\deg_{\textbf{w}}h\in\bR$. This is equivalent to requiring that $h$ is ``homogeneous" with respect to the non-standard grading on $\bC[x_1,\ldots,x_n]$ given by assigning $x_i$ to have degree $w_i$, and then $\deg_{\textbf{w}}h$ equals the degree of $h$ with respect to this grading. When $f$ is assumed to be in $\bC[x_1, \ldots, x_n]$, our convention is that it is positively weighted homogeneous when it is itself homogeneous with respect to some valid weight system.

    We say $f \in \scO_X$ or a reduced effective divisor $D \subseteq X$ is \emph{positively weighted homogeneous locally everywhere} if for all $\fx \in \text{Var}(f)$ (or $\fx \in D$), $f$ (or $D$) is positively weighted homogeneous at $\fx$. 
\end{defn}

\begin{rem} \enspace 
    \begin{enumerate}
        \item Any divisor is positively weighted homogeneous at a smooth point. 
        \item Positively weighted homogeneous at $\fx$ implies strong Euler homogeneity there.
        \item If $Z \subseteq \bP^{n-1}$ is positively weighted homogeneous locally everywhere, so is the attached cone $C(Z) \subseteq \bC^n.$
        \item Any positively weighted homogeneous locally everywhere divisor is Saito holonomic.
    \end{enumerate}
    
\end{rem}

When $f \in R = \bC[x_1, \ldots, x_n]$ is positively weighted homogeneous with respect to the weight system $w_1, \ldots, w_n$, the affine Jacobian ideal $(\partial f) = (\partial_1 \cdot f, \ldots, \partial_n \cdot f) \subseteq R$ is graded. So if $\fm \subseteq R$ is the irrelevant ideal, the zeroeth local cohomology module 
\begin{equation*}
    H_{\fm}^{0} R / (\partial f) = \{ \xi \in R / \partial f \mid \fm^j \xi = 0 \text{ for } j \gg 0 \}
\end{equation*}
is also graded. Using this graded data, we have the following class of $\bC^3$ divisors satisfying all of Hypotheses \ref{hyp-mainHypotheses} and hence Theorem \ref{thmmain-new} and Theorem \ref{thmgamma2-new}. (In what follows we use a global Bernstein-Sato polynomial, since we are working with polynomials.)

\begin{thm}[Corollary 7.10 \cite{Bath24}] \label{thm-PWHomLENiceClass}
    Suppose that $f \in R = \bC[x_1, x_2, x_3]$ is positively weighted homogeneous with respect to the weight system $w_1, w_2, w_3$ and also positively weighted homogeneous locally everywhere. Then $f \in \scO_{\bC^3}$ is of linear Jacobian type. Moreover, $Z(b_f(s)) \subseteq (-2,0)$ if and only if 
    \begin{equation*}
        [H_{\fm}^0 R / (\partial f)]_t = 0 \text{ for all } t \in [2 \deg_{\mathbf{w}}(f) - (w_1 + w_2 + w_3), 3 \deg_{\mathbf{w}}(f) - (w_1 + w_2 + w_3))
    \end{equation*}
    In particular, $f$ and $\text{\normalfont div}(f)$ satisfy Hypotheses \ref{hyp-mainHypotheses} and so Theorem \ref{thmmain-new} and Theorem \ref{thmgamma2-new} apply.
\end{thm}

\section{Calculations} \label{calns}



\noindent In this section we calculate the Hodge filtration for some examples of divisors/germs satisfying the hypotheses of the main theorems of this paper. The computational complexity is concentrated in two areas: calculating the Bernstein-Sato polynomial of $f \in \scO_X$; determining a Gröbner basis for the ideal $\Gamma_f$.

To improve the presentation/computation of Hodge filtrations we now introduce \emph{Hodge ideals}. The following is due to Saito (and can be recovered from Theorem \ref{thmMP}):

\begin{thm}[Saito, \cite{MSai93}] 

Let $X$ be a complex manifold and let $D$ be a reduced effective divisor on $X$. Then
\[F_k^H\scO_X(*D)\subseteq \scO_X((k+1)D) \;\;\; \text{ for all } \;\; k \geq 0.\]
    
\end{thm}

\begin{defn}[Mustaţă-Popa, \cite{MP19a}]

Let $X$ a complex manifold and let $D$ be a reduced effective divisor on $X$. Then the \emph{$k$-th Hodge ideal} associated to $D$ is the ideal sheaf on $X$ defined by
\[F_k^H\scO_X(*D)=I_k(D)\otimes_{\scO_X}\scO_X((k+1)D).\]
We define the Hodge ideals $I_k(f)$ associated to a reduced non-constant non-invertible global section $f\in\scO_X$ identically.

\end{defn}

\begin{thm}[Mustaţă-Popa, \cite{MP19a}]

The $0$-th Hodge ideal associated to a reduced effective divisor $D$ is equal to the multiplier ideal $\scJ(D,1-\epsilon)$, $0<\epsilon\ll 1$.
    
\end{thm}

A measure of the complexity of the Hodge filtration is its generating level, which tells us the first point at which the filtration determines all higher steps. 

\begin{defn}

Let $X$ be a complex manifold and let $D$ be a reduced effective divisor on $X$. The \emph{generating level} $\genLevel(\scO_X(*D))$ of the Hodge filtration $F_{\bullet}^H\scO_X(*D)$ is 
\begin{equation*}
    \genLevel(\scO_X(*D)) = \min \{ k \in \mathbb{Z}_{\geq 0} \mid  F_{k'+k}^H\scO_X(*D)=F_{k'}\scD_X\cdot F_k^H\scO_X(*D) \;\;\; \forall k ' \in \bZ_{\geq 0} \}.
\end{equation*}
We define the generating level $\genLevel(\scO_X(*f))$ for $f \in \scO_X$ analogously. 
    
\end{defn}

\begin{thm}[Mustaţă-Popa, \cite{MP19a}] For $D \subseteq X$ a reduced effective divisor
\begin{equation*}
    \genLevel (\scO_X(*D)) \leq \dim X - 2.
\end{equation*}

\end{thm}

\subsection{Free divisors}

Recall that \cite{CNS22} considers strongly Koszul free divisors, in which case our main Hypotheses \ref{hyp-mainHypotheses} are known to hold. Computations can be found in \cite{CNS22}. Note that being free is equivalent to the true Jacobian ideal being Cohen--Macaulay of codimension two. Tameness is much weaker: for example, all divisors in $\bC^3$ are tame, irrespective of Cohen--Macaulay-ness.

The following clarifies the relationship between (one of) our use cases of linear Jacobian type and the contents of \cite{CNS22}.

\begin{thm}

\cite{Nar15} Let $X$ be a complex manifold, $D$ a reduced effective divisor on $X$. Assume that $D$ is a free divisor, i.e. that $\text{\emph{Der}}_X(-\log D)$ is locally free as an $\scO_X$-module. Then the following conditions are equivalent:

\begin{enumerate}[label=\roman*)]

\item $D$ is of linear Jacobian type.

\item $D$ is strongly Koszul.

\item $D$ is weakly Koszul and $\text{\emph{ann}}_{\scD_X}f^s$ is generated by differential operators of degree $1$.

\item $D$ is Koszul and strongly Euler homogeneous.

\item $D$ is Koszul and satisfies the log comparison theorem.
    
\end{enumerate}

\label{thmfree}
    
\end{thm}

\subsection{Hyperplane arrangements} As already mentioned, we have the following:

\begin{prop}[Example 3.14 \cite{KSai80} and \cite{MSai16}]

$X=\bC^n$ and $f\in\scO_X$ a reduced defining equation for a central hyperplane arrangement. Then $f$ is strongly Euler homogeneous and Saito holonomic, and $Z(b_f(s)) \subseteq (-2,0)$.
    
\end{prop}

\begin{cor}

$X=\bC^n$ and $f\in\scO_X$ a reduced defining equation for a tame central hyperplane arrangement. Then $f$ is of linear Jacobian type and $Z(b_f(s)) \subseteq(-2,0)$. In particular, Theorem \ref{thmmain-new} and Theorem \ref{thmgamma2-new} apply.
\end{cor}

Since any divisor $D \subseteq X$ is tame provided that $\dim X \leq 3$, any central hyperplane arrangement in $\bC^3$ fits into the above Corollary. Then recall that $\ann_{\scD_X} f^{s-1}$ is generated by $\Der_X(-\log_0 f)$, reducing the complexity of computation. Here are two examples in $\bC^3$:

\begin{egs} In the following examples, note that our formulas for $I_0(f)$ confirm the general formula for hyperplane arrangements found in \cite{Mustata06} (see also \cite{Teitler08}). The point here is that we can also compute generation level.

\begin{enumerate}
    
\item $f=xyz(x+y+z) \in \bC[x,y,z]$. This is a generic arrangement of four hyperplanes. The sheaf of logarithmic vector fields has $4$ generators in this case and is not free. Using the results of this paper we calculate that 
\[I_0(f)=(x,y,z) \quad \text{and} \quad \genLevel (\scO_{\bC^3}(*f)) = 0.\]

\item $f=xyz(x+y+z)(x+y+2z) \in \bC[x,y,z]$. The associated hyperplane arrangement again isn't free but this time is of course not generic. Nevertheless, the generating level turns out again to be zero, with 
\[I_0(f)=(z^2,yz,xz,xy+y^2,x^2-y^2) \quad \text{and} \quad \genLevel (\scO_{\bC^3}(*f)) = 0.\]

\end{enumerate}

\end{egs}

Even in cases where linear Jacobian type fails, Hypotheses \ref{hyp-mainHypotheses} may hold:

\begin{eg} Let 
\begin{equation*}
    f=xyz(w+x)(w+y)(w+z)(w+x+y)(w+x+z)(w+y+z) \in \bC[x,y,z,w].
\end{equation*} 
and $\scA \subseteq \bC^4$ the attached divisor. This is one of the few non-tame hyperplane arrangements known. And: $\scD_{\bC^4} \cdot \Der_{\bC^4}(-\log_0 f) \subsetneq \ann_{\scD_{\bC^4}} f^{s-1} $; any generating set of the $\scD_{\bC^4}[s]$-annihilator of $f^{s-1}$ contains an operator of total order at least two and $\text{div}(f)$ is not of linear Jacobian type. Nevertheless $\gr^{\ord} (\ann_{\scD_{\bC^4}} f^{s-1})$ is prime (see Example 5.7-Question 5.9 \cite{Wal17}). By Theorem \ref{thmmain-new} we conclude that, for all $k \in \mathbb{Z}_{\geq 0}$,
\begin{equation*}
    F_{k}^H \scO_{\bC^4}(*f) \subseteq F_k^{\ord} \scO_{\bC^4}(*f).
\end{equation*}
\end{eg}


\subsection{Positively weighted homogeneous locally everywhere divisors in $\bC^3$}

By Theorem \ref{thm-PWHomLENiceClass}, Theorem \ref{thmmain-new} and Theorem \ref{thmgamma2-new} apply to any positively weighted homogeneous locally everywhere divisor given by a positively weighted homogeneous polynomial $f \in \bC[x_1, x_2, x_3]$, provided the appropriate vanishing of graded pieces of $H_{\fm}^{0} R / \partial f$ holds. Moreover, $\scD_{\bC^3} \cdot \Der_{\bC^3}(-\log_0 f) = \ann_{\scD_{\bC^3}} f^{s-1}$ helping with Gr{\"o}bner computations.

\begin{egs}

\; \begin{enumerate}

\item $f=x^5+y^4z \in \bC[x,y,z]$. The singular locus is the $z$-axis. Working over $\bP^2$, the local equation at $[0:0:1]$ is $x^5 + y^4$, which is positively weighted homogeneous with an isolated singularity. So $\text{div}(f)$ is positively weighted homogeneous locally everywhere. One checks the appropriate graded pieces of $H_{\fm}^0 R / \partial f$ vanish and so $\text{div}(f)$ falls under Theorem \ref{thm-PWHomLENiceClass}'s purview. In particular, $\text{div}(f)$ is of linear Jacobian type.

Saito has computed the Bernstein-Sato polynomial of $f$ in Example I \cite{MSai20}. So
\begin{equation*}
    \beta_f(s) = \left(s+\frac{1}{20}\right)\left(s+\frac{2}{20}\right)\left(s+\frac{3}{20}\right)\left(s+\frac{6}{20}\right)\left(s+\frac{7}{20}\right)\left(s+\frac{11}{20}\right).
\end{equation*}
The sheaf of logarithmic vector fields has 4 generators, so $\text{div}(f)$ is not free. We obtain that 
\begin{equation*}
    I_0(f)=(x,y)^3 \quad \text{and} \quad \genLevel \scO_{\bC^3}(*f) = 0.
\end{equation*}

\item $f=x^5+x^2y^3+y^4z \in \bC[x,y,z]$. The singular locus is the $z$-axis. The singular locus of the associated projective hypersurface is the point $[0:0:1]$, and the divisor is defined by the equation $h=x^5+x^2y^3+y^4$ at this point. A simple calculation gives that 
\[(90xy+400)h=(18x^2y+9y^2+80x)\partial_x h +(18xy^2-15x^2+100y)\partial_y h,\]
so that $h$ is strongly Euler homogeneous at the origin. Since $h$ has an isolated singularity at the origin, this implies in fact that $h$ is positively weighted homogeneous at the origin (see \cite{Ksai71}), and so $\text{div}(f)$ is positively weighted homogeneous locally everywhere. One checks the appropriate graded pieces of $H_{\fm}^0 R / \partial f$ vanish and so Theorem \ref{thm-PWHomLENiceClass} applies to $\text{div}(f)$. In particular, $\text{div}(f)$ is of linear Jacobian type. 

Saito (Example II \cite{MSai20}) calculated $b_f(s)$. From this we read off:
\[\beta_f(s)=\left(s+\frac{1}{20}\right)\left(s+\frac{2}{20}\right)\left(s+\frac{3}{20}\right)\left(s+\frac{6}{20}\right)\left(s+\frac{7}{20}\right)\left(s+\frac{11}{20}\right)\left(s+\frac{1}{5}\right)\left(s+\frac{2}{5}\right).\]
Again the sheaf of log vector fields has 4 generators and $\text{div}(f)$ is not free. Here:
\begin{align*}
    &I_0(f) = (x^3,x^2y,xy^2,y^3) = (x,y)^3; \\
    &I_1(f) =(x^7,x^6y,x^5y^2,xy^6,y^7,y^6z,xy^5z,x^2y^3) = (x,y)^7 \cap (z,x^2,y^7,xy^6); \\
    &\genLevel(\scO_{\bC^3}(*f)) = 1.
\end{align*}
The associated primes of $I_1(f)$ are $\{(x,y), (x,y,z)\}$. Its primary decomposition is given above.
\end{enumerate}

\end{egs}

\subsection{Non-linear Jacobian type divisors}

Here we give some examples of divisors that are not of linear Jacobian type, but nevertheless satisfy Hypotheses \ref{hyp-mainHypotheses} and so succumb to our main theorems. These examples are also not Saito-holonomic. For presentation's sake, set
\begin{equation*}
    \ann_{\scD_X}^{(k)} f^{s-1} = \scD_X\cdot (F_k^{\ord} \scD_X \cap \ann_{\scD_X} f^{s-1}) \subseteq \ann_{\scD_{X}} f^{s-1}.
\end{equation*}
Note that $\scD_X \cdot \Der_X(-\log_0 f) = \ann_{\scD_X}^{(1)} f^{s-1}.$ In all examples in this subsection we will have that
\begin{equation*}
    \ann_{\scD_X}^{(1)} f^{s-1} \neq \ann_{\scD_X} f^{s-1}.
\end{equation*}

\begin{egs} \enspace
\begin{enumerate} 

    \item $f = xy(x+y)(x+yz) \in \bC[x,y,z]$. This is the notorious ``four lines'' example introduced in Remark 4.2.4 \cite{CalderonMoreno99}. Here $\text{div}(f)$ is free, strongly Euler homogeneous, but not Saito holonomic: every point on the $z$-axis is its own logarithmic stratum. So $\text{div}(f)$ does not fit into the paradigm of \cite{CNS22}. One checks that
    \begin{equation*}
        \ann_{\scD_{\bC^3}}^{(1)} f^{s-1} \subsetneq \ann_{\scD_{\bC^3}}^{(2)} f^{s-1} = \ann_{\scD_{\bC^3}} f^{s-1}.
    \end{equation*}
    Nevertheless, Macaulay2 confirms that $\gr^{\ord} (\ann_{\scD_X}f^{s-1})$ is prime. Macaulay2 also gives:
    \begin{equation*}
        b_f(s) = (s+1)^3 (s+1/2)(s+3/4)(s+5/4).
    \end{equation*}
    (And $b_f(s) = b_{f,0}(s)$ despite $f$ not being homogeneous). So Hypotheses \ref{hyp-mainHypotheses} are satisfied and the main theorems hold. We compute that:
    \begin{equation*}
        I_0(f) = (x,y)^2 \quad \text{and} \quad \genLevel(\scO_{\bC^3}(*f)) = 0.
    \end{equation*}

    \item $f = wx(w+x)(w+xyz) \in \bC[x,y,z,w].$ Here $\text{div}(f)$ is not Saito holonomic, as $\{w=x=0\}$ is an irreducible component of $\{\fx\in \bC^4\, :\, \text{dim}_{\bC}(\text{Der}_{\bC^4} (-\log f) \otimes \scO_{\bC^4,\,\fx}/\fm_{\bC^4,\,\fx})=1\}$, since $E=\frac{1}{4}(w\partial_w+x\partial_x)\in\text{Der}_{\bC^4}(-\log f)$ (each line in the $w$-$x$ plane parallel to $w=x$ is its own logarithmic stratum). Here $\text{div}(f)$ is tame, but neither free nor of linear Jacobian type: Macaulay2 says that
    \begin{equation*}
    \ann_{\scD_{\bC^4}}^{(1)} f^{s-1} \subsetneq \ann_{\scD_{\bC^4}}^{(2)} f^{s-1} = \ann_{\scD_{\bC^4}} f^{s-1}.
    \end{equation*}

    However Macaulay2 says that $\gr^{\ord} (\ann_{\scD_{\bC^4}} f^{s-1})$ is prime and 
    \begin{equation*}
        b_f(s)=\left(s+1\right)^4\left(s+\frac{1}{2}\right)\left(s+\frac{3}{4}\right)\left(s+\frac{5}{4}\right)\left(s+\frac{3}{2}\right)
    \end{equation*}
    so $f$ and $\text{div}(f)$ satisfy Hypotheses \ref{hyp-mainHypotheses}. (And $b_f(s) = b_{f,0}(s)$ despite $f$ not being homogeneous.) We compute that 
    \begin{equation*}
        I_0(f) = (x,w)^2 \quad \text{and} \quad \genLevel \scO_{\bC^4}(*f) = 0.
    \end{equation*}

\end{enumerate}

If one replaces $y$ with $w$ in $(1)$ one observes the hypersurface in $(1)$ is obtained by restricting the hypersurface in $(2)$ to the smooth hyperplane $\{y = 1 \}$. The Hodge ideals of $(1)$ are obtained from those in $(2)$ by the same restriction. Compare to Theorem 16.1 \cite{MP19a}.
\end{egs}

\addtocontents{toc}{\protect\setcounter{tocdepth}{2}}


\bibliographystyle{siam}
\bibliography{bibliography}

\end{document}